\documentclass[11pt]{article}

\usepackage[]{changes}

\usepackage{amssymb}
\usepackage{amsmath}
\usepackage[utf8]{inputenc}

\usepackage{graphicx}
\usepackage{psfrag}
\usepackage{amsmath,todonotes,amsthm}
\usepackage{latexsym}
\usepackage{amssymb}
\usepackage{graphicx}

\newcommand{\R}{\mathbb{R}}
\newcommand{\Z}{\mathbb{Z}}

\newcommand{\B}{\mathbb{B}}

\newtheorem{thm}{Theorem}[section]

\newtheorem{lem}[thm]{Lemma}

\numberwithin{equation}{section}
\def \x{\times}

\def \de{\partial}
\def \e{\epsilon}
\def \m{\mathbb}

\def \<{\langle}
\def \>{\rangle}

\def\and{\quad\text{ and }\quad}

\begin{document}
%\listofchanges
%\begin{frontmatter}

%% Title, authors and addresses

%% use the tnoteref command within \title for footnotes;
%% use the tnotetext command for the associated footnote;
%% use the fnref command within \author or \address for footnotes;
%% use the fntext command for the associated footnote;
%% use the corref command within \author for corresponding author footnotes;
%% use the cortext command for the associated footnote;
%% use the ead command for the email address,
%% and the form \ead[url] for the home page:
%%
%% \title{Title\tnoteref{label1}}
%% \tnotetext[label1]{}
%% \author{Name\corref{cor1}\fnref{label2}}
%% \ead{email address}
%% \ead[url]{home page}
%% \fntext[label2]{}
%% \cortext[cor1]{}
%% \address{Address\fnref{label3}}
%% \fntext[label3]{}

\title{Optimal Control Problems with Time Delays (Preliminary Version)}

%% use optional labels to link authors explicitly to addresses:
%% \author[label1,label2]{<author name>}
%% \address[label1]{<address>}
%% \address[label2]{<address>}

\author{A. Boccia\thanks{ Dept.~of Mechanical Engineering, Massachusetts Institute of Technology, 77 Massachusetts Avenue, Cambridge, MA, USA.(e-mail: aboccia@mit.edu).}  \hspace{0.01 in} and R. B. Vinter\thanks{EEE Dept., Imperial College London,
Exhibition Road, London SW7 2BT, UK (e-mail: r.vinter@imperial.ac.uk).
%{\tt a.boccia@imperial.ac.uk}
}}
\maketitle
%%
%\address{\thanks{Department of Electrical and Electronic Engineering, Imperial College London,
%Exhibition Road, London SW7 2BT, UK.
%{\tt a.boccia@imperial.ac.uk} }

\begin{abstract}
\noindent
This paper provides necessary conditions of optimality for optimal control problems with time delays in both state and control variables. Different versions of the necessary conditions cover fixed end-time problems and, under additional hypotheses, free end-time problems. The conditions improve on previous available conditions in a number of respects. They can be regarded as the first generalized Pontryagin Maximum Principle for fully non-smooth optimal control problems, involving delays in state and control variables, only special cases of which have previously been derived. Even when the data is smooth, the conditions advance the existing theory. For example, we provide a new `two-sided' generalized transversality condition, associated with the optimal end-time, which gives more information about the optimal end-time than the `one-sided' condition in the earlier literature.
 But there are improvements 
in other respects, relating to the treatment of initial data, specifying past histories of the state and control, and to the unrestrictive nature of the hypotheses under which the necessary conditions are derived.

%\added{Andrea's changes implemented + small changes to English}
\end{abstract}
\vspace{0.05 in}

\vspace{0.01 in} 


\section{Introduction}
This paper concerns optimal control problems, in which we seek to minimize a cost
$$
 J(x(.),u(.))\,=\,g(x(S),x(T))
\\
\,+ \,\int_{[S,T]} L ( t,\{x(t-h_{k})\}_{k=0}^{N}, \{u(t-h_{k})\}_{k=0}^{N})dt\,,
$$ 
 over control functions $u(.)$ such that $u(t)\in U(t)$, a.e., and  state trajectories $x(.)$ satisfying an end-point constraint  $(x(S),x(T))$ $\in C$ and a dynamic constraint, formulated as a controlled delay differential equation:
\begin{equation}
\label{intro_1}
  \dot x(t)=f ( t,\{x(t-h_{k})\}_{k=0}^{N}, \{u(t-h_{k})\}_{k=0}^{N}), \; \mbox{ a.e. } t \in [S,T]\,.
 \end{equation}
% and the cost function takes the form
%$$
% J(x(.),u(.))\,=\,g(x(S),x(T))
%\\
%\,+ \,\int_{[S,T]} L ( t,x(t-h_{0}),\ldots, x(t-h_{N}), u(t-h_{0}), %\ldots, u(t-h_{N}))\,.
%$$ 
Here, $[S,T]$ is a given time interval, $h_{0} < h_{1}< \ldots< h_{N}$ are given numbers such that $h_{0}=0$,  $f(.\,.):[S,T] \times \R^{(1+N)\times n}\times \R^{(1+N)\times m} \rightarrow \R^{n}$ and $L(.\,.):[S,T] \times \R^{(1+N)\times n}\times \R^{(1+N)\times m} \rightarrow \R$ are  given functions and $U(t)$, $S \leq t \leq T$, and $C$ are given sets. We write $h:=h_{N}$.  Notice that, according to this formulation, delays may occur in both $x$ and $u$ variables. 
\ \\

\noindent 
Under suitable hypotheses on the function $f(.\,.)$, we can unambiguously associate a state trajectory $x(.):[S,T]\rightarrow \R^{n}$ with a given control function $u(.):[S,T]\to \R^m$ (in some appropriate function class) and initial data in the form of (a.e.) specified values $d^{x}(s)$ and $d^{u}(s)$, $S-h\leq s <  S$,   of the $x$ variable and the $u$ variable, respectively,  on the `delay interval' $[S-h, S]$, and the initial value $x_{0}$ of the $x$ variable. The state trajectory $x(.)$ is the absolutely continuous solution to (\ref{intro_1}), consistent with the initial data, in the sense that, for each $t \in [S,T]$:
%the right side is interpreted as
%\todo[inline]{Technically the initial data $d(.)$ is defined on the delay interval $[S-h,S)$. However since it is defined almost everywhere I kept the notation with the closed interval $[S-h,S]$.}
\begin{eqnarray}
\label{state}
&&
\hspace{-0.3 in} x(t) = x_{0}+
 \int_{[S,t]} f(s, ( t,\{ x(s-h_{k} )\}, \{u(s-h_{k})\};\{d(t-h_{k})\})ds\,.
\end{eqnarray}
%
%  \mbox{ a.e. }t\in [S,T],
%\\
%\hspace{0.2 in} (x(t),u(t))= d(t) \mbox{ a.e. }t\in[S,T],
%\ldots,
%\left\{
%\begin{array}{ll}
%x(s -h _{k} )& \mbox{if } s-h_{k} \geq S
%\\
% d^{x}(s-h_{k}) & \mbox{if }s-h_{k} < S
%\end{array}
%\right\}_{k=0}^{N} , 
%\ldots,
%\left\{\begin{array}{ll}
%u(s-h _{k} )& \mbox{if } s-h_{k} \geq S
%\\
% d^{u}(s-h_{k}) & \mbox{if }s-h_{k} < S
%\end{array}
%\right\}_{k=0}^{N} 
%\ldots
%\right) \, ds\,.
%\end{eqnarray}
%(In this expression the delay index $k$ takes values $0$ through $N$.) 
Here, and throughout the paper, $\{ x(s-h_{k} )\}_{k=0}^{N}$ is written simply as $\{ x(s-h_{k} )\}$, etc.
The function $f(.\,.; \{d(s-h_{k})\})$  appearing in (\ref{state})
is 
\begin{eqnarray*}
&&f( t,x_{0},\ldots, x_{N},u_{0},\ldots, u_{N};d_{0},\ldots,d_{N})
\\
&& \qquad :=  
f
\left( t,\left\{ 
\begin{array}{ll}
x_{k} & \mbox{if }t-h_{k} \geq S
\\
 d^{x}_{k} & \mbox{if }t-h_{k} < S
\end{array}
\right\}_{k=0}^{N},
%x_{N},u_{0},\ldots, u_{N});d_{0},\ldots,d_{N}),
%\end{eqnarray*},
\left\{ \begin{array}{ll}
u_{k} & \mbox{if }t-h_{k} \geq S
\\
 d^{u}_{k} & \mbox{if }t-h_{k} < S
\end{array}
\right\}_{k=0}^{N}
\right)\,,
\end{eqnarray*}
 describes how the initial segments of the state and control variables, gathered together as a single function $d(.)=(d^{x}(.),d^{u}(.))$ on the time interval $[S-h,S]$, affect the evolution of the state trajectory $x(.)$. 
 %is defined as follows: 
% It will be convenient to use the  following notation to specify the delay differential equation and the initial data $d(s)= (d^{x}(s),d^{u}(s))$, $S-h \leq s< S$:  given $f(t, x_{0},\ldots, x_{N},u_{0},\ldots, u_{N})$ we define
%For specified initial data $d(.):[S-h,S]\rightarrow \R^{n+m}$ and control function $u(.):[S,T] \rightarrow \R^{m}$, eqn. (\ref{intro_1}) can be written, with the help of this notation, as
%\begin{eqnarray}
%\label{intro_3}
% &&  \dot x(t)=f(t, x(t-h_{0}),\ldots,x(t-h_{N}),u(t-h_{0}), \ldots, u(t-h_{N});
%  \\
% \nonumber
%  &&
%  \hspace{2.0 in} d(t-h_{0}), \ldots, d(t-h_{N}) ),\;\mbox{ a.e. } t \in [S,T]\,.
% \end{eqnarray}
%(Given a function $f(.\,.)$, the function symbol $f(.\,.; .)$, with a semicolon separating its arguments, is used to denote the same function, except that the variables {\it before} the semicolon are modified, according to the initial data, as specified by the values of the variables appearing {\it after}  the semicolon.) 
Note that the right side of (\ref{state}) makes sense
%  even though $x(t)$ and $u(t)$ are defined only for $t \in [S,T]$ and $d(t)$ is defined only for $t \in [S-h,S]$. This is 
because 
%, in view of the definition of $L(t, x_{0}, \ldots,x_{N},u_{0}, \ldots, u_{N}; d_{0},\ldots, d_{N} )$, the vectors 
$x(t-h_{k})$ and  
  $u(t-h_{k})$ need to be evaluated only when $t-h_k \in [S,T]$ and the vector $d(t-h_{k})$ needs to be evaluated only when $t-h_k \in [S-h,S)$.
This formulation of the dynamic constraint and cost covers, as special cases, situations in which there are only time delays in the states, only time delays in the controls, or when the delay times for controls and states differ, since, if the delay times differ, we can take $\{h_{1}, \ldots, h_{N}\}$ to comprise all the time delays (in states and controls). 
\ \\

\noindent
This paper provides necessary conditions of optimality for a `feasible  process' $(\bar x(.),\bar u(.))$ (i.e. a state trajectory/control policy pair satisfying the constraints of the problem) and accompanying initial data to be a minimizer, in the form of a generalized Pontryagin Maximum Principle (PMP). Necessary conditions for optimal control problems with time delays, of this nature, go back to the beginnings of optimal control theory (see, e.g., \cite{banks}). Early derivations of necessary conditions (see, e.g., \cite{halanay}, \cite{huang},  \cite{warga1} and the extensive references in \cite{banks} and \cite{KT}) were typically based on the application of abstract multiplier rules (due to Hestenes, Neustadt, Warga, Gamkrelidze and others), which are specially adapted to the stucture of optimal control problems interpreted as optimization problems over function spaces, and which take account of density theorems relating to `original' and `relaxed' state trajectories, through consideration of Gamkrelidze's `quasiconvex families of functions' (or by other means).
%in the West, see e.g. \cite{huang} and references in \cite{warga}, were based on application of abstract multiplier rules taking account of the kinds of constraints arising on the calculus of variations and optimal control, such as those provded by Neustadt \cite{neustadt} and Warga \cite{warga}. There is an extensive early literature on necessary conditions for time delays problems from the former Soviet Union, references to which are given in \cite{Karah} including,  notably work by Karag et al,in which is based on an extension of Gamkrelidze's work on quasi-convex  
In common with the classical (delay-free conditions), these necessary conditions assert the existence of a `co-state' trajectory $p(.)$ satisfying a co-state equation and transversality conditions, and Weierstrass condition telling us that a Hamiltonian-type function, evaluated along $(\bar x(.),p(.))$ is maximized over possible values of the control variable at $\bar u(.)$. A distinctive feature of these conditions is that the co-state equation is  an `advance functional differential equation', namely
\begin{eqnarray}
\label{adjoint}
\hspace{-0.5 in}
&&
-\dot p(t) \,=\, 
\sum_{k=0}^{N} 
%\chi_{[0,\bar T -h_{i}]}
p(t+h_{k})\cdot
\nabla_{x_{k}}f(t+h_{k},
\{\bar x(t-h_{j}+h_{k})\}_{j=0}^{N}\,,
%_{j=0}^{N}\bar x(t-h_{0}+h_{k}),\ldots, \bar x(t-h_{N}+h_{k})),
\\
\nonumber
&&
\hspace {2.2 in} 
\{\bar u(t-h_{j}+h_{k})\}_{j=0}^{N}; \{\bar d(t-h_{j}+h_{k})\}_{j=0}^{N} )
\quad\mbox{ a.e. }
% t \in [S,T]
\end{eqnarray}
%; \bar d (t- h_{0}+ h_{k}), \ldots,\bar d (t- h_{N}+ h_{k}) )$
%((x_{0},\ldots, x_{N})=  (\bar x(t-h_{0}),\ldots, \bar x(t-h_{N})))
%$
%
%\hspace{1.5 in} 
%$((x_{0},\ldots,,x_{N})=  (\bar x(t-h_{0}+h_{k}),\ldots, \bar x(t-h_{N}+h_{k}))) $
($\nabla_{x_{k}}$ refers to partial differentiation with respect to the $k$'th delayed state argument.)
\ \\

\noindent
We derive necessary conditions of this nature, for general, possibly non-commensurate, delays in both state and control variables. We also provide generalizations in which the initial data (specifying the past histories of $x(.)$ and $u(.)$), are included in the cost, and in which the terminal time $T$ is a choice variable (`free-time' problems). They reduce to Clarke's nonsmooth PMP \cite{clarkeMP} when when there are no time delays. Some special cases of these results were announced in \cite{Boccia}. The novel aspects of our work are as follows: 
\ \\

\noindent
{\it Nonsmoothness}: We provide the first set of  necessary conditions, in the form of a generalized Pontryagin Maximum Principle, for `fully' nonsmooth problems (i.e. problems in which the only regularity hypothesis on the data w.r.t. the state variable is `Lipschitz continuity') involving delays in states and controls. They resemble the classical necessary conditions for `smooth' problems except that, in the costate relation (\ref{adjoint}) classical derivatives are replaced by set-valued subdifferentials of nonsmooth analysis. Two earlier papers \cite{CW}, \cite{cw:91} provide necessary conditions for nonsmooth optimal control problems with time delays in the state alone. The difference is that, in these papers, the dynamic constraint is modelled as a differential inclusion, and the relation for the costate arc (combined with the Weierstrass condition)  is a generalization of Clarke's Hamiltonian
 inclusion condition with `advanced' arguments. As observed in \cite[Section 1,]{cw:91}, necessary conditions expressed in terms of the Hamiltonian inclusion imply the nonsmooth Maximum Principle only for problems having special structure and not `fully' nonsmooth problems, as in this paper. Furthermore, the methods of \cite{CW} and \cite{cw:91} cannot be adapted to cover problems with time delays in the control, because it is not possible to express a controlled delay differential equation (with delays in the control) as a delay differential inclusion (which can take account only of delays in the state). \cite{cw:91} allows both distributed and discrete delays (in the state variable), whereas we allow only discrete delays (in both state and control variables). Necessary conditions for optimal control problems involving differential inclusions are also provided in \cite{cernea} and \cite{Mord} for fixed time optimal control problems involving a single time delay in the state. 
 We mention that Warga \cite{warga2}  showed that a broad class of optimal control problems involving delays and/or functional differential equations, distinct from the problems considered in this paper,  can be fitted to an abstract framework within which nonsmooth necessary conditions can be derived; \cite{warga2} requires a special `additively-coupled' structure for the control delay dependence.
%\ \\
%
%\noindent
%{\it Methodology:} In common with the proof of Clarke's nonsmooth Maximum Principle \cite{clarkeMP}, our analysis is based on the construction of a solution to a simpler `free right endpoint' problem and passage to the limit. In the extensive literature on the use of these so called `perturbational' methods to prove necessary conditions (typified by the analysis in \cite{Vinter}) a key step is `relaxation', which involves replacing the velocity set in the dynamic constraint by its convex hull; this change does not reduce the infimum cost (in the case of a free right end-point problem). A  difference of the analysis in this paper, as compared with its predecessors based on perturbational methods, arises from the fact that, for control systems described by a functional  differential equation with delays in the control, relaxation techniques are no longer available to us in general, because  state trajectories of the convexified control system cannot be approximated by state trajectories of the original control system. A significant new feature of the perturbational techniques in this paper is that the use of relaxation is altogether avoided, the advantages of which are discussed later in this introduction.
%\added{(have removed second `in general', to avoid repetition and have anticipated writing about the advantages of a relaxation-free analysis)}
\ \\

\noindent
{\it Free End-Time:} 
%\added{(I have replaced `Free Time' by `Free End-Time' throughout)}
 This paper treats free end-time optimal control problems with time delays. 
In a delay-free context, optimal control problems with free end-time can be reformulated as standard optimal control problems on a fixed time interval, as a result of  a transformation of the time variable. Optimality conditions for free end-time problems can be obtained from those for fixed end-time problems by applying fixed time conditions to the reformulated problem. For no time delays then, the derivation of free end-time conditions is straightforward. For optimal control problems with time delays, the reduction of free end-time problems to fixed end-time problems, in order to derive optimality conditions, cannot generally be achieved. This is because the time transformation, whose object is  to fix the end-time, also generates a non-standard optimal control problem with time delays, since the time delays now depend on the state variable. We follow a different approach, which is new in a time delays context, based on a perturbation of the end-time. We thereby derive a modified transversality condition, which supplies additional information about the optimal end-time, expressed in terms of the `essential value' of a  maximized Hamiltonian-like function, introduced in \cite{CV1}. Our free-time transversality condition, which is `two-sided', is stronger than the `one-sided' condition in  \cite{KT}.
\ \\

\noindent
{\it The Weierstrass Condition:} A significant feature of the  necessary conditions provided in this paper is that they incorporate an `integral' form of the Weierstrass condition for problems involving general, non-commensurate time delays in state and control. For optimal control problems involving delays only in the state, the integral and pointwise forms of the Weierstrass condition are equivalent. But when we allow non-commensurate control delays, the integral form of the condition (appearing in this paper) is stronger than the pointwise form. While integral forms have been proved in special cases (`additively-coupled' non-commensurate time delays in the control (see, e.g. \cite{warga2}) or commensurate control delays\cite{KT}, only pointwise forms (or weak `differentiated' forms), of the Weierstrass condition are provided for general time delays in the control,  elsewhere in the literature. An exception is the important, but apparently overlooked, work of Warga and Zhu \cite{warga3}. For controlled functional differential equations with non-additively-coupled, non-commensurate control delays, these authors establish the requisite `quasi-convexity' properties required for the  derivation of the integral form of the condition, though they explore their implications to the theory of necessary conditions only in a special case.  Ideas in \cite{warga3} play a key role in the derivation of the integral  condition in this paper.
\ \\

\noindent
{\it Initial Data:} In this paper, the `initial data' function $d(t)=(d^{x}(t),d^{u}(t))$, $S-h\leq t \leq S$, specifying  the past history of state variable (the $d^{x}(.)$ component) and control variable (the $d^{u}(.)$ component) is regarded as a choice variable, which is required to satisfy 
$$
d(t) \in D(t)\mbox{ a.e. } t \in [S-h,S]
$$
and is taken account of in the cost by the integral cost term
`$\,
+ \int_{S-h}^{S} \Lambda(t,d(t))dt
\,$'.
\ \\

\noindent
The multifunction $D(.)$ and the integrand $\Lambda(.,.)$ are  required to satisfy  merely weak measurablility hypotheses and the component of the necessary conditions relating the optimal choice of initial data takes the form of a `strong' Maximum Principle. 
Optimality conditions relating to the initial data to be found in earlier work provide less information (in the case of \cite{CW}) and are derived under much stronger hypotheses.
In \cite{cw:91}, which concerns only state delays, it is assumed that the integral cost term is a Lipschitz function of  $d^{x}(.)$ (w.r.t. the sup norm) and $D(t)$ is required to be closed for each $t$. The relevant component of the necessary conditions is a less informative `weak' Weierstrass condition governing the initial data. 
\cite{KT} provides a `strong' Weierstrass condition (in integrated form) for the initial data, but under stronger hypotheses: $D(t)$ is must be a closed, convex product set and the control delays are assumed  commensurate.
\ \\

 \noindent 
Consider the important special case of a single delay or, more generally, commensurate delays (i.e. all delays are integer multiples of a single positive number). For fixed time problems, all the necessary conditions of this paper can be simply derived, using a transformation technique widely attributed to Guinn (\cite{Guinn}, \cite{maurer}), but which is, in fact, due to Warga \cite{warga0}. This transformation converts an optimal control problem with commensurate delays (in state and control) to a delay-free problem, to which the standard PMP is applicable. (Note that 
%this approach (for commensurate time-delays) only works for fixed end-time problem 
free end-time problems  with commensurate delays cannot be reduced to delay-free problems in this way because, when the end-time is free, the transformed problem does not have a suitable structure for application of delay-free necessary conditions.)
% the transformation introduces a  by this transformation and free end-times  (even with commensurate delays), as studied in this paper, cannot be treated in this way.
%\ \\
%
%\noindent
%%Finally, we make some comparisons with  \cite{KT}, which can be regarded as a culmination of earlier Russian work aimed at  generalizing  the PMP
%%Pontryagin's Maximum Principle 
%for optimal control problems with time delays, via abstract Lagrange multiplier rules. This is one of the very few earlier papers to treat free end-time problems. \cite{KT} provides a `one-sided' transversality condition relating to the free end-time. By contrast, this paper provides additional information about the optimal end-time, via a  `two-sided' transversality condition.   The analysis in \cite{KT} is restricted to the case of commensurate control delays. This is due to difficulties, relating to construction of `quasiconvex families of functions' in a `control-delay' setting. By contrast, our necessary conditions apply for general time delays, which may fail to be commensurate.
%; {\it this is possible because our analytical techniques do not make use of relaxation.} 
%Of course, this paper allows nonsmoothness of the data w.r.t. the state variable, while the data in \cite{KT} is assumed to be differentiable. On the other hand, \cite{KT} covers time-dependent delays, of a kind not considered in this paper.
 \ \\

\noindent
{\it Notation}: 
%\begin{footnotesize}
%Elements $x(.)\in W^{1,1}([a,b];\m R^n)$ are called arcs. 
The Euclidean norm of a vector $x\in \m R^n$ is $|x|$.  $\m B$ indicates the closed unit ball in $\m R^n$. The distance function $d_{A}(.): \R^{n} \rightarrow \R$ of a non-empty set $A\subset\m R^n$ is defined as 
$$
d_{A}(x)\;:=\; \inf\{|x-y| \,:\, y \in A \},\qquad \mbox{for } x \in \R^{n}\,.
$$
The convex hull of the set $A$ is written co$\, A$. Let $I\subset\m R$. The indicator function of the set $I$ is written $\chi_{I}(t):=\{1\mbox{ if }t\in I \mbox{ and }0\mbox{ otherwise} \}$. Given a multifunction $F(.): \R^{n} \leadsto \R^{k}$, we denote by Gr$\, F(.)$ the graph of $F(.)$, namely the set $\{(x,v) \in \R^{n+k}\,|\, v \in F(x) \}$. Given real numbers $a$ and $b$, $a\vee b := \max \{a,b\}$ and $a \wedge b := \min \{a,b\}$.
%\added{(Your def. of the indicator function is fine)]}
 \ \\
 
 \noindent
$W^{1,1}(\,[a,b]\,;\,\m R^n\, )$ denotes the space of absolutely continuous  functions $x:[a,b]\to\m R^n$, with norm 
$$
\|x\|_{W^{1,1}}:=|x(a)|+\int_a^b|\dot x(t)|\,dt\;.
 $$ 
%We write $W^{1,1}$ in place of $ W^{1,1}([a,b];\m R^n)$ when no ambiguity arises.
%\added{(Unnecessary to say we use $W^{1,1}$ notation)}
% $NBV^{+}[a,b]$ denotes the space of increasing, real-valued functions $\mu(.)$ on $[a,b]$  of bounded variation, vanishing at the point $a$ and right continuous on $(a,b)$. The total variation of a function $\mu(.)\in NBV^{+}[a,b]$ is written $||\mu||_{T.V.}$. As is well known, each point $\mu(.)\in NBV^{+}[a,b]$ defines a Borel measure on $[a,b]$. This associated measure is also denoted $\mu$.
%\  \\
%
% \noindent    
We make use of several constructs from nonsmooth analysis, described in detail, for example, in \cite{Vinter} or \cite{ClarkeLed}: given a closed set $E\subset\m R^n$ and $x\in E$,  the proximal normal cone of $E$ at $x$ is
$$
N_E^P(x):= \{
\zeta \in \R^{n}: \exists \,\epsilon>0 \mbox{ {\it and} }M>0 \mbox{ {\it s.t.} } \zeta \cdot (y-x) \leq M|y-x|^{2} \mbox{ {\it for all} } y \in x+ \epsilon \B
\}.
$$
The limiting normal cone at $x$ is
% `normal' directions to the set $E$ at $x$. 
$$
N_E(x):=\{\lim_{i\to\infty}\zeta_i:\zeta_i\in N^P_E(x_i)\mbox{ {\it and} }x_i\in E \mbox{ {\it for all } }i, \mbox{ {\it and} } x_i\to x\}\,.
$$
If $E$ is convex, these two normal cones coincide with the normal of cone of convex analysis.
%, and are written $N_E^{C}(x)$.
%\added{(No need to define normal cone of convex analysis.)}

%The Clarke normal cone of $E$ at $x$ is $N_E^C(x):=co\,N_E^L(x)\,.
%$
%\ \\
%
%\noindent
\ \\
\noindent
%The Clarke tangent cone of $E$ at $x$ is the polar set of the limiting normal cone, i.e.
%$$
%T^{C}_{E}(x):= \{ \eta \in \R^{n}\,|\,  \langle \eta, \zeta \rangle \leq 0 \mbox{ for all } \zeta \in N^{L}_{C}(x)\}\,.
%$$
Given a lower semicontinuous function $f(.):\m R^n\to \m R\cup\{+\infty\}$ and a point $ x\in \mbox{dom }f(.) \,:=\, \{ x\in \m R^n\,|\ f(x) < +\infty\}$, the proximal subdifferential of $f(.)$ at $x$ is the set 
$$
\de_Pf(x):=
\left\lbrace
\begin{array}{l}
\zeta\in\m R^n:\exists \;\sigma>0\text{ and } \e>0 \mbox{ such that, for all } \, y\in x+\e\m B\,,\\ 
\hspace{0.6 in} \<\zeta,y-x\> \leq f(y)-f(x) +\sigma|y-x|^2
\end{array}
\right\rbrace\;.
$$
The limiting subdifferential of $f(.)$ at $x$ is %\added{(Have used $\partial f$, not $\partial^{L}f$ notation throughout)}
$$
\de f(x):=\{\lim_{i\to\infty}\zeta_i:\zeta_i\in \de_Pf(x_i), x_i\to x \mbox{ and } f(x_i)\to f(x)\}\,.
$$
%The partial limiting subdifferential $\partial^{L}_{x_{i}} f(x_{0},\ldots, x_{n})$ of a function of several variables is the limiting subdifferential of the function w.r.t. to the $i$'th argument, when the other arguments are fixed at their base point values.  
%
The partial limiting subdifferential $\partial_{x_{i}}f(\bar x)$ w.r.t. $x_{i}$ at $\bar x=(\bar x_{0}, \ldots, \bar x_{N})$ is  the limiting subdifferential with respect to the $x_{i}$ variable at $\bar x_i$ when the other variables are fixed. The projected limiting subdifferential w.r.t. $x_{i}$ of $f(.):\R^{n}\rightarrow \R$ at $\bar x=(\bar x_{0}, \ldots, \bar x_{N})$, written 
$$
\tilde \partial_{x_{i}} f(\bar x) := \Pi_{x_{i}} \partial f(\bar x)\,,
$$
is the projection of the limiting subdifferential of $f$(.) at $\bar x$ onto the $i$'th coordinate. The partial and projected limiting subdifferentials coincide with the classical partial derivative, when $f(.)$ is continuously differentiable near $\bar x$, but can differ for Lipschitz functions.

%In the case $f(.)$ is Lipschitz continuous on a neighborhood of $x$, we define the Clarke generalized gradient  $\de_Cf(x) := \mbox{ {\it co}}\, \de_L f(x)$.  
%We say that $f(.)$ is strictly differentiable at $x$ if $f(.)$ is Fr\^echet differentiable at $x$, Lipschitz continuous on a neighborhood of $x$ and
%$$
%\de_C f(x) = \{\nabla f(x)\}.
%$$
\vspace{0.1 in}

\noindent
Given an essentially bounded function $h(.):(a,b) \rightarrow \R$ and a point $T \in (a,b)$,  the {\it essential value} of $h(.)$ at $T$ is the closed interval
$$
\underset{t \rightarrow T}{\mbox{ess}}\;h(t) \;:=\;\left[
\underset{\epsilon \downarrow 0}{\mbox{lim}}
 \;
 \underset{T-\epsilon \leq t \leq T+\epsilon }{\mbox{ess inf}}\;h(t),\;\;
 \underset{\epsilon \downarrow 0}{\mbox{lim}}\underset{T-\epsilon \leq t \leq T+\epsilon }{\mbox{ess sup}}\;h(t)
 \right].
$$

%\end{footnotesize}

% Necessary Conditions-----------------------------------------------------------------------

\section{Necessary Conditions for Fixed End-Time Problems}\label{sectionNC}

We consider the following optimal control problem:
$$
(P)
\left\{
\begin{array}{l}
\hspace{0.2 in}\mbox{Minimize } \; 
\displaystyle{J(x(.),u(.),d(.)):=
g(x(S),x(T))+\int_{[S-h,S]} \Lambda(t,d(t))\,dt}
\\
\hspace{0.2 in}+\displaystyle{ \int_{[S,T]} L ( t,\{ x(s-h_{k} )\}, \{u(s-h_{k})\};\{d(t-h_{k})\})  dt
}
%\left\{ 
%\begin{array}{ll}
%x(t -h _{k} )& \mbox{if } t-h_{k} \geq 0
%\\
 %d^{x}(t-h_{k}) & \mbox{if }t-h_{k} < 0
%\end{array}
%\right\}_{k=0}^{N},
%x_{N},u_{0},\ldots, u_{N});d_{0},\ldots,d_{N}),
%\end{eqnarray*},
%\left\{ \begin{array}{ll}
%u(t-h_{k}) & \mbox{if }t-h_{k}) \geq 0
%\\
% d^{u}(t-h_{k} & \mbox{if }t-h_{k} )< 0
%\end{array}
%\right\}_{k=0}^{N}
%)
\\
\hspace{0.2 in} \mbox{over  $x(.) \in W^{1,1} ([S,T];\R^{n})$ and measurable functions}
\\ 
\hspace{01.0 in}\mbox{ 
$u(.):[S,T] \rightarrow \R^{m} $,  $d(.) =(d^{x},d^{u})(.):[S-h,S] \rightarrow \R^{n}\times \R^{m}$, }
\\
\noindent
\hspace{0.2 in} \mbox{such that}
\\
%\hspace{4.0 in}
%\mbox{
%$u(.):[S,T]\rightarrow \R^{m}$ 
%and $d(.):[S-h_{N},S]$}
%\\ 
%\hspace{0.2 in} \mbox{such that } x(.)|_{[S,T]} \mbox{ is absolutely continuous},
\\
\hspace{0.2 in} \dot x(t)= f ( t,\{ x(s-h_{k} )\}, \{u(s-h_{k})\};\{d(t-h_{k})\})
\mbox{ a.e. }t\in [S,T],
\\
%\hspace{0.2 in} (x(t),u(t))= d(t) \mbox{ a.e. }t\in[S,T],
\\
\hspace{0.2 in}u(t)\in U(t)\mbox{ a.e. }t\in[S,T],
\\
\\
\hspace{0.2 in} d(t) \in D(t) \mbox{ a.e. }t\in[S-h,S],
\\
\\
\hspace{0.2 in}(x(S),x(T))\in C\;.
\end{array}
\right.
$$
Here, and below, expressions such as $\{x(t-h_{k})\}$should always be interpreted as $\{x(t-h_{k}\}_{k=0}^{N}$. The index  $k$ will be reserved for such expressions, and the values of $k$ will always run from $0$ to $N$. We write  $h:=h_{N}$.
\ \\

\noindent
The data comprises an interval $[S,T]$, real numbers $h_{0}, \ldots, h_{N}$ such that  $h_{0}=0 <\ldots < h_{N}$, functions
 $g(.,.): \R^{n} \times \R^{n} \rightarrow \R$, $\Lambda(.,.,.):[S-h,S] \times\R^{n}\times \R^{m}\rightarrow \R$, 
 $f(.,.,.):[S,T] \times\R^{(1+N)\times n} \times \R^{(1+N)\times m} \rightarrow \R^{n}$ and $L(.,.,.):[S,T]\x \R^{(1+N)\times n} \times \R^{(1+N)\times m} \rightarrow \R$, multifunctions $U(.):[S,T]\leadsto \R^{m}$ and $D(.):[S-h,S]\leadsto \R^{n+m}$ and a set $C\subset \R^{n} \times \R^{n}$. 
\ \\ 
 
\noindent
A {\it feasible process} is a $3$-tuple $(x(.),u(.), d(.))$, in which $x(.):[S,T]\rightarrow \R^{n}$,  $u(.):[S,T]\rightarrow \R^{m}$ and $d(.):[S-h,S]\rightarrow \R^{n+m}$ are functions satisfying the constraints in ($P$), and for which $t\rightarrow \Lambda(t,d(t))$ in integrable on $[S-h,S]$ and
$$
t\rightarrow L( t,\{ x(s-h_{k} )\}, \{u(s-h_{k})\};\{d(t-h_{k})\})
$$
 is integrable on $[S,T]$.
We say that a feasible process $(\bar x(.),\bar u(.),\bar d(.))$ is a {\it $L^{\infty}$} local minimizer if there exists $\epsilon >0$ such that
$$
J(x(.),u(.),d(.))\,\geq\, J(\bar{x}(.),\bar{u}(.),\bar d (.))
$$
for any feasible process $(x(.), u(.),d(.))$ satisfying
%$(\|x(.)-\bar x(.))|_{[S,T]}\|_{L^{\infty}([S,T];\R^{n})} \leq\e $.
$\|x(.)-\bar x(.) \|_{L^{\infty}([S,T]; \R^{n})} \leq \e $.
\ \\

\noindent
We shall invoke
the following hypotheses,
in which $\tilde f(t,\{x_{k}\},\{u_{k}\}):=
(f,L)(t,\{x_{k}\}, \{u_{k}\})$ and $(\bar x(.),\bar u(.),\bar d(.))$ is a given feasible process.  For some $\e>0$:
\begin{itemize}
\item[(H1):] $g(.,.)$ is Lipschitz continuous on  $(\bar x(S),\bar x(T)) + \epsilon \B$ and
$C$ is a closed subset of $\R^{2n}$. 
\item[(H2):] For every $z \in \R^{(1+N)\times n}$, the function $\tilde f(.,z,.)$ is ${\cal L}([S,T])\times {\cal B}$ measurable, the set Gr$\,U(.)$  is
${\cal L}([S,T])\times {\cal B}$ measurable and the set Gr$\,D(.)$ is
${\cal L} ([S-h,S])\times {\cal B}$ measurable. Here, for a given interval $I$, ${\cal L}(I)$ denotes
Lebesgue subsets of $\R$,  ${\cal B}$ denotes the Borel sets of a Euclidean space, and ${\cal L}(I)\times {\cal B}$ denotes the product $\sigma$-algebra.
\item[(H3):]
There exists a function a Borel measurable function $k(.,(.\,.), (.\,.)): [S,T] \times \R^{n \times (N+1)}\times \R^{m \times (N+1)}$
%\},\{d_{k}\})$ $(t,\{u_{k}\},\{d_{k}\})\rightarrow k(t,\{u_{k}\},\{d_{k}\})$ 
such that 
 $t \rightarrow k(t,\{\bar u(t-h_{k})\},\{\bar d(t-h_{k})\})$  is integrable and  
 $$
 \mbox{$ \tilde f(t,., \{u_{k}\};\{d_{k})$ is $k(t,\{u_{k}\},\{d_{k}\})$-Lipschitz on  $\{\bar x(t-h_{k})\}+ \epsilon 
 \B^{n \times (N+1)}$ }
 $$
for  all $\{u_{k}\} \in U(t-h_{0}) \times\ldots \times U(t-h_{N})$,  $\{d_{k}\} \in D(t-h_{0}) \times\ldots \times D(t-h_{N})$  a.e. $t \in [S,T]$.
\end{itemize}

\noindent
Fix $x(.):[S,T]\rightarrow \R^{n}, u (.):[S,T]\rightarrow \R^{m}$ , $d (.):[S-h,S]\rightarrow \R^{n+m}$, $\lambda \geq 0$ and $p(.): [S,T] \rightarrow \R^{n} $.  
Define, for $t \in [S,T]$, ${\bf u} \in\R^{m}$ and ${\bf d} \in \R^{m}$  
\begin{eqnarray}
\label{Hamiltonian}
&& \hspace{-0.2 in }{\cal H}_{\lambda}(t,{\bf u}, {\bf d}; x(.), u(.),d(.), p(.)) \, :=\,
(p(t+h_{0})\cdot f  - \lambda L)
%(.\,.\,.)
\\
\nonumber
%\nonumber
%&&  :=\,
%p(t+h_{0})\cdot f  - \lambda L(.\,.)
%\\
&&
(t+h_{0}, \{x(t+h_{0}-h_{k})\}),{\bf u}, u(t+h_{0}-h_{1}),\dots,
\\
\nonumber
&& 
\nonumber
\hspace{0.5 in}
  u(t+h_{0}-h_{N});{\bf d}, d(t+h_{0}-h_{1}),\dots,  d(t+h_{0}-h_{N}))\chi_{[S,T]}(t+h_{0})
\\
\nonumber
&& 
\hspace{0.6 in}+  \ldots + \ldots
\\
\nonumber
&&
+ \,(p(t+h_{N})\cdot f  - \lambda L)
%(.\,.\,.)
\\
\nonumber
&&
\hspace{0.1 in}
(t+h_{N}, \{x(t+h_{N}-h_{k})\}), u(t+h_{N}-h_{0}),\dots,  
\\
\nonumber
&&
 \hspace{0.5 in}
u(t+h_{N}-h_{N-1}),{\bf u}; d(t+h_{N}-h_{0}),\dots,  d(t+h_{N}-h_{N-1}), {\bf d})\chi_{[S,T]}(t+h_{N}).
\end{eqnarray}
%%
%%
%%
%\begin{eqnarray}
%\label{Hamiltonian}
%&& \hspace{-0.2 in }{\cal H}_{\lambda}(t,{\bf u}; x(.), u(.),d(.), p(.)) \, :=\,
%(p(t+h_{0})\cdot f  - \lambda L)(.\,.\,.)
%\\
%\nonumber
%\nonumber
%&&  :=\,
%p(t+h_{0})\cdot f  - \lambda L(.\,.)
%\\
%&&
%(t+h_{0}, x(t+h_{0}-h_{0}), \ldots, x(t+h_{0}-h_{N}),{\bf u}, u(t+h_{0}-h_{1}),\dots,
%\\
%\nonumber
%&& 
%\nonumber
%\hspace{0.5 in}
%  u(t+h_{0}-h_{N});d(t+h_{0}-h_{0}),\dots,  d(t+h_{0}-h_{N}))\chi_{[S,T]}(t+h_{0})
%\\
%\nonumber
%&& 
%\hspace{0.6 in}+  \ldots + \ldots
%\\
%\nonumber
%&&
%+ \,(p(t+h_{N})\cdot f  - \lambda L)(.\,.\,.)
%\\
%\nonumber
%&&
%\hspace{0.1 in}
%(t+h_{N}, x(t+h_{N}-h_{0}), \ldots, x(t+h_{N}-h_{N}), u(t+h_{N}-h_{0}),\dots,  
%\\
%\nonumber
%&&
 %\hspace{0.5 in}
%u(t+h_{N}-h_{N-1}),{\bf u}; d(t+h_{N}-h_{0}),\dots,  d(t+h_{N}-h_{N}))\chi_{[S,T]}(t+h_{N}).
%\end{eqnarray}
%Define also,  for $t \in [S-h,S]$,  $d \in D(t)$,
%%%%
%%%%

\begin{thm}\label{main}
\vspace{0.1 in}

\noindent
Let  $(\bar x(.),\bar u(.),\bar d(.)) $ be an $L^{\infty}$ local minimizer for (P).
%-weak minimizer for $(P)$, for a given radius multifunction $R(.)$.
Assume hypotheses (H1)-(H3) are satisfied for some $\epsilon >0$.
\ \\

\noindent
Then there exist  $p_{k}(.)\in W^{1,1}([S-h_{k},T];\R^{n})$, $k=0,\ldots, N$, and  $\lambda \geq 0$ such that
\begin{equation}
\label{p_comp}
\left\{
\begin{array}{ll}
\dot p_{k}(t)=0 &\mbox{ for } t \in [S-h_{k},S] 
\\
p_{k}(t)=0 &\mbox{ for } t \in [(T-h_{k})\vee S,T]  \,,
\end{array}
\right.
%\mbox{, for $j=1,\dots,N$\,.
\end{equation}
for  $k=1,\ldots,N$, with the following properties, in which $p(.) \in W^{1,1}([S,T];\R^{n})$ is the function 
%%\added{(Have replaced `$k=0,..$' by `$k=1,..$')}
\begin{equation}
\label{p_eqn}
p(t):= \sum_{k=0}^{N} p_{k}(t) \mbox{ for } t \in [S,T]\,.
\end{equation}

% \mu(.),\lambda^{0})\in W^{1,1}\x NBV^{+}[a,b]\x\m R^{+}$ and a $\mu$-integrable function $m(.):[a,b]\rightarrow \R^{n}$ such that
\begin{itemize}
\item[(a):] (nontriviality)
$(\lambda,p(.))\,\not=\, 0$,
\item[(b):] (adjoint inclusion) 

\hspace{-0.3 in}$\{-\dot p_{0}(t-h_{k})\} \in
%$\hspace{0.2 in} \in 
\mbox{co}\, \partial_{\{x_{k}\}}\,(p\cdot f -\lambda (L+ \Lambda))
(t,\{x_{k}\},\{ \bar u(t-h_{k}) \}; \{\bar d(t-h_{k})\})$

\hspace{3.0 in}$(\{x_{k}\}=  \{\bar x(t-h_{k})\})$,
  $\mbox{a.e. } t \in [S-h,T]$.
\item[(c):] (integral Weierstrass condition)
%
%\hspace{-1.5 in}(c)^{*}:
%&& 
%(c)$^{\prime}$
\begin{eqnarray}
\label{integral}
&&\int_{[S-h,T]}(p\cdot f-\lambda (L+ \Lambda))(t,\{\bar x(t-h_{k})\},\{u(t-h_{k})\};\{d (t- h_{k})\})dt \hspace{0.3 in}
\\
\nonumber
&&
\hspace{0.3 in}
\leq \, \int_{[S-h,T]}(p\cdot f-\lambda(L+ \Lambda))(t,\{\bar  x(t-h_{k})\},\{ \bar u(t-h_{k})\};\{\bar d (t- h_{k})\})dt  \hspace{0.6 in}
\end{eqnarray} 
 for any selectors $u(.)$ of $U(.)$ and $d(.)$ of $D(.)$ such that the integrand on the left side of (\ref{integral}) is integrable.
%\mbox{for all selectors } u(.) \mbox{ of } U(.)\,,
%\nonumber
%\end{eqnarray*}

%%
%%
%
%\noindent
%\hspace{4.0 in} 
%
%\ \\
%
%\noindent
%
%
%\item[(d):] 
%
%$
%{\cal H}^{0}_{\lambda}(t,\bar d(t); \bar x(.), \bar u(.), \bar d(.), p(.))= \underset{d \in D(t)}{\max} 
%{\cal H}^{0}_{\lambda}(t,d; \bar x(.), \bar u(.), \bar d(.), p(.))  
%$
%
%\noindent
%\hspace{4.0 in}  
%a.e. $t \in [S-h,S]$\,.
  \item[(d):] $\left(p(S),-p(T)\right) \in \lambda \partial g(\bar x(S),\bar x(T))+N_C(\bar x(S),\bar x(T))$. 
\end{itemize}
It follows from the definition (\ref{p_comp}) and (\ref{p_eqn}) of $p(.)$ and from condition (b) that $p(.)$ satisfies the `advance' functional differential inclusion:
\begin{itemize}
\item[(b$^{*}$):] $-\dot p(t) \in \sum_{j=0}^{N} \chi_{[S,T -h_{k}]}(t)
\, \mbox{co}\, \tilde{\partial}_{x_{j}}\; (p\cdot
f-\lambda L)
$

$\noindent
(t+h_{j},
\{x_{k}\},\{\bar u(t-h_{k}+h_{j})\}_{k}, \{\bar u(t-h_{k}+h_{j})\}; \{\bar d (t- h_{k}+ h_{j})\}$
%((x_{0},\ldots, x_{N})=  (\bar x(t-h_{0}),\ldots, \bar x(t-h_{N})))

\hspace{2.8 in} 
$(\{x_{k}\}=  \{\bar x(t-h_{k}+h_{j})\})
%\ldots, \bar x(t-h_{N}+h_{k}))) 
\hspace{0.2 in}  \mbox{ a.e. } t \in [S, T]$
\end{itemize}
in which $\tilde{\partial}_{x_{i}}$ denotes the projected  limiting subdifferential onto the  $i$'th delayed state coordinate.
\vspace{0.00 in}

\noindent
Condition (c) implies
\begin{itemize}
\item[(c$^{*}$):] (Pointwise Maximum Principle)

$
{\cal H}_{\lambda}(t,\bar u(t), \bar d(t) ; \bar x(.), \bar u(.), \bar d(.), p(.))= \underset{u \in U(t),\, d \in D(t)}{\max}
{\cal H}_{\lambda}(t,u,d; \bar x(.), \bar u(.), \bar d(.), p(.))
$
\end{itemize}
\hspace{4.4 in} a.e. $t \in [S-h,T]$\,, 
%\added{(Replaced `coordinate' by `delayed state coordinate')} 
  %$-\lambda \int_{S-h_{k}}^{S}\Lambda (t, \phi(t),u(t))dt +$
% 
 %\; 
% $\int_{S}^{T}
%\left(
%p(t)\cdot f(t), \bar x(t+h_{0};\phi(.)), \ldots,\bar x(t+h_{k};\phi(.)), u(t-h_{0},\ldots, u(t-h_{k}))
%\right. $
%
%$\left. \hspace{0.8 in}-\lambda 
%L(t), \bar x(t+h_{0};\phi(.)), \ldots,\bar x(t+h_{k};\phi(.)),u(t-h_{0},\ldots, u(t-h_{k}))\right)dt$
%
%$\geq \; - \lambda \int_{S-h_{k}}^{S}\Lambda (t, \bar \phi(t),\bar u(t))dt +$
%
 %$\; \int_{S}^{T}
%\left(p(t)\cdot f(t), \bar x(t+h_{0};\bar \phi(.)), \ldots,\bar x(t+h_{k};\bar \phi(.)), u(t-h_{0},\ldots, \bar u(t-h_{k}))
%\right. $
%
%$\hspace{0.8 in} \left. -\lambda 
%L(t), \bar x(t+h_{0};\bar \phi(.)), \ldots,\bar x(t+h_{k};\phi(.)),\bar u(t-h_{0},\ldots, \bar u(t-h_{k}))\right)dt$
%
%\hspace{2.5 in} for all selectors $u(.)$ of $U(.)$ and $\phi(.)$ of $\Phi(.)$.
%\end{itemize}
%Condition (d) implies (but is not in general equivalent to) the `pointwise' conditions:
%\begin{itemize}
%\item[(e):]
%$u \rightarrow, for $j=0,\ldots,N$,$ is maximized over $U(t)$ at $\bar u(t)$, a.e. $t \in [S,T]$\,.
%\end{itemize}
%and
%\begin{itemize}
%\item[(f):]
%$(\phi,u)\rightarrow - \lambda \Lambda(t,\phi,u)+ {\cal H}_{\lambda}(t, \phi,u ;\bar x(.), \bar u(.), \bar \phi (.) )$ is maximized over $U(t) \times \Phi(t)$ 
%
%
%\hspace{3.0 in}
%at $(\bar u(t) , \bar \phi(t))$, a.e. $t \in [S-h_{k},S]$.
%
%\end{itemize}
\end{thm}
\vspace{0.1 in}
\noindent
A proof of Thm. \ref{main} is given in a later section.

\noindent
\vspace{0.1 in}

\noindent

\noindent
{\bf Comments}
\begin{itemize}
\item[(a):] The key conditions in Thm. \ref{main} are (a), (b$^{*}$), (c)-(e) involving the function $p(.)$.  When $(\{x_{k}\}) \rightarrow (f,L)(t,\{x_{k}\},\{u_{k}\}; \{d_{k}\})$ 
is $C^{1}$ near $(\{\bar x(t-h_{k})\})$ , they reduce to standard necessary conditions  expressed in terms of a costate function $p(.)$ satisfying the `advance' functional differential equation (\ref{adjoint}) of the Introduction.  
The condition (b), expressed in terms of the collection of $p_{k}(.)$'s in the sum decomposition (\ref{p_comp}) of $p(.)$ is a more precise condition than (b$^{*}$) in a non-smooth setting because the subdifferential $\partial_{\{x_{k}\} } (p\cdot f)$ is a subset, and in some cases a strict subset, of the product of projected partial subdiffererentials
$\tilde \partial_{x_{0}} (p\cdot f)\times \ldots \times \tilde \partial_{x_{N}} (p\cdot f)$.
\item[(b):] The integral Weierstrass condition (c) (for control functions and initial data functions), which allows simultaneous variation of the entries in all the control delay slots (provided they are all associated with some control function) is a stronger condition (when there are time delays in the control), than the pointwise condition (c)$^{*}$, expressed in terms of the Hamiltonian function (\ref{Hamiltonian}), which involves variations in the control slots only one at a time. Note that, elsewhere in the literature (see, e.g., \cite{warga2}), integral forms of the Weierstrass condition are given only under the addition hypothesis that the control delays are additively coupled. Apparently the only exception is \cite{warga3}, where, in a special setting, Warga and Zhu derive an integral form of the condition for non-additively coupled control delays. 
\item[(c):] The initial data segments $d^{x}(s)$ and $d^{u}(s)$, $s\in[S-h,S]$, for the $x$ and $u$ variables are treated in a more general way than in the previous literature. Here they are regarded as choice variables that are required to satisfy $d(t) =(d^{x},d^{u})(t) \in D(t)$ a.e.. $D(t)$ need not be closed or bounded, and is not necessarily a product set that captures, separately, constraints on initial state and initial control segments. The optimal initial segment $\bar d(.)$ is characterized by two versions of `strong' Weierstrass  condition (the pointwise condition ($d$) and integral condition (d$^{*}$)), which provide more information about $\bar d(.)$ than the `weak' condition in \cite{cw:91}  (when it is  applicable) expressed in terms of normal cones of $D(t)$. An integrated version of the `strong' Weierstrass condition on the initial data for the optimal state variable is included in the necessary conditions of \cite{KT}, but D(t) is required to be a closed convex set. 
%\added{(Have incorporated Andrea's corrections + small changes to English)}
\item[(d):] Our nonsmooth necessary conditions allow time delays in the control. They improve on earlier nonsmooth necessary conditions for time delay problems, which allow delays only in the state \cite{CW}, \cite{cw:91}, or require a separable structure for the control delay dependence \cite{warga2}. In common with \cite{warga3}, they improve on available necessary conditions  in \cite{KT} for smooth problems with delays in both state and control, because they do not require the control delays to be commensurate.
%\added{(Small changes to English)}
\end{itemize}

%%%%%%%%%%%%%%%%%%%%%%%%%%%%%%%%%%%%%%%%%%%%%%%%%%%%%%%%%%%%%%%%%%%%%%%%%%%%
\section{Necessary Conditions for a Free End-Time Problem}
Consider next a related problem to ($P$) above, in which the end-time $T$ is free, and included in the choice variables, and in which there are no control delays.
$$
(P_{FT})
\left\{
\begin{array}{l}
\hspace{0.2 in}\mbox{Minimize } \; 
J(x(.),u(.),d(.),T):=
\tilde g(x(S),x(T),T)+\displaystyle{\int_{[S-h,S]} \Lambda(t,d(t))\,dt}
\\
\hspace{0.2 in}+\displaystyle{ \int_{[S,T]} 
L ( t,\{ x(t-h_{k})\},u(t);
\{d(t-h_{k})\})dt,}
%\begin{array}{ll}
%x(t -h _{k} )& \mbox{if } t-h_{k} \geq 0
%\\
 %d^{x}(t-h_{k}) & \mbox{if }t-h_{k} < 0
%\end{array}
%\right\}_{k=0}^{N},
%x_{N},u_{0},\ldots, u_{N});d_{0},\ldots,d_{N}),
%\end{eqnarray*},
%\left\{ \begin{array}{ll}
%u(t-h_{k}) & \mbox{if }t-h_{k}) \geq 0
%\\
 %d^{u}(t-h_{k} & \mbox{if }t-h_{k} )< 0
%\end{array}
%\right\}_{k=0}^{N}
%\right)
% dt
\\
\\
\hspace{0.2 in} \mbox{over  $T \geq S$, $x(.) \in W^{1,1} ([S,T];\R^{n})$ and measurable functions}
\\ 
\\
\hspace{01.0 in}\mbox{ 
$u(.):[S,T] \rightarrow \R^{m} $,  $d(.):[S-h,S] \rightarrow \R^{n}$}
%\hspace{4.0 in}
%\mbox{
%$u(.):[S,T]\rightarrow \R^{m}$ 
%and $d(.):[S-h_{N},S]$}
%\\ 
%\hspace{0.2 in} \mbox{such that } x(.)|_{[S,T]} \mbox{ is absolutely continuous},
\\
\hspace{0.2 in}  \mbox{such that}
\\
\\
\hspace{0.2 in} \dot x(t)= f(
t,\{x(t-h_{k})\}, u(t); \{d(t-h_{k})\})
\\
\hspace{4.0 in}
\; \mbox{ a.e. }t\in [S,T],
\\
\hspace{0.2 in}u(t)\in U(t) \mbox{ a.e. }t\in[S,T],
\\
\\
\hspace{0.2 in} d(t) \in D(t) \mbox{ a.e. }t\in[S-h,S]
\\
\\
\hspace{0.2 in}(x(S),x(T),T)\in \tilde C\;.
\end{array}
\right.
$$
%%%
%$$
%(P_{FT})
%\left\{
%\begin{array}{l}
%\hspace{0.2 in}\mbox{Minimize } \; 
%J_{FT}(x(.), \phi(.),u(.),T):=
%g(x(S),x(T),T)+\int_{S-h_{k}}^T\Lambda(t,\phi(t),u(t))\,dt
%\\
%\hspace{0.2 in}+ \int_{S}^{T} L(t, x(t-h_{0};\phi(.)),\ldots,x(t-h_{k};\phi(.)),u(t-h_{0}), \ldots u(t-h_{k})) dt
%\\
%\hspace{0.2 in}\mbox{over $T > S $, $x(.)\in W^{1,1}([S,T];R^{n})$  and measurable functions } \\
%\hspace{2.3 in}
%$u(.):[S,T]\rightarrow \R^{m}$ 
% \mbox{$u(.):[S-h_{k},T]\rightarrow \R^{m}$ and $\phi(.):[S-h_{k},S]$}
%\\ 
%\hspace{0.2 in}\mbox{such that }
%\\
%\hspace{0.2 in} \dot x(t)=f(t, x(t-h_{0};\phi(.)),\ldots,x(t-h_{k};\phi(.)),u(t-h_{0}), \ldots u(t-h_{k})),
%\\
%\hspace{4.2 in}\mbox{ a.e. } t\in[S,T],
%\\
%\hspace{0.2 in}u(t)\in U(t): \mbox{ a.e. }t\in[S-h_{k},T],
%\\
%\hspace{0.2 in} \phi(t) \in \Phi(t) \mbox{ a.e. }t\in[S-h_{k},S]
%\\
%\hspace{0.2 in}( x(S),x(T),T)\in E\;.
%\end{array}
%\right.
%$$
Here, $h:= h_{N}$. The data for ($P_{FT}$) comprises a real number  $S$, real numbers $h_{0}, \ldots, h_{N}$ such that  $0=h_{0} <\ldots < h_{N}$, functions
 $\tilde g(.,.,.): \R^{n} \times \R^{n}\x \R \rightarrow \R$, $\Lambda(.,.):[S-h_{k},S] \times\R^{n}\rightarrow \R$, 
 $f(.,.,.):[S,\infty) \times\R^{(1+N)\times n} \times \R^{ m} \rightarrow \R^{n}$ and $L(.,.,.):[S,\infty) \times\R^{(1+N)\times n} \times \R^{m} \rightarrow \R$, multi-functions $U(.):[S,\infty)\leadsto \R^{m}$ and $D(.):[S-h,S]\leadsto \R^{n}$ and a set $\tilde C \subset   \R^{n} \times \R^{n}\x [0,\infty)$. 
 %(The notation $x(t-h;\phi(.))$ was explained in the previous section.)
\ \\
%%%

\noindent
A {\it feasible process}  for (P$_{FT}$) is a $4$-tuple $(x(.),u(.),d(.) ,T)$, where $T$ is a number ($T \geq S$), $x(.):[S,T] \rightarrow \R^{n}$, $u(.):[S,T]\rightarrow \R^{m}$ and $d(.):[S-h,S]\rightarrow \R^{n}$ are functions in the specified spaces, satisfying the constraints in ($P_{FT}$) and such that $t\rightarrow \Lambda(t,d(t))$ is integrable on $[S-h,S]$, and 
$
t \rightarrow L(t, \{x(t-h_{k})\};\{ d(t-h_{k})\})
$
is integrable on $[S,T]$.
\\

\noindent
A feasible process $(\bar x(.),\bar u(.),\bar d(.), \bar T)$ is {\it an {\it $L^{\infty}$} local minimizer} if there exists $\epsilon >0$ such that
$$
J(x(.), u(.),d(.),T)\,\geq\, J(\bar{x}(.), \bar{u}(.),\bar d(.), \bar T)
$$
for any feasible process $(x(.), u(.),d(.),T)$ satisfying
$
\|x(.)-\bar x(.)\|_{L^{\infty} ([S,T\wedge \bar {T}];\R^{n})} + |T-\bar T| \leq\e\,.
$
% in which $a \wedge b := \min \{a,b\}$.
%\todo[inline]{Usually the metric considered is $
%\|x(.)-\bar x(.)\|_{L^{\infty} ([S,T\vee \bar {T}];\R^{n})} + |T-\bar T| \leq\e\,
%$ where $x$ and $\bar x$  are extended constant out of their interval of definition. Does this matter?}
\ \\

\noindent
We shall invoke
the following hypotheses, in which  $\tilde f(.,.,.):=(f, L)(.,.,.)$ and
$(\bar x(.),\bar u(.),\bar  d(.),\bar T)$ is a given feasible process.
%and $T'$ is some given number $T' > \bar T$.  
 For some $\e>0$:
\begin{itemize}
\item[(HFT1):] 
$g(.,.,.)$ is Lipschitz continuous on  $(\bar x(S),\bar x(\bar T),\bar T) + \epsilon \B^{2n+1}$.
$\tilde C$ is a closed subset of $\R^{2n+1}$. 
\item[(HFT2):] For every $z \in \R^{(1+N)\times n}$, the function $\tilde f(.,z,.)$ is ${\cal L}([S,\bar T + \epsilon])\times {\cal B}$ measurable, the set Gr$\,U(.)$  is
${\cal L}([S,\infty])\times {\cal B}$ measurable and the set Gr$\,D(.)$ is
${\cal L} ([S-h,S])\times {\cal B}$ measurable. 
\item[(HFT3):]
There exists a function a Borel measurable function $k(.\,.): [S,\bar T] \times \R^{m}\times \R^{n \times (N+1)}$ and numbers $\bar k >0$ and $\bar c > 0$ such that
%\},\{d_{k}\})$ $(t,\{u_{k}\},\{d_{k}\})\rightarrow k(t,\{u_{k}\},\{d_{k}\})$ 
$t \rightarrow k(t,\bar u(t),\{\bar d(t-h_{k})\})$  is integrable on $[S,\bar T]$ and  
\vspace{0.1 in}

 $
 \hspace{0.3 in} \mbox{$ \tilde f(t,., u;\{d_{k}\})$ is $k(t,u,\{d_{k}\})$-Lipschitz on  $\{ \bar x(t-h_{k})\}+ \epsilon 
 \B^{n \times (N+1)}$ }
 $
\vspace{0.1 in}

 \noindent 
for  all $u \in U(t)$,  $\{d_{k}\} \in D(t-h_{0}) \times\ldots \times D(t-h_{N})$  a.e. $t \in [S,T]$. Moreover
$$
\tilde f(t,., u;\{d_{k}\}) \mbox{ is } \bar k\mbox{-Lipschitz  and }\bar c\mbox{-bounded on } \{\bar x(t-h_{k})\}+ \epsilon 
 \B^{n \times (N+1)} 
 $$
 for all $u \in U(t)$,  $\{d_{k}\} \in D(t-h_{0}) \times\ldots \times D(t-h_{N})$  a.e. $t \in [\bar T -\epsilon,\bar T+ \epsilon]$, 
\end{itemize}
\vspace{0.1 in}

\noindent
There follows a set of necessary conditions for $(\bar x(.),\bar u(.),\bar d(.), \bar T)$ to be an $L^{\infty}$-local minimizer for  the free end-time problem ($P_{FT}$). 
%Here ${\cal H}_{\lambda}(.,.,.;.,.,.)$ is as in (\ref{Hamiltonian}). 
%%%
\begin{thm}\label{thm3_1}
Let  $(\bar x(.),\bar u(.),\bar d(.), \bar T) )$ be an $L^{\infty}$ local minimizer for $(P_{FT})$.
%-weak minimizer for $(P)$, for a given radius multifunction $R(.)$.
Assume hypotheses (HFT1)-(HFT3). Assume also that $\bar T -S >  h$\,.
\ \\

\noindent
Then there exists  $p_{k}(.)\in W^{1,1}([S-h_k,\bar T];\R^{n})$ such that
\begin{equation*}
%\label{p_comp}
\left\{
\begin{array}{ll}
\dot p_{k}(t)=0 &\mbox{ for } t \in [S-h_{k},S] 
\\
p_{k}(t)=0 &\mbox{ for } t \in [ \bar T-h_{k},\bar T]  \, \,,
\end{array}
\right.
%\mbox{, for $j=1,\dots,N$\,.
\end{equation*}
$k=1,\ldots, N$,  $\lambda \geq 0$ and  $\xi \in \R$, with the following properties, in which 
$p(.)\in W^{1,1}([S,\bar T];\R^{n})$ is the function:
%\added{(Have changed `$k=0$,. . ' to $k=1,..$)}
\begin{equation}
\label{p equation}
p(t)=\sum_{k=0}^{N}p_{k}(t) \quad \mbox{for } t \in [S,\bar T]\,.
\end{equation}
Conditions (a)-(d), (b$^{*}$) and  (c$^{*}$) of Thm. \ref{main} (in which $(f,L)$ does not depend on the `delayed' control variables $(u(t-h_{1}),\ldots,u(t-h_{N}))$ and $[S,T]$ is replaced by $[S, \bar T]$). Furthermore, in place of  (d), the following `free end-time' transversality condition is satisfied:
\begin{itemize}
  \item[(d$\,'$):] $\qquad \left(p(S),-p(\bar T),\xi \right) \in \lambda \,\partial \tilde g(\bar x(S),\bar x(\bar T),\bar T)+N_{\tilde C}(\bar x(S),\bar x(\bar T), \bar T)$, 
  \end{itemize}
  in which  $\xi$ is some number that satisfies
$$
\xi \in \underset{t\rightarrow \bar T}{\mbox{ess}}\;
\left\{
 \max_{u \in U(t)}\;(p(\bar T)\cdot f-\lambda L)(t,\{\bar x(\bar T-h_{k})\}, u)
 %; \bar d(\bar T-h_{0}),\ldots, \bar d(\bar T-h_{N}) )
 \right\}\,.
$$
%The function $p(.)$, given by (\ref{p equation}), satisfies
%\ \\
%
%$-\dot p(t) \in \sum_{i=0}^{N} \chi_{[0,\bar T -h_{i}]}
%\,
%  \mbox{co}\, \tilde{\partial}^{L}_{x_{i}}\; p(t+h_{i})\cdot$
%$f(t+h_{i},
%(x_{0},\ldots, x_{N}),\bar w(t+h_{i}); \xi(.),\eta(.))$
%((x_{0},\ldots, x_{N})=  (\bar x(t-h_{0}),\ldots, \bar x(t-h_{N})))
%
%\hspace{2.0 in} 
%$((x_{0},\ldots,,x_{N})=  (\bar x(t-h_{0}),\ldots, \bar x(t-h_{N}))) 
%\mbox{ a.e. } t \in [0,\bar T]$.
%Here, for any measurable function $\phi(.):\R\rightarrow \R$ and $\bar t \in \R$
%$$
%\underset{t\rightarrow \bar t}{\mbox{ess}}\; \phi(t) \; :=\;  \left[  \underset{\delta \downarrow 0}{\mbox{lim}} \;\underset{t \in [\bar{t}-\delta, \bar t + \delta]}{\mbox{ess inf}} \; \phi(t),\;
% \underset{\delta \downarrow 0}{\mbox{lim}}\;
% \underset{t \in [\bar{t}-\delta, \bar t + \delta]}{\mbox{ess sup}}\phi(t) \right] 
%$$  
%(the `essential value' of $\phi(.)$ at $\bar{t}$.)
%\end{itemize}
\end{thm}

\noindent
{\bf Comments}
\begin{itemize}
\item[(a):] Condition (d$'$) is  apparently the first generalized tranversality condition for free end-time optimal control problems with delays, when the data is assumed to be merely measurable w.r.t. to the time variable. But it provides new information even in the continuous case. Indeed, suppose that $\tilde C=C \times \R$, i.e. the free end-time $T$ is unconstrained and  $g$ is a $C^{1}$ function, $(f,L)$ are continous w.r.t. the time variable and $U(t)$ is a constant compact set U. Then the `essential value' is single valued and the transversality condition provides the following information about the optimal end-time $\bar T$:
\begin{equation}
\label{2_sided}
 \max_{u \in U}\;\big( p(\bar T)\cdot f-\lambda L\big)(\bar T,\{ \bar x(\bar T-h_{k})\}, u) = \nabla_{T}g(\bar x(S), \{\bar x(\bar T),\bar T)\,.
 %; \bar d(\bar T-h_{0}),\ldots, \bar d(\bar T-h_{N}) )
\end{equation}
Note that this is an equality (or `two sided') relation. The necessary conditions in \cite{KT} include an inequality (or `one sided') version of this relation which conveys less information. 
\item[(b):]
The transversality condition (\ref{2_sided}) is used in \cite{Boccia}, to derive sensitivity relations and construct algorithms for the computation of solutions to free end-time optimal control problems with delays.
\end{itemize}

%%%%%%%%%%%%%%%%%%%%%%%%%%%%%%%%%%%%%%%%%%%%%%%%%%%%%%%%%%%%%%%%%%%%%%%%%%%%
\section{Proof of Theorem \ref{main}}
\label{Thm_2.1}
%We prove only the special case when $L(.,(\ldots), (\ldots)) \equiv 0$, since an integral cost  (involving the state and control) can be reduced to the special case by state augmentation.
%\ \\
%
%\noindent
%We also limit attention to the case when the following additional hypotheses are also satisfied. 

\noindent
We shall make extensive use of the following existence theorem for delay differential equations with accompanying estimates, which is can be regarded as a generalization of Filippov's Existence Theorem \cite[Thm. 2.4.3.]{Vinter} for (delay free) differential inclusions. 

\begin{thm}
%[Filippov's Existence Theorem]
\label{Filippov}
Take $0=:h_{0} < \ldots < h_{N-1} < h_{N}$, $\bar \epsilon  \in (0, \infty) \cup \{+ \infty\}$, a function $f(.,.):[S,T]  \times \R^{(N+1)\times n} \rightarrow \R^{n}$ and a measurable function $d(.):[S-h,S] \rightarrow \R^{n}$. 
Take also $y(.) \in W^{1,1}([S,T];\R^n)$ (the `reference trajectory') and $\xi \in \R^{n}$ (the `initial state').
Assume
\begin{itemize}
\item[(a):] $f(.,\{x_{k}\})$ is measurable for each $\{x_{k}\} \in \R^{(N+1)\times n}$.
\item[(b):] 
there exists $k(.) \in L^{1}$ such that  
\begin{equation}
\label{Filippov}
\{x_{k}\} \rightarrow f(t,\{x_{k}\};\{d(t-h_{k})\})
\end{equation}
is $k(t)$ Lipschitz continuous on 
$(y((t-h_{0}) \vee S), \ldots, y((t-h_{N})\vee S))+\bar \epsilon \B$, a.e.. 
%(We set $\bar x(t-h_{k}):=\bar x(S)$ for $t-h_{k} < S$.) 
\end{itemize} 
%In the case $\epsilon < \infty$, 
Assume also that 
%$$
 %t \rightarrow f(t,y(t-h_{0}), ldots, y(t-h_{N});e(t-h_{0}), ldots, e(t-h_{N}))%
%$$
%is integrable and 
\begin{multline*}
|y(S)-\xi| +
\int_{S}^{T} | \dot y(t) - f(t,\{y(t-h_{k})\});\{ d(t-h_{k})\})| dt
\\
\hspace{-0.5 in} \leq \;\bar\epsilon \times \left(\exp\{(N+1)\,\int_{S}^{T}k(t)\,dt\}\right)^{-1}\,. 
% \\
% \hspace{-2.5 in}\mbox{(The latter condition is automatically satisfied if $\bar \epsilon = + \infty$.)}
\end{multline*} 
(In the case $\bar \epsilon = + \infty$,  this last condition is automatically satisfied and the function \eqref{Filippov} is required to be $k(t)$ Lipschitz on $\R^{n}$.)
Then:
\ \\ 

\noindent
{\bf (A):} there exists a solution $x(.)\in W^{1,1}([S,T];\R^{n})$ to the equation
\begin{equation}
\label{dde}
\dot x(t) = f(t,\{x(t-h_{k})\});\{ d(t-h_{k})\}) \mbox{ a.e. }t \in [S,T]
\end{equation}
with initial state $x(S)=\xi$ and satisfying
\begin{eqnarray*}
&&||y(.)- x(.)||_{L^{\infty}(S,t)}\;  \leq \;  |y(S)-x(S)| + \int_S^t |\dot y(s) - \dot x(s)| \ ds 
\; \leq \;
\\
&&   \hspace{0.5 in} e^{(N+1)\int_S^t k(s) \ ds }\bigg(   |\xi - y(S)| \;+ \;\int_{S}^{t} | \dot y(s) - f(s,\{y(s-h_{k})\});\{ d(s-h_{k})\})| ds
 \bigg) \,,
\end{eqnarray*}
for all $t \in [S,T]$.
\ \\

\noindent
{\bf  (B):}  if $\bar \epsilon = +\infty$, \eqref{dde} has a unique  solution $x(.)$ in $W^{1,1}$ (for given $x(S)$, $u(.)$ and $d(.)$).

\end{thm}
\begin{proof}
(A) is proved by means of a straightforward adaptation of the proof  of the well-known related existence theorem for differential inclusions, based on `Picard iteration' (see, e.g., \cite[Proof of Thm. 2.4.3.]{Vinter}). To prove (B) note that, if the hypotheses are satisfied with $\bar \epsilon = +\infty$, then (\ref{dde}) has a solution, by part (A). If there are two solutions $x(.)$ and $y(.)$, we deduce from hypothesis (b) that the continuous function $\xi(t):= ||x(.)-y(.)||_{L^{\infty}([0,t])}$ satisfies
$$
|\xi(t)| \,\leq \, \int_{[S,t]}(N+1)\x k(s)|\xi(s)|ds \quad \mbox{ for } t \in [S,T]\,.
$$
We conclude from Gronwall's inequality that $\xi(.)\equiv 0$, whence $x(.)=y(.)$\,.
\end{proof}

\noindent 
%Proceeding with the proof of Thm. \ref{main}, w
We first validate the assertions  of Thm. \ref{main} when hypotheses (H1)-(H3) are supplemented by several additional hypotheses. We then show that the assertions remain true when the additional hypotheses are removed. The additional hypotheses 
(which make reference to the initial state $\bar x(S)$ of the process $(\bar x(.), \bar u(.))$ under consideration) 
are as follows: 
%for given $\epsilon >0$ and $\alpha \geq 0$,
\begin{itemize}
%\item[(A1):] 
%$f(t,\{x_{k} \}, \{u_{k} \};\{d_{k}\})$ does not depend on the initial data $\{d_{k}\}$ and 
%$L(.\,.) \equiv 0$
 %and  
% $\Lambda(.,.)\equiv 0$.
 %\item[]
%
%(When (A0)$'$ is satisfied, we write $f(t,\{x_{k} \}, \{u_{k} \})$ in place of $f(t,\{x_{k} \}, \{u_{k} \};\{d_{k}\})$.)
\item[(A0):]
$(f,L)(t,\{x_{k} \}, \{u_{k} \};\{d_{k}\})$ does not depend on the initial data $\{d_{k}\}$ and $\Lambda(.,.)\equiv 0$.
\item[]
(When (A0) is satisfied, we write $f(t,\{x_{k} \}, \{u_{k} \})$ in place of $f(t,\{x_{k} \}, \{u_{k} \};\{d_{k}\})$.)
%initial data integrand 
\item[(A1):] 
There exists a bounded, ${\cal L}\times {\cal B}$ measurable function $\tilde L(.,.):[S,T]\times \R^{m}\to \R$ such that
$ 
L(t,\{x_{k}\},\{ u_k \})= \tilde L(t,u_{0})
$.
%\item[(A2):] $D(s)$ and $U(t)$ are finite sets for all $s \in [S-h,S]$ and $t \in [S,T]$.

\item[(A2):] There exist integrable functions $c_{0}(.):[S,T]\rightarrow \R$ and  $k_{0}(.):[S,T]\rightarrow \R$ such that, for all selectors  $u(.)\in U(.)$ and a.e. $t \in [S,T]$, the mapping
\vspace{0.05 in}

\noindent
(i): $\{x_{k}\} \rightarrow f(t,\{x_{k}\},\{u(t-h_{k}) \})$
is $c_{0}(t)$ bounded  on  $\R^{n \times (N+1)}$,
\vspace{0.05 in}

\noindent
(ii): $\{x_{k}\} \rightarrow f(t,\{x_{k}\},\{u(t-h_{k}))$ is $k_{0}(t)$-Lipschitz continuous on $\R^{n \times (N+1)}$.
%(iii):
% $t \rightarrow \Lambda(t,d(t))$ is $c_{0}(t)$  bounded for a.e. $t \in [S-h, S]$.
 %\todo[inline]{I took $c_0$ defined in $[S-h,T]$ because it is present both in (i) and (iii)}
 \item[(A3):]  $C= \R^{n} \times \R^{n}$\,.
 \item[(A4):] There exist continuously differentiable functions  $l_{0}(.): \R^{n} \rightarrow \R$ and $l_{1}(.): \R^{n} \rightarrow \R$ and $\alpha \geq 0$, such that
\begin{equation}
\label{sep}
g(x_{0},x_{1})= l_{0}(x_{0}) 
+ \alpha |x_{0}- \bar x(S) |
+\, l_{1}(x_{1})\quad \mbox{for all } (x_{0},x_{1}) \in \R^{n}\times \R^{n}\,.
\end{equation}
%$l_0(.)$ is Lipschitz continuous on $\R^n$ and $l_1(.)$ is continuously differentiable.
\end{itemize}
%\todo[inline]{I needed to change (A5) to include Lipschitz continuous initial costs. This is because when invoking the basic case in later steps we have the term $|x(S)-x^*(S)|$ deriving from Ekeland. This could have been fixed in several ways for example keeping $l_0$ continuously differentiable and adding a term $\alpha|x(S)-\hat x|$. I found however that this way was the most economical in terms of corrections. I simply used a nonsmooth mean value theorem. }
%functons satisfying condition 
%Take functions $p_{k}(.)$, $k=0,\ldots, N$, as in Thm. \ref{main}.  Define the absolutely continuous function $p(.):[0,T] \rightarrow \R^{n}$
%according to (\ref{p_eqn}), i.e.
%%\hspace{0.1 in}
%
%
%\hspace{0.5 in} 
%$p(t):= \sum_{k=0}^{N} p_{k}(t) \mbox{ for } t \in [0,T].
%$
%\vspace{0.1 in}
%
%%We deduce from the properties of the $p_{i}(.)$'s that $p(.)$ is an absolutely continuous function. 
%\noindent
%From condition (b) in Thm. (\ref{main}) and the Sum Rule for subdifferentials \cite{ClarkeLed}, we have:
%\vspace{0.1 in}
%
%$
%-\dot p(t)\,\in\, \chi_{[0, \bar T-h_{i}]} \tilde \partial^{L}_{x_{i}}\;
%p(t+h_{k})\cdot$
%$f(t+h_{k}, x_{0},\ldots,x_{N}, \bar w(t+h_{k}))$
%\vspace{0.1 in}
%
%\hspace{2.0 in} $((x_{0},\ldots ,x_{N})= (\bar x(t-h_{0}+h_{k}),\ldots, \bar x(t-h_N+h_{k})) )
%$ a.e.
%\vspace{0.1 in}
%
%\noindent
%Summing both sides of the inclusion over $k$ yields the condition (b$^{*}$) of the corollary.
\vspace{0.1 in}

\noindent
{\bf Step 1:} {\it We confirm the  assertions of Thm. \ref{main}   (with $\lambda=1$) under (H1)-(H3), (A0)-(A4).}
\ \\

\noindent
%Take a bounded, ${\cal L} \times {\cal B}$ function $\tilde L(.,.):[S,T]\times \R^{m} \rightarrow \R$, continuously differentiable functions $l_{j}():\R^{n}\rightarrow \R$, $j=1,2$, and $\alpha >0$. 
Assume that $(\bar x(.), \bar u(.))$ is an $L^\infty$ local minimizer  for problem ($P$), under hypotheses (H0)-(H3) and (A1)-(A4). Then, for some $\epsilon>0$, $(\bar x(.), \bar u(.), \bar d(.))$ is a minimizer for
$$
(P^{\epsilon})
\left\{
\begin{array}{l}
\hspace{0.2 in}\mbox{Minimize } \; 
J(x(.),u(.)):= l_{0}(x(S))+\alpha|x(S)-\bar x(S)|+ l_{1}(x(T))+ +\int_{[S,T]}\tilde L(t,u(t))\,dt ,
\\
\hspace{0.2 in} \mbox{subject to}
\\
\hspace{0.2 in} \dot x(t)= f(t,\{ x(t-h_{k})\}, \{ u(t-h_{k}\}),
% t\in[S,T],
%( t,
%)
%\\
%\\
%\hspace{5.0 in } \mbox{ a.e. }t\in [S,T],
\\
%\hspace{0.2 in} (x(t),u(t))= d(t) \mbox{ a.e. }t\in[S,T],
%\\
\hspace{0.2 in}u(t)\in U(t)\mbox{ a.e. }t\in[S,T],
\\
\hspace{0.2 in}||x(.)-\bar x(.) ||_{L^{\infty}} \leq \epsilon\,.
%\\
%\hspace{0.2 in}(x(S),x(T))\in C\;.
\end{array}
\right.
$$
% Notice that we have added the constraint $||x(.)-\bar x(.)||_{L^\infty} \leq \epsilon$, for $\epsilon>0$ suitably small , to ensure that the process is a minimizer, not merely an $L^{\infty}$ local minimizer.
For each positive integer $i$, consider the related problem:

$$
(P_{i})
\left\{
\begin{array}{l}
\hspace{0.2 in}\mbox{Minimize } \; 
J_{i}(x(.),\{y_{k}(.)\},u(.)):= \, l_{0}(x(S))+ \alpha|x(S)-\bar x(S)|+l_{1}(x(T))
\\
\hspace{0.7 in}+\int_{[S,T]}\tilde L(t,u(t))\,dt \,+
i \times \left(\sum_{k=0}^{N}\int_{[(S+h_{k})\wedge T,T]} k_{0}(t)|y_{k}(t)-x(t-h_{k})|^{2} dt  \right)
\\
\hspace{0.2 in} \mbox{over } \{y_{k}(.)\in L^{1}_{k_0(.)}([(S+h_{k}) \wedge T, T];\R^{n})\}
%, \, k=0,\ldots,N,\, x(.)\in W^{1,1} 
\mbox{ and selectors } u(.) \mbox{ of } U(.)\, \mbox{ s.t. } 
\\
\hspace{0.2 in} \dot x(t)= f(t, \{y_{k}(t)\},\{u(t-h_{k}\}),\mbox{ a.e. }
% t\in[S,T],
%( t,
%)
%\\
%\\
%\hspace{5.0 in } \mbox{ a.e. }t\in [S,T],
\\
%\hspace{0.2 in} (x(t),u(t))= d(t) \mbox{ a.e. }t\in[S,T],
%\\
\hspace{0.2 in}u(t)\in U(t)\mbox{ a.e. }t\in[S,T],
\\
\hspace{0.2 in}||x(.)-\bar x(.) ||_{L^{\infty}} \leq \epsilon.
%\hspace{0.2 in}(x(S),x(T))\in C\;.
\end{array}
\right.
$$
Here, $k_{0}(.)$ is the integrable bound of hypothesis (A2) and $L^1_{k_0(.)}([a,b])$ denotes measurable functions $\phi:[a,b]\to\R^n$ such that $k_0(t)\phi(t)\in L^1$.
%We may assume $k(t)\geq 1$ a.e. 
Observe that the cost can be infinite because the $y_{k}(.)$'s are allowed to be $L^{1}$ functions and the cost involves $L^{2}$ norms. Write the infimum costs  of  $(P_{i})$ and $(P^{\epsilon})$ as $\inf(P_{i})$ and  $\inf(P^{\epsilon})$, respectively.
\begin{lem}
\label{lemma_4_2}
$$
\underset{i \rightarrow \infty}{\lim} \; \inf(P_{i})\;=\;  \inf(P^{\epsilon})\,.
$$
\end{lem}
\begin{proof} We deduce from the special structure of the $(P_{i})$'s and $(P^{\epsilon})$  that 
\begin{equation}
\label{weak}
- \infty \;< \; \inf(P_{i})\; \leq \; \inf(P_{j})\; \leq \; \inf(P^{\epsilon}) \quad \mbox{ for any index values $i < j$}\,.
\end{equation}
Fix $i>0$ and take any feasible process $(x(.), \{y_{k}(.)\},u(.))$ for $(P_{i})$ such that  
$J_{i}(x(.), \{y_{k}(.)\},u(.))$ $< \infty$. (Such a feasible process, namely $(\bar x(.), \{\bar x(. - h_{k})\},\bar u(.), \bar d(. ))$, exists.)
\ \\

\noindent
By Filippov's Thm., applied with  $y(.)=x(.)$ and initial state $\xi = x(S)$, there exists a feasible process $(x_{i}(.), u(.))$ for ($P^\epsilon$) (with the same $u(.)$ and $d(.)$) such that $x_{i}(S)=x(S)$ and
$$
||x_{i}(.)- x(.)||_{L^{\infty}}\,\leq \, K \times \sum_{k=0}^{N}
(  \int_{[(S+h_{k})\wedge T, T]}k_{0}(t) |y_{k}(t)-x(t-h_{k})|  dt )\,.
$$
 ($K$ is a number that does not depend on our choice of $(x(.),\{y_k(.)\}, u(.))$.)
 %Let $k_{l_{1}}$ is a Lipschitz constant for $l_{1}(.)$. We can show and $\bar N = \max\{k \,|\,  S+h_{k} < T \}$. Since  $k(t)\geq 1$, we 
 With the help of  H\"older's inequality, we can show that
 %$$
 %(\int_{[a,b]} k(t)|r(t)|dt )^{2}\leq  (\int_{[a,b]} k(t)dt)^{2} \times  \int_{[a,b]} k(t)|r(t)|^{2}dt \leq
 %|b-a|^{2} \leq \int_{[a,b]} k(t)|r(t)|^{2} dt.
 %$$
 %It follows
 \begin{eqnarray*}
 &&
J_{i}(x(.),\{y_{k}(.)\},u(.)) \,-\, J_{i}(\{x_i(.),x_i(.-h_k)\} ,u(.))
\\
&&
\hspace{0.1 in} \geq \quad\sum_{k=0}^{N}\;\int_{[(S+h_{k})\wedge T,T]}  k_{0}(t)\left(i \times |y_k(t)-x(t-h_{k})|^{2}  - K k_{l_{1}} |y_{k}(t)-x(t-h_{k})| \right)dt \,,
\\
&&
\hspace{0.1 in} \geq
\quad
\sum_{k=0}^{ N}\; \left(( \int_{[S,T]} k_{0}(t)dt)^{-1} \times i 
\times z_{k}^{2} -Kk_{l_{1}} z_{k}
\right),\,
(z_{k}= \int_{[(S+h_{k})\wedge T,T]} k_{0}(t)|y_k(t)-x(t-h_{k})|dt )\,
\\
&&
\hspace{0.1 in}
\hspace{0.1 in} \geq
\quad
\sum_{k=0}^{ N}\;   \underset{z\in \R}{\min}  \left(( \int_{[S,T]} k_{0}(t)dt)^{-1} \times i 
\times z^{2} -Kk_{l_{1}} z
\right)
\; \geq\; - \gamma \, \times i^{-1}\,,
%k}\int_{[(S+h_{k})\wedge T,T]} ( k(t)(i \times |y^{i}(t)-x(t-h_{k})|^{2}  - K k_{l_{1}} |y_{i}(t)-x(t-h_{i})| )dt \,,
\end{eqnarray*} 
in which $\gamma := \frac{1}{4} \; 
(N+1) \times (Kk_{l_{1}})^{2} \times ||k_{0}(.)||_{L^{1}} $ and $k_{l_{1}}$ is a Lipschitz constant for  $l_{1}(.)$. Since $(x(.), \{y_{k}(.)\}, u(.))$ was chosen arbitrarity, we have shown that
$$
\inf\, (P_{i}) \, \geq\, \inf (P^{\epsilon}) - \gamma\times i^{-1}\,.
$$
Combining this inequality with (\ref{weak}) give the desired relation.

\end{proof}

\noindent
Now write problem $(P_{i})$ as
\ \\

$
\mbox{Minimize } \; \{
J_{i}(x(.),\{y_{k}(.)\},u(.))\,|\,  (x(.),\{y_{k}(.)\},u(.)) \in {\cal A}_{\epsilon} 
\} \,,
$
\ \\

\noindent
in which
\ \\

$ {\cal A}_{\epsilon} := \{(x(.),\{y_{k}\},u(.))\,|\, x(.) \in W^{1,1}, \hspace{0.1 in}y_{k}(.) \in L^{1}_{k_0(.)}([(S+ h_{k})\wedge T, T])
\mbox{ for } k=0,\ldots,N, $

\hspace{1.5 in}$u(.) \mbox{  is a selector of } U(.),\,\dot x= f \mbox{ a.e., } \, ||x(.)- \bar x(.)||_{L^{\infty}} \leq \epsilon  \}$. 
\vspace{0.1 in}

\noindent
Equip ${\cal A}_{\epsilon}$ with the metric
\begin{multline*}
d_{{\cal E}}((x'(.),\{y_k'(.)\},  u'(.)) ,(x(.),\{y_k(.)\}, u(.)) ) \;=\; 
\\
|x'(S)- x(S)|+\sum_{k=0}^{N} \int_{[(S+h_{k})\wedge T,T]}k_0(t)|y_{k}'(t)-y_{k}(t)|\,dt \;+\;
\mbox{meas } \{t \in [S,T]\,|\, u'(t) \not= u(t) \}
\end{multline*}
It can be shown that, w.r.t. this metric, $ {\cal A}_{\epsilon}$ is  complete and $ J_{i}(.,.,.)$
is lower semicontinuous on $ {\cal A}_{\epsilon}$. (The lower semicontinuity can be deduced 
%from the fact the  `unbounded' term in the cost is convex.
%is a consequence of 
from Fatou's Lemma, as demonstrated in \cite[page 245]{Vinter}.) By Lemma  \ref{lemma_4_2}, there exists $\gamma_{i} \downarrow 0$ such that, for each $i$, $(\bar x(.), y_{0}(t)\equiv \bar x(t-h_0),\ldots, y_{N}(t)\equiv \bar x(t-h_N),\bar u(.),\bar d(.))$ is a $\gamma_{i}$ minimizer for $(P_{i})$\,. According to Ekeland's Theorem, there exists, for each $i$, 
$( x_{i}(.), y_{0}^{i}(.), \ldots, y_{N}^{i}(.),u_{i}(.)) \in {\cal A}_{\epsilon}$ which is a minimizer for the optimization problem:
%($\tilde P_{i})$:
\begin{multline*}
(\tilde P_{i})\quad \mbox{Minimize } \{J_{i}((x(.),\{y_k(.)\},  u(.))
\\
 + \gamma_{i}^{\frac{1}{2}}
d_{{\cal E}}((x(.),\{y_k(.)\},u(.)),( x_{i}(.), \{y_{k}^{i}(.)\},u_{i}(.)) ) \,|\,\\
 (x(.),\{y_k(.)\},  u(.)) \in {\cal A}_{\epsilon} \}\,,
\end{multline*}
\vspace{-0.4 in}

and
\begin{equation}
\label{eq:ekeland}
d_{{\cal E}}((x_{i}(.),\{y_{k}^i(.)\},  u_{i}(.)), \,(\bar x(.),\{\bar x(.-h_k)\}, \bar u(.)) ) \,\leq \, \gamma_{i}^{\frac{1}{2}}\,.
\end{equation}
It can be deduced from (\ref{eq:ekeland}), with the help of Thm. \ref{Filippov},  that $||x_{i}(.)- \bar x(t)||_{L^{\infty}}\rightarrow {0}$, as $i\rightarrow \infty$. We have then, for $i$ sufficiently large,
$$
||x_{i}(.)-\bar x(.)||_{L^{\infty}} \,\leq \epsilon /2\,.
$$
Define $m_{i}(.,.):[S-h,S]\times \R^{n+m} \rightarrow \R$ 
%and $m^{i}_{1}(.,.):[S,T]\times \R^{m} \rightarrow \R$ 
as
\begin{equation}
\label{errorx}
m_{i}(t,u):= 
\left\{
\begin{array}{ll}
0 & \mbox{if } u=u_{i}(t)\; .
\\
1  & \mbox{if } u \not= u_{i}(t)
\end{array}
\right.\;.
\end{equation}
The cost function for $(\tilde P_{i})$ can be written
\begin{eqnarray*}
&&
\tilde J_{i} ((x(.),\{y_k(.)\},  u(.)):=
%\, \int_{[S-h,S]}( \Lambda (t,d(t)) + \gamma_{i}^{\frac{1}{2}} m_{0}^{i}(t,d(t)))dt \,+\,
%\\
%&&
%\hspace{1.5 in}
i \times \left(\sum_{k=0}^{N}\int_{[(S+h_{k})\wedge T,T]} k_{0}(t)|y_{k}(t)-x(t-h_{k})|^{2} dt \right)
\\
&&  \hspace{1.8 in} + \, 
\int_{[S,T]}( \tilde L(t,u(t))+ \gamma_{i}^{\frac{1}{2}} m_{i}(t,u(t)))dt 
\\
&&
\hspace{2.0 in}  +\, l_{0}(x(S))+ \alpha |x(S)- \bar x(S)| +  \gamma_{i}^{\frac{1}{2}}|x(S)-x_{i}(S)|+ l_{1}(x(T))\,.
\end{eqnarray*}
\begin{lem}
\label{lemma_4_3}
Take a minimizer $( x^{*}(.), \{y_{k}^{*}(.)\},u^{*}(.))$ for $(\tilde P_{i})$ such that
$||x^{*}(.)-\bar x(.)||_{L^{\infty}} \leq \epsilon /2$.  Define
the functions  $p_{k}(.): [S-h_{k},T] \rightarrow \R^{n}$, $k=0,\ldots, N $ according to 
$$
\left\{
\begin{array}{l}
-\dot p_{k}(t-h_{k}) =  2 \times i \times k_{0}(t)( y^{*}_{k}(t) -x^{*}(t-h_{k})) \mbox{ a.e. } t \in [(S+h_{k})\wedge T,T]
\\
p_{k}(t) =0 \mbox{ for } t \in [(T-h_{k}) \vee S,T] \mbox{ if } k > 0
\\
 \dot p_{k}(t-h_{k}) =0 \mbox{ a.e. } t \in [S,(S+h_k)\wedge T]
 \\
  -p_{0}(T) = \nabla l_{1}(x^*(T))\,,
\end{array}
\right.
$$ 
and let $p(.):[S,T] \rightarrow \R^{n}$ be the $W^{1,1}$ function 
$
p(t):= \sum_{k=0}^{N}\, p_{k}(t)\quad \mbox{for } t \in [S,T]\,.
$
Then
\begin{itemize}
%\item[(a):] $\lambda+\sum_{k}\|p_{k}(.)\|_{L^{\infty}}\,\not=\, 0$,
\item[(b)$\,'$:] $(\{-\dot p_{k}(t-h_{k})\} )
\in 
\mbox{co}\, \partial_{\{ x_{k} \}}\,(p\cdot f - L)
 (t,\{ y^{*}_{k}(t)\},\{u^{*}(t-h_{k})\})
 \mbox{ a.e. } t \in [S,T],$
\item[(c)$\,'$:] For any selector $u(.)$ of $U(.)$,
\begin{eqnarray}
\label{integral}
\hspace{-0.5 in}
&&
 \int_{[S,T]}(p\cdot f-\lambda L)(t,\{y_{k}^{*}(t)\},\{u(t-h_{k})\})dt 
%&&\hspace{0.8 in}
\leq \, \int_{[S,T]}(p\cdot f-\lambda L)(t,\{y_{k}^{*}(t)\},\{ \bar u(t-h_{k})\})dt,  \hspace{0.6 in}
\end{eqnarray} 
 %$
%\tilde {\cal H}_{\lambda=1}(t,u^{*}(t);  y^{*}_{0}(.),\ldots, y^{*}_{N}(.),u^{*}(.), d^{*}(.), p(.))\geq 
%$
%
%
%$\hspace{0.5 in}  \underset{u \in U(t)}{\max}
%\tilde {\cal H}_{\lambda=1}(t,u;  y^{*}_{0}(.),\ldots, y^{*}_{N}(.), u^{*}(.), d^{*}(.), p(.)) - \gamma_{i}^{\frac{1}{2}} 
% \mbox{ a.e. }t \in [S,T]\,$ 
%\item[(d)$\,'$:] 
%
%$
%\tilde {\cal H}^{0}_{\lambda=1}(t,d^{*}(t); y^{*}_{0}(.),\ldots, y^{*}_{N}(.), u^{*}(.), d^{*}(.), p(.)) \, \geq
%$
%
%
%$ \hspace{0.5 in} \underset{d \in D(t)}{\max} 
%\tilde {\cal H}^{0}_{\lambda=1}(t,d; y^{*}_{0}(.),\ldots, y^{*}_{N}(.), u^{*}(.), d^{*}(.), p(.)) -\gamma_{i}^{\frac{1}{2}}\,  
 %\mbox{ a.e. } \in [S-h,S]$.
%\vspace{-0.1 in}
%
%\noindent
%\hspace{-0.4 in}and

  \item[(\,d)$\,'$:] $p(S) \in \nabla l_{0}(x^{*}(S)) + (\gamma_{i}^{\frac{1}{2}}+\alpha)\B, \ \ \ -p(T)= \nabla l_{1}(x^{*}(T))
 %g(\bar x(S),\bar x(T))+N_C(\bar x(S),\bar x(T))
   $.
\end{itemize}
%(In conditions (c)$\,'$ and (d)$\,'$, $\tilde {\cal H}_{\lambda}(t,u; y^{*}_{0},\ldots, y^{*}_{N}, u^{*}(.), d^{*}(.), p(.))$  and  
%
%\noindent
%$\tilde {\cal H}^{0}_{\lambda}(t,d; y^{*}_{0},\ldots, y^{*}_{N}, u^{*}(.), d^{*}(.), p(.))$ are defined as
%  ${\cal H}_{\lambda}(t,u; x^{*}(.), u^{*}(.), d^{*}(.), p(.))$ 
% and 
% 
% \noindent ${\cal H}^{0}_{\lambda}(t,d; x^{*}(.), u^{*}(.), d^{*}(.), p(.))$  in 
 % (\ref{Hamiltonian}) and (\ref{Hamiltonian_0}), except $y^{*}_{k}(t)$ replaces $x^{*}(t-h_{k})$, $k=0,\ldots,N$.)
\end{lem} 
\begin{proof}
Take $\delta >0$,  $\xi \in \R^{n}$, functions  $y_{k}(.)\in L^{2}_{k_{0}(.)}([(S+h_{k}) \wedge T, T])$, $k=0,\ldots,N$, and a selector $u(.)$ of $U(.)$. Let $x(.)$ be the corresponding state trajectory, with initial condition $x(S)= x^{*}(S)+\delta \xi$. Assume that $x(t) \in  \bar x(t)+ \epsilon /2$ for all $t \in [S,T]$.
\ \\

\noindent
By `optimality' of $( x^{*}(.), \{y_{k}^{*}(.)\},u^{*}(.))$ and since $x(.)$ and $x^{*}(.)$
satisfy the dynamic constraint, 
\begin{eqnarray}
\nonumber
\hspace{-0.3  in}
&&
\delta^{-1}\,\left(
 \sum_{k=0}^{N}\int_{[(S+h_{k})\wedge T,T]} i \times k_{0}(t)\left( |y_{k}(t)-x(t-h_{k})|^{2} -
|y_{k}^{*}(t)-x^{*}(t-h_{k})|^{2}
\right) \, dt\right)
\\
\nonumber
\hspace{-0.3  in}
&&
\nonumber
+ \delta^{-1}\, \Big( 
l_{0}(x^{*}(S) + \delta \xi) - l_{0}(x^{*}(S)) + \gamma_{i}^{\frac{1}{2}} \delta|\xi|+ l_{1}(x(T)) - l_{1}(x^{*}(T))
\\
\nonumber
\hspace{-0.3  in}
&&\hspace{2.5 in}+ \alpha( |x^{*}(S) +\delta \xi -\bar x(S)| -| x^{*}(S)  -\bar x(S)| )\Big) 
\\
\label{estimate}
%&&
%\hspace{0.2 in}
%+ \delta^{-1}\left( 
%\int_{[S-h,S]}(\Lambda(t,d(t))-\Lambda(t,d^*(t)))+\gamma_i^{\frac12}(m_{0}^i(t,d(t))-m_{0}^i(t,d^*(t))) \,dt\right)
%\\
&&
 \label{estimate}
\hspace{-0.3 in}
+ \delta^{-1}\,\left( 
  \int_{[S,T]}(\tilde L(t,u(t))-\tilde L(t,u^*(t)))+\gamma_i^{\frac12}(m_i(t,u(t))- m_i(t,u^{*}(t)))\,dt
 \right)
\\
\hspace{-0.3  in}
&&
\nonumber
\hspace{-0.3 in}
 + \delta^{-1}\left(
 \int_{[S,T]} p(t)\cdot \Big( \dot x(t)  -\dot x^{*}(t) - [f(t,\{y_{k}(t)\},\{u(t-h_{k})\}) \,\right.
  -\; f(t,\{y^{*}_{k}(t)\},\{u^{*}(t-h_{k})\}
)]\,  \Big)\,dt\Bigg)
\;
\geq \; 0\,.
\end{eqnarray}
Performing an integration by parts yields the identity
\begin{eqnarray*}
&& \int_{[S,T]} p(t)\cdot( \dot x(t)  -\dot x^{*}(t))dt = - \int_{[S,T]} \dot p(t)\cdot (x(t)-x^{*}(t))dt 
\\
&&
\hspace{1.5 in}
+p(T)\cdot (x(T)-x^{*}(T))
- p(S)\cdot(x(S)-x^{*}(S))\,.
 \end{eqnarray*}
Substituting this expression into (\ref{estimate}), employing the  expansion
\begin{multline*}
 |y_{k}(t)-x(t-h_{k})|^{2} = |y_{k}(t)-x^{*}(t-h_{k})|^{2}\;+\; |x(t-h_{k})-x^{*}(t-h_{k})|^{2}
\\
-2 (y^{*}_{k}(t)-  x^{*}(t-h_{k}))\cdot (x(t-h_{k})-x^{*}(t-h_{k}))
%\\
%&&
%\hspace{0.4 in}
-2(  y_{k}(t) -y^{*}_{k}(t))\cdot(   x(t-h_{k})-x^{*}(t-h_{k}))
%(y_{k}(t)-  y_{k}^{*}(t) ) \cdot (x(t-h_{k})-x^{*}(t-h_{k}))
 \end{multline*}
 and using the estimates
\begin{eqnarray*}
%&& 
%l_{0}(x^{*}(S)+ \delta \xi)-l_{0}(x^{*}(S)) \leq  \zeta\cdot \delta\xi, \ \ \zeta\in \de l_0(z), z\in[x^*(S)+\delta\xi,x^*(S)]
%\\
%&&
%\hspace{-0.2 in}\mbox{and }
%\\
&& 
l_{1}(x(T))-l_{1}(x^{*}(T)) - \nabla l_{1}(x^{*}(T))\cdot (x(T)-x^{*}(T)) 
\;\leq \;\theta(|x(T)- x^{*}(T)|)
\\
&&
%l_{0}(x(S))-l_{0}(x^{*}(S)) + \alpha |x(S)- x^{*}(S)|\;\leq\; 
% \delta \nabla l_{1}(x^{*}(S))\cdot\xi + \alpha\delta |\xi|
%+ \theta(\delta |\xi|)
l_{0}(x(S))-l_{0}(x^{*}(S)) -
 \delta \nabla l_{0}(x^{*}(S))\cdot\xi \;\leq\; 
\theta(\delta |\xi|)
\\
&&
|x(S)-\bar x(S)| -|x^{*}(S)-\bar x(S)| \,\leq\ \delta |\xi| , 
%
%|y_{k}(t)-x(t-h_{k})|^{2} = |y_{k}(t)-x^{*}(t-h_{k})|^{2}
%\\
%&&
%\hspace{0.2 in}-2 (y^{*}_{k}(t)-  x^{*}(t-h_{k}))\cdot (x(t-h_{k})-x^{*}(t-h_{k}))
%\\
%&&
%\hspace{0.4 in}
 %-2 |x(t-h_{k})-x^{*}(t-h_{k})|^{2}
%(y_{k}(t)-  y_{k}^{*}(t) ) \cdot (x(t-h_{k})-x^{*}(t-h_{k}))
 \end{eqnarray*}
 (for some function $\theta(.):[0,\infty) \rightarrow [0, \infty)$ such that $\lim_{\alpha \downarrow0}\alpha^{-1 }\theta(\alpha)=0$),
 %, and the nonsmooth mean value inequality
 %\todo[inline]{this inequality maybe needs a reference}
% $$
% l_{0}(x^{*}(S)+ \delta \xi)-l_{0}(x^{*}(S)) \leq  \zeta\cdot \delta\xi, \ \ \zeta\in \de l_0(z), z\in[x^*(S)+\delta\xi,x^*(S)],
% $$
 %, and noting the relations defining $p(.)$, 
 we arrive at
\begin{eqnarray}
\nonumber
&& 
\delta^{-1}\left( \sum_{k=0}^{N}\int_{[S,T]} i \times k_{0}(t)\left( |y_{k}(t)-x^{*}(t-h_{k})|^{2} -
|y_{k}^{*}(t)-x^{*}(t-h_{k})|^{2}  \right) \chi_{[S,T]}(t-h_{k}) dt \right)
\\
\nonumber
&&
\hspace{0.1 in}
+ \;
(\nabla l_0(x^*(S))-p(S))\cdot \xi   + (\gamma_{i}^{\frac{1}{2}} + \alpha)| \xi|
%{(l_{0}(x(S)) - l_{0}(x^{*}(S)) - l_{1}(x(T)) - l_{0}(x^{*}(T))]
%\right)  
\\
%&&
%\hspace{0.4 in}
%\nonumber
%+ \delta^{-1}\left( 
%\int_{[S-h,S]}(\Lambda(t,d(t))-\Lambda(t,d^*(t)))+\gamma_i^{\frac12}(m_{0}^i(t,d(t))-m_{0}^i(t,d^*(t))) \,dt\right)
%\\
&&
\nonumber
\hspace{0.2 in}
+ \delta^{-1}\left( 
  \int_{[S,T]}(\tilde L(t,u(t))-\tilde L(t,u^*(t)))+\gamma_i^{\frac12}(m_i(t,u(t))- m_i(t,u^{*}(t)))\,dt
 \right)
 \\
&&
\hspace{0.3 in}
\nonumber
 +\; \delta^{-1}
 \int_{[S,T]} -p(t)\cdot[f(t,\{y_{k}(t)\},\{u(t-h_{k})\}) \,
 -\, f(t,\{y^{*}_{k}(t)\},\{u^{*}(t-h_{k})\})\} 
)]\, dt
\\
&&
\label{estimate1}
\hspace{0.4 in}+\; E_{1}(\{y_{k}(.)\}, x(.),\{y^{*}_{k}(.)\}, x^{*}(.) ) + E_{2}(\{y_{k}(.)\}, x(.),\{y^{*}_{k}(.)\}, x^{*}(.) ) 
\geq 0\,. 
\end{eqnarray}
in which the first `error' term $E_{1}(.\,.)$ is 
$$
\begin{array}{l}
 E_{1}(\{y_{k}(.)\}, x(.),\{y^{*}_{k}(.)\}, x^{*}(.) ) :=
\\
\displaystyle{\delta^{-1} \sum_{k=0}^{N} \int_{[S,T]}}
 \Big( -2ik_{0}(t)
( y^{*}_{k}(t)- x^{*}(t-h_{k})) - \dot p_{k}(t-h_{k})
\Big) \cdot
\Big(x(t-h_k)-x^{*}(t-h_k)\Big)\vspace{-2mm}
\\
 \hspace{5in}\chi_{[S,T]}(t-h_{k})dt
\\
\hspace{0.8 in}
+\delta^{-1}\left(  \nabla l_{1}(x^{*}(T)) +p(T) \right) \cdot (x(T)- x^{*}(T))\,,
\end{array}
$$
and the second `error' term $E_{2}(.)$ is some function that satisfies 
\begin{multline} \label{error}
 E_{2}(\{y_{k}(.)\}, x(.),\{y^{*}_{k}(.)\}, x^{*}(.) ) \,\leq \, \delta^{-1}\Big(\theta(\delta |\xi|) +\theta(|x(T)-x^*(T)|)\Big) + \delta^{-1}  K(i) \times
\\
 \left(\sum_{k=0}^{N}\int_{[(S+h_{k})\wedge T, T]} k_{0}(t) |y_{k}(t)-y_{k}^{*}(t)|dt 
\times ||x(.)-x^{*}(.)||_{L^{\infty}} + \int_{[S,T]}k_0(t)|x(t)-x^{*}(t)|^{2}dt  \right).
\end{multline}
($K(i)$ is some number that depends on $i$, but not on the choice of   $\xi$, $u(.)$ and $\{y_{k}(.)\}$.)
\ \\

\noindent
Note that $
E_{1}(.\,.) \equiv 0
$,
because of the  defining relations of the $p_{k}(.)$'s.
\ \\

\noindent
We now confirm the assertions of the lemma by examining inequality (\ref{estimate1}), for various choices of  $\delta>0$, $\xi$, the $y_{k}(.)$'s, $u(.)$.
\ \\

\noindent
{\it Confirmation of (d)$\,'$:}  
%Let $y_{k}(.)=  y^{*}_{k}(.)$ for $k=0,\ldots,N$, $u(.)=u^{*}(.)$ and $d(.)=d^{*}(.)$. 
Notice that $-p(T)= \nabla l_{1}(x^{*}(T))$, by definition of the $p_{k}(.)$'s.  To verify the other transversality condition, take any $\xi \in \R^{n}$ and sequence $\delta_{j} \downarrow 0$. For each $j$ 
%Fix $j$, set $\delta= \delta_{j} $ and
 let $x_{j}(.)$ be the state trajectory corresponding to $y^{*}_{k}(.)$, $k=0,\ldots,N$, $u^{*}(.)$ and $d^{*}(.)$ and with initial value $x_{j}(S)= x^{*}(S)+ \delta_{j} \xi$. It is easy to show that
\begin{equation}
\label{conv1}
||x_{j}(.)-x^{*}(.)||_{L^{\infty}} \,\rightarrow\ 0, \mbox{ as }j \rightarrow \infty
\end{equation}
and there exist a number $C$ such that, for $j=1,2,\ldots$,
\begin{equation}
\label{conv2}
\delta_{j}^{-1}  ||x_{j}(.)-x^{*}(.)||_{L^{\infty}} \, \leq \, C.
\end{equation}
Now consider (\ref{estimate1}) when $\delta =\delta_{j}$, $y_{k}(.)= y^{*}_{k}(.)$ for $k=0,\ldots, N$, $u(.)=u^{*}(.)$, $d(.)= d^{*}(.)$ and $x(.)= x_{j}(.)$. From \eqref{conv1} and (\ref{conv2}) we see that 
$$
E_{2}(\{y^{*}_{k}(.)\}, x_{j}(.),\{y^{*}_{k}(.)\}, x^{*}(.) ) \rightarrow 0 \mbox{ as } j \rightarrow  \infty\,. 
$$
We may pass to the limit as $j \rightarrow \infty$, to obtain
$$
\left(\nabla l_{0}(x^{*}(S))-p(S) \right)\cdot \xi +(\gamma_{i}^{\frac{1}{2}}+ \alpha ) |\xi|\geq 0\,.
%, \mbox{ for some }\zeta\in \de l_0(x^*(S)).
 $$
Since this inequality is valid for every $\xi\in \R^{n}$, we conclude that $p(S)\in \nabla l_0(x^*(S))+ (\gamma_{i}^{\frac{1}{2}}+ \alpha) \B$. 
\ \\

\noindent
{\it Confirmation of (b)$\,'$:}  
Choose  $y_{k}(t) \in L^{2}_{k_0(.)}([(S+h_{k})\wedge T,T])$ such that  
\begin{equation}
\label{y constraint}
y_{k}(t)\in y^{*}_{k}(t)+ (1+k_{0}(t))^{-1}\B \mbox{ a.e. $t$, for $k=0,\ldots,N$} \,.
\end{equation}
Define
\begin{eqnarray*}
\hspace{-0.3 in}
&&
{\cal S}\,:=\, \{  \bar t\in (S,T)\,|\, \bar t \mbox{ is a Lebesgue point of } 
%\\
%&& \hspace{0.2 in} 
t \rightarrow k_{0}(t)(|y_{k}(t)-x^{*}(t-h_{k})|^{2}- |y^{*}_{k}(t)-x^{*}(t-h_{k})|^{2})
\\
 \hspace{-0.3 in}
 &&
\hspace{1.0 in}  \mbox{ and } t \rightarrow f(t,\{y_{k}(t)\}, \{u^{*}(t-h_{k})\})  -
f(t,\{y^{*}_{k}(t)\}, \{u^{*}(t-h_{k})\} )
%k_{0}(t)(|y_{k}(t)-x^{*}(t-h_{k})|^{2}- |y^{*}_{k}(t)-x^{*}(t-h_{k})|^{2}) 
\end{eqnarray*}
Take $\delta_{j} \downarrow 0$ and $\bar t\in \mathcal S$.  For each $j$, $t\in[(S+h_k)\wedge T,T]$, let
$$
y^{j}_{k}(t)\,=\,
\left\{
\begin{array}{ll}
y_{k}(t) & \mbox{if }t \in [\bar t, \bar t + \delta_{j}]
\\
y^{*}_{k}(t) & \mbox{if }t \notin [\bar t, \bar t + \delta_{j}]\,,
\end{array}
\right.
$$
for $k=0,\ldots,N$. Write $x_{j}(.)$ for the state trajectory corresponding to $\{y^j_{k}(.)\}$ and $u^{*}(.)$, with initial value $x^{*}(S)$. For each $j$, consider (\ref{estimate1}) with $\delta=\delta_{j}$, $\xi=0$, $\{y_{k}(.)\}=\{y^{j}_{k}(.)\}$ and $u(.)=u^{*}(.)$. Making use of (\ref{y constraint}),  we can show that (\ref{conv1}) and (\ref{conv2}) are satisfied, and $\int k_{0}(t)|y^{j}_{k}(t)- y^{*}(t)|dt \rightarrow 0 $ as $j \rightarrow \infty$. It follows that
$$
E_{2}(\{y^{j}_{k}(.)\}, x_{j}(.),\{y^{*}_{k}(.)\}, x^{*}(.) ) \rightarrow 0 \mbox{ as } j \rightarrow \infty  \,. 
$$ 
Since $\bar t \in {\cal S}$, we can pass to the limit in (\ref{estimate1}) as $j \rightarrow \infty$, to obtain
$$
\phi(\bar t, \{y_{k}(\bar t)\}) \geq 0\,,
$$
where
\begin{eqnarray*}
&&
\phi( t,\{y_{k}\}) \,:= \sum_{k} k_{0}( t) i \left(   
|y_{k}-x^{*}(t-h_{k})|^{2}- |y^{*}_{k}(t)-x^{*}(t-h_{k})|^{2}
 \right) \chi_{[S,T]}(t-h_{k})
 %\\
% \hspace{0.2 in} 
% k_{0}(t)(|y_{k}(t)-x^{*}(t-h_{k})|^{2}- |y^{*}_{k}(t)-x^{*}(t-h_{k})|^{2}) \mbox{ and }
\\
&&
\hspace{0.2 in} -p(t)\cdot\left( f(t,\{y_{k}\}, \{u^{*}(t-h_{k})\}) -
f(t,\{y^{*}_{k}(t)\}, \{u^{*}(t-h_{k})\}  \right).
  \end{eqnarray*}
Since $\cal S$ has full measure
$$
\int_{[S,T]} \phi(t, \{y_{k}(t)\})dt \,\geq 0\,. 
$$
But the $y_{k}(.)$'s are arbitrary measurable functions satisfying $y_{k}(t) \in y^{*}_{k}(t)+(1+k_{0}(t))^{-1} \B$ for a.e. $t \in [(S+h_{k})\wedge T, T]$. Invoking a measurable selection theorem, we can deduce that  
$$
 \phi(t,\{y^{*}_{k}(t)\}) = \min\{  \phi(t,\{y_{k}\}) \,|\,  y_{k} \in y^{*}_{k}(t)+ (1+k_{0}(t))\B, \mbox{ for } k=0,\ldots,N \} \,. 
 $$
But then
\begin{multline*}
-2ik_0(t)\Big( (y^{*}_{0}(t)-x^{*}(t-h_{0}))\chi_{[S,T]}(t-h_{0}) ,\dots,(y^{*}_{N}(t)-x^{*}(t-h_{N}))\chi_{[S,T]}(t-h_{N})\Big)  \,\in \\
\partial_{\{ x_{k}\} }
-p\cdot f(t,\{ y^{*}_{k}(t)\}, \{ u^{*}(t-h_{k})\})
\,.
\end{multline*}
From the defining relations for the $p_{k}(.)$'s we deduce
$$
\{ \dot  p_0(t-h_{k}) \} \in \partial_{\{ x_{k} \}}
(-p\cdot f(t,\{ y^{*}_{k}(t)\}, \{ u^{*}(t-h_{k})\}))
\mbox{ a.e } t \in [S,T].
$$
This relation implies (b)$'$.
\ \\

\noindent
{\it Confirmation of (c)$\,'$:} See Appendix.
%For an arbitrary positive integer $M$, define 
%\begin{equation}
%\label{u const}
%U_{M}(t):= U(t) \cap  \{u \in \R^{m}\,|\, |u- u^{*}(t)| \leq M\,\sum_{k=0}^{N} (1+ k_{0}(t+h_{k}))^{-1}\chi_{[S,T]}(t+h_{k})\}\,.
%\end{equation}
\end{proof}

\noindent
To complete Step 1, we observe that Lemma \ref{lemma_4_3},
%  $(x_{i}(.),\{y_k^i(.)\},u_{i}(.),d_{i}(.))$, 
provides perturbed versions of the desired conditions, with cost multiplier $\lambda =1$, in terms of costate functions that we now write $p^{i}_{k}(.)$, $k=0,\ldots,N$ to emphasize the $i$ dependence, under the stated hypotheses. Condition
 (\ref{eq:ekeland}) ensures that, along a subsequence,  $y^{i}_{k}(.)$ converges in $L^{1}$ and a.e. to $\bar x(t-h_k)$ on $[(S+h_{k})\wedge T,T]$ (for each $k$), $x_{i}(.)$  converges uniformly to $\bar x(.)$, $u_{i}(.)$ converges a.e. to $\bar u(.)$, as $i \rightarrow \infty$. The $p_{k}^{i}(.)$'s converge uniformly to  $W^{1,1}$ functions $p_{k}(.)$, $k=0,\ldots,N$ and their time derivatives converge weakly in $L^{1}$ to the time derivatives of the $p_{k}(.)$'s. A standard convergence analysis can be used to justify passage to the limit in conditions (b)$'$-(d)$'$ as $i \rightarrow \infty$, to recover the required necessary conditions (b)-(d). Notice, in particular, that the limiting  $p(.)$ function satisfies
  $$
 p(S) \in  \nabla l_{0}(\bar x (S)) +\alpha \B
  \,=\, \partial \tilde l(\bar x(S))\,,
   $$
 where $\tilde l(.)$ is the function $x_{0} \mapsto l(x_{0})+\alpha |x_{0}-\bar x(S)|$, which is the left transversality condition.
% For this special case the transversality condition $(e)$ takes the form
% $$
% p(S)\in \nabla l_0(\bar x(S))+\alpha\B, \ \ -p(T)=\nabla l_1(x(T)).
%  $$
   Of course (a) is automatically satisfied because $\lambda=1$.
\vspace{0.1 in}

\noindent
{\bf Step 2:} {\it We show that, if the assertions of Thm. \ref{main} are valid under (H1)-(H3) and (A0)-(A4), then they are also valid  under (H1)-(H3), (A0)-(A3).
}
\vspace{0.1 in}

% and when (A4) is replaced by the weaker hypothesis:
%\begin{itemize} 
%\item[(A4)$\,'$:] 
%There exist Lipschitz continous functions $l_{0}(.): \R^{n} \rightarrow \R$ and  $l_{1}(.): \R^{n} \rightarrow \R$ such that
%$$
%g(x_{0},x_{1})= l_{0}(x_{0}) + l_{1}(x_{1}) \quad \mbox{for all } (x_{0},x_{1}) \in \R^{n}\times \R^{n}\,.
%$$
%\end{itemize}
%}
%\noindent
%{\bf Step 1.} We show that the assertions of Thm. \ref{main} are valid under (H1), (H2) and (A1)-(A5). Then they are valid  under (H1),(H2), (A1)-(A4).
% and (A5)$\,'$, where 
%\begin{itemize}
%\item[(A5)$\,'$:] There exist Lipschitz continous functions $l_{0}(.): \R^{n} \rightarrow \R$ and  $l_{1}(.): \R^{n} \rightarrow \R$ such that 
%(\ref{sep}) is true.
%\begin{equation*}
%\label{sep}
%g(x_{0},x_{1})= l_{0}(x_{0}) + l_{1}(x_{1}) \quad \mbox{for all } (x_{0},x_{1}) \in \R^{n}\times \R^{n}\,.,
%\end{equation*}
%\end{itemize}
%\end{lem}
\noindent
Assume  that Thm. \ref{main} is valid under (H1)-(H3) and (A0)-(A4).  Suppose  $(\bar x(.), \bar u(.))$ is an $L^{\infty}$  local minimizer for (P) when we impose hypotheses (H1)-(H3) and (A0)-(A3).
 %and (A4)$'$. 
 By (H1), $l(.,.)$ is Lipschitz continuous on a ball about $(\bar x(S), \bar x(T))$. By redefining this function outside the ball (a change which does not affect $L^{\infty}$ local minimizers), we can arrange that $l(.,.)$ is Lipschitz continuous on $\R^{n} \times \R^{n}$; write the Lipschitz constant $k_{l}$.
% Notice $l_{0}(.)$ and $l_{1}(.)$ are arbitrary Lispchitz functions (with  Lipschitz constants $k_{l_0}$ and $k_{l_1}$ respectively.) 
\vspace{0.1 in}

\noindent
For  $i=1,2,\ldots$, let $l^{i}(.)$ be the `$i$ - quadratic inf convolution' of $ l(.)$: 
\begin{equation}
 \label{infconv}
l^{i}(z)\;:=\; \inf_{y \in R^{n}\times \R^{n}} \{ l(y) + i \times |y-z|^{2}  \} \;.
 \end{equation}
 The key `quadratic inf convolution' properties of $l^{i}(.)$ (see \cite{ClarkeLed}) are: take any $z \in\R^{n}\times \R^{n}$  and let  $y \in \R^{n}\times \R^{n}$ 
 be any vector achieving the infimum in (\ref{infconv}) (one such vector exists). Let  
 $$
 \eta^{i}(z)\,:=\, 2 i (y- z)\,.
 $$
 \begin{itemize}
 \item[(i):] 
 $ l^{i}(.)$ is locally Lipschitz continuous  with Lipschitz constant $k_{l}$.
\item[(ii):]  $l^{i}(z) \,\geq\, l_{j}(z) - k^2_{l} \times i^{-1}$.
 \item[(iii):]  $
 l^{i}(z')-l^{i}(z) \,\leq\, \eta^{i}  (z) \cdot (z'-z) + i \times |z'-z|^{2}$ for all $z' \in \m R^{n}\times \R^{n}$
\item[(iv):] $\eta^{i}(z) \in \partial_{P} l(y)$.
\item[(v):] $|y-z| \leq k_{l} \times i^{-1} $\;.
 \end{itemize} 
Since $(\bar x(.), \bar u(.))$ is an $L^{\infty}$ local minimizer for ($P$), $(\bar x(.), \bar u(.))$ is a minimizer for 
$$
(Q):\quad \mbox{Minimize } \{ J((x(.),  u(.))\,|\, ( x(.),  u(.)) \in {\cal B}_{\epsilon}\}\,,
$$
for some $\epsilon >0$, where 
$$
J((x(.),  u(.)):=\,  \int_{[S,T]} \tilde L(t,u(t))dt   + l(x(S),x(T))\,,
$$
and 
$
{\cal B}_{\epsilon} := \{ \mbox{feasible processes } (x(.),  u(.))   \mbox{ for } (P)\,|\,  ||x(.)-\bar x(.)||_{L^\infty} \leq \epsilon \} \,.
$
\noindent

\ \\
For each $i$, consider  the problem
$$
(Q^{i}):\quad \mbox{Minimize } \{ J^{i}((x(.),  u(.))\,|\, ( x(.),  u(.)) \in {\cal B}_{\epsilon}\}
$$
in which 
$$
J^{i} ((x(.),  u(.))\, := \, \int_{[S,T]} \tilde L(t,u(t))dt   + l^{i}(x(S),x(T))\,.
$$
%and
%$$
%{\cal A}_{\epsilon} := \{ ( \mbox{ processes } (x(.),  u(.),  d(.))   \mbox{ for } (P)\,|\,  ||x(.)-\bar x(.)||_{C} \leq \epsilon \} \,.
%$$
Equip ${\cal B}_{\epsilon}$ with the metric
$$
d_{{\cal E}}((x'(.),  u'(.))) ,(x(.),  u(.)) ) =
\mbox{ meas } \{t \in [S,T]\,|\, u'(t) \not= u(t) \}+ |x'(S)- x(S)|\,.
$$
It can be shown that, w.r.t. this metric, $ {\cal B}_{\epsilon}$ is complete and $ J^{i}(.,.,.)$
is continous on $ {\cal B}_{\epsilon}$.
\ \\

\noindent
Now note that, in view of property (ii) of the quadratic inf convolution operation, $(\bar x(.), \bar u(.))$ is a $\gamma_i\,$-minimizer for $(Q^{i})$, where $\gamma_i=k_{l}^2\, i^{-1}$. In consequence of Ekeland's Theorem, there exists 
$( x_{i}(.), u_{i}(.)) \in {\cal B}_{\epsilon}$ which is a minimizer for ($\tilde Q^{i})$:
$$
(\tilde Q^{i})\; \mbox{Min } \{ J^{i}((x(.),  u(.)) + \gamma_i^{\frac{1}{2}}
d_{{\cal E}}((x(.),  u(.)) ,(x_{i}(.),  u_{i}(.)) )
| ( x(.),  u(.)) \in {\cal B}_{\epsilon}\}.
$$
and
\begin{equation}
\label{ekeland}
d_{{\cal E}}((x_{i}(.),  u_{i}(.)) , (\bar x(.), \bar u(.)) ) \,\leq \, \gamma_i^{\frac{1}{2}}\,.
\end{equation}
%Define $m^{i}_{0}(.,.):[S-h,S]\times \R^{n+m} \rightarrow \R$, $m^{i}_{1}(.,.):[S,T]\times \R^{m} \rightarrow \R$
%$$
%m^{i}_{0}(t,d):= 
%\left\{
%\begin{array}{ll}
%0 & \mbox{if } d=d_{i}(t)
%\\
%1  & \mbox{if } d \not= d_{i}(t)
%\end{array}
%\right.\;\mbox{ and }\;
%m^{i}_{1}(t,u):= 
%\left\{
%\begin{array}{ll}
%0 & \mbox{if } u=u_{i}(t)\; .
%\\
%1  & \mbox{if } u \not= u_{i}(t)
%\end{array}
%\right.\;.
%$$
The cost function for $(\tilde Q^{i})$ can be written
\begin{eqnarray*}
\hspace{-0.3 in}&&
\tilde J^{i} (x(.),  u(.)):=\, 
\int_{[S,T]}( \tilde L(t,u(t))+ \gamma_i^{\frac{1}{2}} m_{i}(t,u(t)))dt 
  + l^{i}(x(S),x(T))+   \gamma_i^{\frac{1}{2}}|x(S)-x_{i}(S)|.
%  + l^{i}_{1}(x(T))\,.
\end{eqnarray*}
($m_{i}(t,u)$ was defined in (\ref{errorx})). But by property (iii) of quadratic inf convolutions (see above),  
\begin{eqnarray*}
&&
 (l^{i}(x_{0}, x_{1})   - (l^{i}(x_{i}(S), x_{i}(T) )    
 \\
 &&
 \qquad \leq \, \eta^{i}_{0}  \cdot (x_{0}-x_{i}(S)) + \eta^{i}_{1}  \cdot (x_{1}-x_{i}(T)) 
+ i \times( |x_{0}-x_{i}(S)|^{2}+ |x_{1}-x_{i}(T)|^{2}    )   
\end{eqnarray*}
for all $ (x_{0},x_{1}) \in \R^{n}\times \R^{n}$, with equality when $ (x_{0},x_{1})= (x_{i}(S),x_{i}(T))$.
Here
$$
(\eta_{0}^{i},\eta_{1}^{i} ) \in  \partial_{P} l_{0}(y_{0},y_{1}) %\mbox{ and } \eta_{1}^{i} \in  \partial_{P} l_{1}(y_{1})
$$
for some  $(y_{0},y_{1}) \in (x_{i}(S),x_{i}(T))+ k_{l}\, i^{-1}(\B\times \B)$. So  $( x_{i}(.), u_{i}(.)(.))$ is also a minimizer for 
$$
(\tilde Q^{i}_{0}):\quad \mbox{Minimize } \{ \tilde J_{0}^{i}((x(.),  u(.))\,|\, ( x(.),  u(.)) \in {\cal A}_{\epsilon}\}
$$
in which
\begin{eqnarray*}
&&
\tilde J_{0}^{i} ((x(.),  u(.))\;:=\;
\int_{[S,T]}( \tilde L(t,u(t))+ \gamma_i^{\frac{1}{2}} m_{i}(t,u(t)))dt 
 +\eta^{i}_{0} \cdot (x(S)-x_{i}(S)) 
\\
&&
\hspace{0.4 in} 
+ \eta^{i}_{1} \cdot (x(T)-x_{i}(T)) + i \times (|x(S)-x_{i}(S)|^{2}+ |x(T)-x_{i}(T)|^{2} ) + \gamma_i^{\frac{1}{2}}|x(S)-x_{i}(S)|\,.
% + (l_{0}(x(S))+   i^{-\frac{1}{2}}|x(S)-x_{i}(S)|)+\eta_{i} \cdot (x-x_{i}(S)) + i \times |x-x_{i}|^{2}\,.
\end{eqnarray*}
The data for Problem $(\tilde Q^{i})$ satisfies (H1), (H2), (A0)-(A3) and (A4).
%, since the terminal cost function has the required structure and regularity properties. 
We may therefore apply the Maximum Principle (with $\lambda=1$), which we know to be valid under these hypotheses. We conclude the existence of 
$p^{i}_{k}(.):[S,T] \rightarrow \R^{n}$, $k=0,\ldots,N$, and $p^{i}(.):[S,T] \rightarrow \R^{n}$ 
%and $\lambda \geq 0$,
 such that (\ref{p_comp}) and (\ref{p_eqn}), as well as conditions (a) and (b) of the theorem statement, are satisfied, when  $( x_{i}(.), u_{i}(.))$ replaces $(\bar  x(.), \bar u(.))$ and $p^{i}_{k}(.)$ replaces $p_{k}(.)$, etc., and when $\lambda =1$. 
 %(Since $(\tilde P)^{i}$ has a free right endpoint, the conditions can be satisfied only with $\lambda >0$; the Lagrange multipliers can then be scaled to yield $\lambda = 1$.) 
 In addition, we have, for any selector $u(.)$ of $U(.)$,
\vspace{0.1 in}

%\hspace{-0.5 in}
(c)$'$: 
$
 \int_{[S,T]}(p\cdot f-\lambda \tilde L)(t,\{x_{i}(t-h_{k})\},\{u(t-h_{k})\})dt $
%&&\hspace{0.8 in}
\vspace{0.1 in}

\hspace{0.8 in}$\leq \, \int_{[S,T]}(p\cdot f-\lambda \tilde L)(t,\{x_{i}(t-h_{k})\},\{  u_{i}(t-h_{k})\})dt  \hspace{0.6 in}
$
%(c)$'$: 
 %${\cal H}_{\lambda=1}(t, u_{i}(t);  x_{i}(.),u_{i}(.), d_{i}(.), p^{i}(.)) \geq \underset{u \in U(t)}{\max}
%{\cal H}_{\lambda=1}(t, u;  x_{i}(.),u_{i}(.), d_{i}(.), p^{i}(.)) - \gamma_i^{\frac{1}{2}}
%$
%
%\noindent
%\hspace{4.0 in} 
%for a.e. $t \in [S,T]$\,, 
%\ \\
%
%\noindent
%and
%
%(d)$'$: 
%${\cal H}^{0}_{\lambda=1}(t,d_{i}(t); x_{i}(.), u_{i}(.), d_{i}(.), p^{i}(.))\geq \underset{d \in D(t)}{\max} 
%{\cal H}^{0}_{\lambda}(t,d; x_{i}(.), u_{i}(.), d_{i}(.), p^{i}(.)) -\gamma_i^{\frac{1}{2}} 
%$
%
%\noindent
%\hspace{4.0 in} for a.e. $t \in [S-h,S]$\,.
%\ \\
\vspace{0.1 in}

(d)$'$:  
$ (p^{i}(S),-p^{i}(T)) \in \partial l((x_{i}(S)+x_{i}(T))+ k_{l}i^{-1}\B)+\gamma_i^{\frac{1}{2}} \B\times \{0\}$.
\ \\

\noindent
We deduce from (\ref{ekeland}) that, along some subsequence, $u_{i}(t) \rightarrow \bar u(t)$, a.e. $t \in [S,T]$ and $x_{i}(t) \rightarrow \bar x(t)$ uniformly over $t \in [S,T]$. On the other hand, we can show from the conditions that  $p^{i}_{k}(.)$, $k=1,\ldots, N$ and $p^{i}(.)$, $i=1,2, \ldots,$ are uniformly bounded on their domains and their derivatives are uniformly integrably bounded. 
We may deduce from Ascoli's theorem that, after 
a further subsequence extraction,  for each $k$, $p^{i}_{k}(.)$ converge uniformly to some $W^{1,1}$ function $p_{k}(.)$ as $i \rightarrow \infty$, and the derivatives $\dot p^{i}_{k}(.)$ converge weakly in $L^{1}$ to $\dot p_{k}(.)$, for each $k$. The  function $p^{i}(.)$ converges likewise to some $p(.)$. A standard analysis permits us to pass to the limit in conditions (a) and (b) (modified as indicated above) and (c)$\, '$- (d)$\, '$. We thereby achieve confirmation of all assertions of Thm. \ref{main} (with $\lambda = 1$), for the special case of when the additional hypotheses (A1)-(A4) and (A5)$\,'$ are satisfied.  
%%%
\vspace{0.1 in}

\noindent
{\bf Step 3:} 
%{\it Assume the assertions of Thm. \ref{main} are valid under (H1)-(H3) and (A1)-(A3), (A4) and (A5)$\,'$. Then they are valid under (H1)-(H3), (A1)-(A3), (A4)$\,'$ and (A5)$\,'$, where 
%\begin{itemize}
%\item[(A4)$\,'$:] 
%There exists Lipschitz continous functions $l_{0}(.): \R^{n} \rightarrow \R$ and  $l_{1}(.): \R^{n}$ .
%There exists a closed set $C_{0} \subset \R^{n}$ such that $C= C_{0} \times \R^{n}$.
%\end{itemize}
%Then they are valid under (H1),(H2), (A1)-(A3).
%\end{lem}
%{\bf Proof.}  
%{\it 
%We show that, if  the assertions of Thm. \ref{main} are valid 
%under hypotheses (H1), (H2), (A1) - (A4) and (A5)$\, '$. %Suppose that 
%when $C=\R^{n}\times \R^{n}$. Now suppose $C=C_{0} \times \R^{n}$. Let 
%$(\bar x(.), \bar u(.), \bar d(.))$ is an $L^{\infty}$ local solution (with parameter $\epsilon >0$).
%Then they are valid also  under (H1), (H2), (A1) - (A3) and (A5)$\, '$, and when (H4) is replaced by 
%the weaker hypothesis (A4)$\, '$, where: 
%\begin{itemize} 
%\item[(A4)$\,'$:] 
%\noindent
%\begin{lem}
%\ref{lem1}
%Suppose the assertions of Thm. \ref{main} are valid under (H1), (H2) and (A1)-(A3), (A4) and (A5)$\,'$. Then they are valid under (H1),(H2), (A1)-(A3), (A4)$\,'$ and (A5)$\,'$, where 
%\begin{itemize}
%\item[(A4)$\,'$:] 
%There exists Lipschitz continous functions $l_{0}(.): \R^{n} \rightarrow \R$ and  $l_{1}(.): \R^{n}$ .
%There exists a closed set $C_{0} \subset \R^{n}$ such that $C= C_{0} \times \R^{n}$.
%\end{itemize}
%}
%Then they are valid under (H1),(H2), (A1)-(A3).
%\end{lem}
%\ \\
%
%\noindent
{\it Assume that the assertions of Thm. \ref{main} are valid 
under hypotheses (H1)-(H3), (A0) - (A2) and (A3).
%, and when $L(.\,.)\equiv 0$. 
Then the assertions remain valid (with $\lambda=1$) when  
%when $C=\R^{n}\times \R^{n}$. Now suppose $C=C_{0} \times \R^{n}$. Let 
$(\bar x(.), \bar u(.))$ is an $L^{\infty}$ local solution 
%(with parameter $\epsilon >0$) 
under the weaker set of hypotheses, in which (A3) is replaced by (A3)$\, '$, where 
\begin{itemize}
\item[(A3)$\,'$:] 
%There exists Lipschitz continous functions $l_{0}(.): \R^{n} \rightarrow \R$ and  $l_{1}(.): \R^{n}$ .
There exists a closed set $C_{0} \subset \R^{n}$ such that $C= C_{0} \times \R^{n}$.
\end{itemize}
}

\noindent
A simple contradiction argument (c.f. the proof of the `Exact Penalization Thm.' \cite[p.48]{Vinter}), based on Thm. \ref{Filippov}, permits us to conclude that 
$(\bar x(.), \bar u(.))$ remains an $L^{\infty}$ local minimizer 
%(for $\epsilon' \in (0, \epsilon] $)   
for a modified problem in which the endpoint constraint set $C_{0}\times \R^{n}$ is replaced by $\R^{n} \times \R^{n}$ and the endpoint cost function $l(.)$ in (P) is replaced by  the Lipschitz continuous function $\tilde l(.)$:
$$
\tilde l(x_{0},x_{1}):= l(x_{0},x_{1})+ K d_{C_{0}}(x_{0})\,,
$$
for any 
$$
K > k_{l} \times \exp \{(N+1)\x \int_{S}^{T} k_{0}(t)dt \}\,,
$$
in which $k_{l}$ is the Lipschitz constant of $l_{i}(.)$. Applying Thm. \ref{main} to the modified problem, which is permissible since $C= \R^{n}\times \R^{n}$ and $l_{0}(.)= \tilde l_{0}(.)$, yields the desired necessary conditions for the original problem. 
%The key facts here, justifying this step, are that we can choose the cost multiplier  $\lambda =1$ (this is a consequence of the fact that the special case of (P) considered is a free right end-point problem) %and the non-triviality condition on the Lagrange multipliers ) 
Note that the transversality condition for the modified problem implies
\vspace{0.1 in}

$
(p(S), -p(T) \in \partial \tilde l(\bar x(S), \bar x(T)) \subset  \partial l(\bar x(S), \bar x(T)) + K \partial d_{C_{0}}(\bar x(S))\times \{0\} $
\vspace{0.05 in}

\hspace{2.02 in}$\subset  \partial l(\bar x(S),\bar x(T)) + N_{C}(\bar x(S),\bar x(T))\,, 
$
\vspace{0.05 in}

\noindent
which is the appropriate left transversality condition for the original problem.
\vspace{0.1 in}

\noindent
{\bf Step 4:} {\it Assume the assertions of Thm. \ref{main} are valid under (H1)-(H3), (A0)-(A2), (A3)$\,'$. Then they are valid under (H0)-(H3), (A0)-(A2).}
%\ref{lem1}
%\end{lem}
\ \\

\noindent
Assume that the assertions of Thm. \ref{main} are valid 
under hypotheses (H1)-(H3), (A0) - (A2) and (A3)$\, '$. Suppose that 
%when $C=\R^{n}\times \R^{n}$. Now suppose $C=C_{0} \times \R^{n}$. Let 
$(\bar x(.), \bar u(.))$ is an $L^{\infty}$ local solution to (P) 
%(with parameter $\epsilon >0$) 
under   (H1)-(H3), (A1) and (A2) alone. We must show that  $(\bar x(.), \bar u(.))$ satisfies the Maximum Principle. 
%\ \\
%
%\noindent
%We may assume  that  
%$\Lambda(.,.)\equiv 0$  and 
%$L(.\,.)\equiv 0$  (no integral cost terms  in (P)) since, if an integral term is present in the cost, it can be accommodated by `state augmentation', as in the proof of Lemma \ref{lemma_reduction}.
\ \\

\noindent
Take $\gamma_{i} \downarrow 0$. For $i=1,2\ldots$, consider the problem with $(n+n)$-dimensional state vector $(z,x)$: 
$$
(P_{1}^i)\left\{ 
\begin{array}{l}
\mbox{Minimize } J_1^{i}(z(.),x(.),u(.))\; \mbox{ subject to}
\\
(\dot z(t),\dot x(t)) = (0,f(t,\{x(t-h_{k})\}, \{u(t-h_{k})\}))
 %\\
% \hspace{3.0 in} u(t-h_{N});d(t-h_{0}), \ldots, d(t-h_{N}),\,0)
\\
%||x(.)-\bar x(.) ||_{L^{\infty}} \leq \bar \epsilon
u(t) \in U(t)\mbox{ a.e. } t\in [S,T],
\\
(z(S),x(S))\in \tilde C:= \{(z,x) \in \R^{n}\times \R^{n}\,|\, z=x\}\,,
\end{array}
\right.
$$
in which
% 
%\begin{eqnarray*}
%&&{\cal A}^{1}_{\epsilon}\,:=\, \{ (x(.),y(.)),u(.),d(.)) \,|  \dot z(T)=0, (x(.),u(.),d(.)) \mbox{ is a  process for } (P),
%\\
%&& \hspace{3.0 in} (x(S),y(S)) \in C, \,  ||x(.)- \bar x(.)||_{L^{\infty}} \leq \epsilon
 %\, \}
%\end{eqnarray*}
%andt
$
J_1^{i}(z(.),x(.),u(.)) = 
%\int_{[S-h,S]}\Lambda(t,d(t))dt + 
\max\{ g(z(T), x(T))- g(\bar z(T), \bar x(T)) + \gamma_{i}\,, d_{C}(z(T),x(T)) \}+ \int_{[S,T]}\tilde L(t,u(t))dt\,.
$
\vspace{0.1 in}

\noindent
Here, $\bar z(.) \equiv \bar x(S)$\,. Since
$
J^{i}_{1}(\bar z(.), \bar x(.),\bar u(.))= \gamma_{i},
$
and $J^{i}_{1}(.\,.)$ is non-negative valued,  $(\bar z(.), \bar x(.),\bar u(.))$ is a $\gamma_{i}$-minimizer.
\ \\

\noindent
Let $\epsilon >0$ be such that $(\bar x(.), \bar u(.) )$ is a %$L^\infty$ local 
minimizer for $(P)$
w.r.t. feasible state trajectories $x(.)$ satisfying $||x(.)-\bar x( .)||_{L^{\infty}}\leq \epsilon$. Define
$$
{\cal A}^{1}_{\epsilon} := \{ \mbox{ feasible processes } (z(.), x(.),  u(.))   \mbox{ for } (P_{1}^{i})\,|\,  ||x(.)-\bar x(.)||_{L^\infty} \leq \epsilon \} \,.
$$
Equip ${\cal A}^{1}_{\epsilon}$ with the metric
\vspace{0.1 in}

$
d_{{\cal E}}((z'(.), x'(.),u'(.)) ,(z(.), x(.), u(.)) ) 
\,=
$
\vspace{0.05 in} 

\hspace{1.7 in}$\mbox{ meas } \{t \in [S,T]\,|\, u'(t) \not= u(t) \}\, +\, |x'(S)- x(S)| + |z'(S)-z(S)|\,.$
\vspace{0.1 in}

\noindent
With this metric, $ {\cal A}^{1}_{\epsilon}$ is a complete metric space and $ J_1^{i}(.,.,.)$
is continuous on $ {\cal A}^{1}_{\epsilon}$.
\ \\

\noindent
 In consequence of Ekeland's Theorem, there exists 
$( z_{i}(.),  x_{i}(.), u_{i}(.)) \in {\cal A}^{1}_{\epsilon}$, for which
\begin{equation}
\label{eke2}
d_{{\cal E}}((z_i(.), x_i(.),u_i(.)) ,(\bar z(.), \bar x(.), \bar u(.)) \leq \gamma_{i}^{\frac{1}{2}}\,,
\end{equation}

\noindent 
and minimizes $\tilde J^{i}_1(.)$ over ${\cal A}^{1}_{\epsilon}$. Here, %which is a minimizer for a perturbed version ($\tilde P_{1}^{i})$ of ($P_{1}^{i})$, in which the cost is now
\begin{multline*}
\tilde J^{i}_1 (z(.),x(.),  u(.)) = J^{i}_1(z(.),x(.),  u(.)) +\\
\gamma_{i}^{\frac{1}{2}}( \int_{[S,T]}m_{i}(t,u(t))\,dt + |x(S)-x_{i}(S)|+ |z(S)-z_{i}(S)|)\,.
\end{multline*}
From (\ref{eke2}) it can be deduced that $|| (z_{i}(.),x_{i}(.)) -(\bar x(S),\bar x (.))||_{L^{\infty}} \rightarrow 0$ as $i \rightarrow \infty$ and therefore that, for $i$ sufficiently large, $(z_{i}, x_{i}(.))$ is an $L^{\infty}$ local minimizer for $(P^{i}_{1})$. The data for this last problem satisfies (H1)-(H3), (A0)-(A3), (A4)$\,'$ and (A5)$\,'$, and we may therefore apply the Maximum Principle (with $\lambda =1$). 
%\ \\
%
%\noindent
%The Maximum Principle, which is valid with cost multiplier $\lambda =1$ because the perturbed problem has a free right endpoint, provides 
We deduce the existence of a costate trajectory $p(.)$ with associated decomposition $p(.)= \sum_{k=0}^{N} p_{k}(.)$ (arising form the $x$-state) and a  costate trajectory $q(.)$ (arising from the $z$-state), and with properties given by Thm. \ref{main}. Conditions (b)and (c) imply that $q(.)$ is a constant, which we write $q$, 
%\item[(a):] $\lambda+\sum_{k}\|p_{k}(.)\|_{L^{\infty}}\,\not=\, 0$,
\ \\

\noindent
$(b)':\{-\dot p_{k}(t-h_{k})\} \in \mbox{co}\, \partial_{x_{0},\ldots,x_{N}}\,p(t)\cdot f (t,\{x_{i}(t-h_{k})\},
 \{u_{i}(t-h_{k})\}, \;
 %,\ldots, u_{i}(t-h_{N})))
\mbox{ a.e. }t \in [S,T] .$
\ \\ 
% $ \hspace{2.3 in}((x_{0},\ldots, x_{N})=  \bar x(t-h_{0}),\ldots, \bar x(t-h_{N}))$,
 
%\noindent
% $\hspace{4.5 in} \mbox{a.e. } t \in [S,T].$
%\item[(c):] 

\noindent
(c)$'$: for any selector $u(.)$ of $U(.)$
\begin{eqnarray*}
&&
\int_{[S,T]} (p(t)\cdot f -\lambda \tilde L)(t,\{x_{i}(-h_{k})\}, \{u(t-h_{k})\})dt \leq
\\
&&\hspace{0.5 in}
\int_{[S,T]} (p(t)\cdot f -\lambda \tilde L)(t,\{x_{i}(-h_{k})\}, \{u_{i}(t-h_{k})\})dt + \gamma_{i}^{\frac{1}{2}}|T-S|
\end{eqnarray*}
%${\cal H}(t,u_{i}(t);  x_{i}(.), u_{i}(.), d_{i}(.), p(.))\geq \underset{u \in U(t)}{\max}
%{\cal H}(t,u;  x_{i}(.), u_{i}(.), d_{i}(.), p(.)) -\gamma_{i}^{\frac{1}{2}} \mbox{ a.e.~$t \in [S,T]$}
%$
%\vspace{0.05 in}
%and%\noindent
%\hspace{4.0 in} for a.e. $t \in [S,T]$\,, 
%\vspace{0.05 in} 
%
%\noindent
%$(d)':
%{\cal H}^{0}(t,d_{i}(t);  x_{i}(.), u_{i}(.), d_{i}(.), p(.))\!\geq\! \underset{d \in D(t)}{\max}
%{\cal H}^{0}(t,d ; x_{i}(.), u_{i}(.), d_{i}(.), p(.))
% -\gamma_{i}^{\frac{1}{2}}
% \!\mbox{ a.e.\,$t \in [S\!-\!h,S]$}
%$
%\vspace{0.05 in}
%
%
%\noindent
%in which ${\cal H}(.\,.)$ and ${\cal H}^{0}(.\,.)$ are obtained from the formulae for ${\cal H}_{\lambda}(.\,.)$ and ${\cal H}_{\lambda}^{0}(.\,.)$ in Section \ref{sectionNC}, for the given $f(.\,.)$ and $L(.\,.)\equiv 0$. (Because $L(.\,.)\equiv 0$, these expressions do not depend on $\lambda$.)
%\ \\
%
%\noindent
Let us examine the implications of the transversality condition (e). In this connection, we make use of the fact that
$$
\max\{g(z_{i}(T),x_{i}(T))- g(\bar z(T), \bar x(T)) + \gamma_{i} ,  d_{C}(z_{i}(T),x_{i}(T)) \}>0,\, \mbox{ for $i$ sufficiently large}\,.
$$
Indeed if this were not the case then, since $z_{i}(T)=x_{i}(S)$, we would have $g(x_{i}(S),x_{i}(T))- g(\bar x(S), \bar x(T))\leq  -\gamma_{i}$ and $d_{C}(x_{i}(S),x_{i}(T))=0$. We could also arrange, by choosing $i$ sufficiently large, that $||x_{i}(.)-\bar x(.)||_{L^{\infty}}\leq \epsilon$ for $\epsilon$ arbitrarily smalll. This contradicts the $L^{\infty}$ local optimality of $( x_{i}(.), u_{i}(.))$ for $(P)$.
%. that This follows from the facts that $z_{i}(T)=x_{i}(S)$ and that  $(\bar x(.), \bar u(.),\bar d(.))$ is an $L^\infty$ local minimizer for $(P)$. 
\ \\

%%since, otherwise, the assumption that $(\bar x(.), u(.),d(.))$ is an $L\infty$ local minimizer is violated.
%%$$
%%d_{C}(y_{i}(T),x_{i}(T)) > 0\,.
%%$$\
%%
\noindent
Note that $N_{\tilde C}(x,x)=\{(e_{0},e_{1})|e_{0}+ e_{1}=0 \}$ which, in combination with the transversality condition $(d)$, yields the information
\begin{equation}
\label{q cond}
q\in -p(S)+2\gamma_i^{\frac12} \B.
\end{equation}
 In consequence of the  max  rule for limiting subdifferentials, we know that
$$
\partial\max\{g(z,x)- g(\bar z(T), \bar x(T)) + \gamma_{i} ,  d_{C}(z,x)\} \subset \lambda \partial g(z,x) +(1-\lambda)\partial d_{C}(z,x)\,,
$$
for some $\lambda \in [0,1]$. Moreover, 
`$1- \lambda >0$' implies 
$$
\mbox{`}d_{C}(z,x)= \max\{g(z,x)- g(\bar z(T), \bar x(T)) + \gamma_{i} ,  d_{C}(z,x)\}\mbox{`}
$$
Since $z_{i}(S)=z_{i}(T) = x_{i}(S)$, and in view of (\ref{q cond}), we deduce from condition $(d)$ that
\ \\

\noindent
(d)$':\;
(p(S),-p(T)) \in \lambda \partial g(x_{i}(S),x_{i}(T)) + (1-\lambda) \partial d_{C}(x_{i}(S), x_{i}(T))+2\gamma_i^{\frac12}\B\x\{0\}\,.
$
\ \\

\noindent
We claim that, for $i$ sufficiently large, 
$$
\lambda + ||p(.)||_{L^{\infty}} >0\,.
$$
Indeed if this were not true, we would have $\lambda =0$ and $(p(S),p(T))=0$. It would follow that
$$
(q,0) \in \partial d_{C}(x_{i}(S),x_{i}(T)) \mbox{ and } |q| \leq 2 \gamma_{i}\,.
$$
But $\lambda =0$ also implies $(1-\lambda) > 0$, whence $d_{C}(x_{i}(S),x_{i}(T)) >0$. This last condition is known to implies that all elements $\xi$ in the set $\partial d_{C}(x_{i}(S),x_{i}(T))$ have unit Euclidean length (see \cite{Vinter}). This is not possible because $(q,0)$ is such an element, and has Euclidean length $2 \gamma_{i}$ (which can be made arbitrarily small).%follows that $|q|\leq 2 \gamma_{i}$.
 We can therefore arrange,  by positive scaling  of the Lagrange multipliers $(\{p_{k}(.)\}, p(.),\lambda)$, that
\ \\

\noindent
 $(a)':\qquad %\begin{equation}
%\label{normal}
||p(.)||_{L^{\infty}} + \lambda =1\,.
%\end{equation}
$
\ \\

\noindent
Bearing in mind that $\partial d_{C}(z) \subset N_{C}(z)$ (for $z \in C$), we have  arrived at a set of relations  $(a)'-(d)'$ that are approximate version of those asserted in Thm. \ref{main}, involving the multiplier set $\{p_{k}(.)\},\lambda)$, with reference to $(x_{i}(.),u_{i}(.))$. To emphasize the fact that these relations depend on  $i$, we rewrite the multiplier set $\{p^{i}_{0}(.)\}, p^{i}(.),\lambda^{i})$. Condition (\ref{eke2}) ensures that, along a subsequence $u_{i}(t) \rightarrow \bar u(t)$ a.e.. We deduce from Thm. \ref{Filippov} that $x_{i}(.) \rightarrow \bar x(.)$ uniformly. For each $k$, $p^{i}_{k}(.)$, $i=1,2,\ldots$, is a uniformly bounded sequence of absolutely continuous functions with  uniformly integrably bounded derivatives. It follows that, for each $k$, $\{p^{i}_{k}(.)\}$ converges to  an absolutely continuous function $p_{k}(.)$, for $k=0,\ldots,N$, and $\{\dot{p}^{i}_{k}(.)\}$  converges to $\dot p_k(.)$ weakly in $L^{1}$, along a subsequence. We can also arrange that $\lambda_{i}\rightarrow \lambda$ for some $\lambda \in [0,1]$. A standard convergence analysis permits us to pass to the limit in relations $(a)'-(d)'$, and thereby arrive at the assertions of Thm. \ref{main}. 
%and
%\begin{equation}
%\label{ekeland}
%d_{{\cal E}}((x_{i}(.),  u_{i}(.),  d_{i}(.)); (\bar x(.), \bar u(.)), \bar  d(.)) ) \,\leq \, i^{-\frac{1}{2}}\,.
%\end{equation}
%Define $m^{i}_{0}(.,.):[S-h,S]\times \R^{n+m} \rightarrow \R$, $m^{i}_{1}(.,.):[S,T]\times \R^{m} \rightarrow \R$
%$$
%m^{i}_{0}(t,d):= 
%\left\{
%\begin{array}{ll}
%0 & \mbox{if } d=d_{i}(t)
%\\
%1  & \mbox{if } d \not= d_{i}(t)
%\end{array}
%\right.\;\mbox{ and }\;
%m^{i}_{1}(t,u):= 
%\left\{
%\begin{array}{ll}
%0 & \mbox{if } u=u_{i}(t)\; .
%\\
%1  & \mbox{if } u \not= u_{i}(t)
%\end{array}
%\right.\;.
%$$
%The cost function for $(\tilde P^{i})$ can be written
%\begin{eqnarray*}
%&&
%\tilde J^{i} ((x(.),  u(.),  d(.)):=\, \int_{[S-h,S]}( \Lambda (t,d(t)) + i^{-\frac{1}{2}} m_{0}^{i}(t,d(t)))dt 
%\\
%&&  \hspace{2.0 in } + 
%\int_{[S,T]}( \tilde L(t,u(t))+ i^{-\frac{1}{2}} m_{1}^{i}(t,u(t)))dt 
%\\
%&&
%\hspace{2.5 in}  + (l_{0}(x(S))+   i^{-\frac{1}{2}}|x-x(S)|)+ l^{i}_{1}(x(T))\,.
%\end{eqnarray*}
%
\vspace{0.1 in}

\noindent
{\bf Step 5:} {\it Suppose the assertions of Thm. \ref{main} are valid under (H1)-(H3), (A0)-(A2). Then they are valid under (H1)-(H3), (A0) and (A1).}
\ \\

\noindent
Assume that the assertions of Thm. \ref{main} are valid under (H1)-(H3), (A0)-(A2). Suppose $(\bar x(.), \bar u(.))$ is an $L^{\infty}$ local minimizer when the hypotheses are satisfied, with the possible exception of  (A2). For $i=1,2,\ldots$, define family of functions
\begin{eqnarray}
\nonumber 
&&{\cal U}_{i}\,:=\, \{u(.) \mbox{ is a selector of $U(.)$}\,|\, k(t,\{u(t-h_{k})\}) +
\\
&&
\label{new_U_1}
\hspace{0.4 in}
 |f(t,\{\bar x(t-h_{k})\},\{u(t-h_{k})\} )|
\leq
 k(t,\{\bar u(t-h_{k})\}) + |\dot{\bar{x}}(t))\} )+ i \}
\end{eqnarray}
For each $i$, let $U_{i}(.):[S,T] \leadsto \R^{m}$ be the multifunction defined, ${\cal L}$, a.e. by the condition 
\begin{equation}
\label{new_U_2}
\mbox{Gr}\,U_{i}(.)= {\cal U}_{i}. 
\end{equation}
We note that
\begin{equation}
\label{approx}
\bar u(t) \in U_{1}(t) \subset \ldots U_{2}(t)\subset \ldots \mbox{ and }  \cup_{i}\,U_{i}(t)= U(t) \mbox{, a.e.}
\end{equation}

\noindent
For each $i$, $(\bar x(.),\bar u(.))$ continues to be an $L^\infty$ local minimizer for (P$_{i})$ which is the modification of (P) in which $U_{i}(.)$ replaces $U(.)$. Because the data for (P$_{i}$) satisfies (A2), the assertions of Thm. \ref{main} are available to us: they yield (for each $i$) a cost multiplier $\lambda_{i}\geq 0$ and a costate arc $p^{i}(.)$, with decomposition $p^{i}(.)=\sum_{k}p^{i}_{k}(.)$, such that 
$\lambda_{i}+ ||p^{i}(.)||_{L^{\infty}}=1$ and (b)-(d) of Thm. \ref{main} are satisfied (when $(\lambda_{i},\{p_{k}^{i}(.)\}, p^{i}(.))$ replaces $(\lambda,\{p_{k}(.)\}, p(.)$)).  Extracting subsequences we can arrange that $p^{i}_{j}(.) \rightarrow p_{j}(.)$, for each $j$ and  $p^{i}(.) \rightarrow p(.)$, uniformly,  as $i \rightarrow \infty$. A standard convergence analysis establishes that conditions (a), (b) and (d) are satisfied when $(\lambda,\{p_{k}(.)\},p(.))$  replaces  $(\lambda_{i},\{p_{k}^{i}(.)\}, p^{i}(.))$  and $\lambda_{i}+ ||p^{i}(.)||_{L^{\infty}}=1$.  However condition (c) is satisfied only in the weaker form:
\begin{eqnarray*}
%\label{integral}
&&\int_{[S,T]}\left(m(t,\{u(t-h_{k})\}) - m(t,\{\bar u(t-h_{k})\})\right) dt
%(p\cdot f-\lambda \tilde L)(t,\{\bar x(t-h_{k})\},\{u(t-h_{k})\})dt 
%\hspace{0.3 in}
%\\
%\nonumber
%&&
%\hspace{0.3 in}
\, \leq \, 0
%\int_{[S-h,T]}m(t,\{\bar u(t-h_{k})\})dt
%
%(p\cdot f-\lambda \tilde L)(t,\{\bar  x(t-h_{k})\},\{ \bar u(t-h_{k})\})\})dt  \hspace{0.6 in}
\end{eqnarray*} 
for all control functions $u(.)$ that are selectors of $U_{i}(.)$ for $i$ sufficiently large, where
$$
m(t,\{u_{k}\}):=(p\cdot f-\lambda \tilde L)(t,\{\bar x(t-h_{k})\},\{u_{k}\})\;.
$$
We must validate (c) when  $u(.)$ is an arbitrary selector of $U(.)$ such that  $m(t,\{u(t-h_{k})\})$ is integrable. For this purpose, define, for each integer $\ell$,
$$
S_{\ell}:= \{t \in [S,T]\,|\, |  
m(t,\{u(t-h_{k})\})-
    m(t,\{\bar u(t-h_{k})\})| \geq \ell  \}
$$
By `integrability',
%Because $t \rightarrow m(t,\{\bar u(t-h_{k})\})$ is integrable,
 meas$\{\, S_{\ell}\} \rightarrow 0$ as $\ell \rightarrow \infty$. For each $\ell$, define the subsets $A^{\ell}_{j} \subset [S,T]$, $j=0,\ldots, K$, in which $K$ is the integer $K := \mbox{int}\, (T-S)/ h_{1}$, recursively: $A^{\ell}_{0}= S_{\ell}$ and 
$A^{\ell}_{j}= ((A^{\ell}_{j-1}+ h_{1})\cap [S,T])\cup \ldots \cup ((A^{\ell}_{j-1}+ h_{N})\cap [S,T])$, $j=1,\ldots, N$.
Now write $A^{\ell}:= \cup_{j=0}^{K} A^{\ell}_{j}$. Define
$$
\tilde u(t) := \bar u(t) \chi_{A^{\ell}}(t)\,+\,u(t) \chi_{[S,T] \backslash  A^{\ell}}(t) \,.
$$
We can deduce from the special structure of $A^{\ell}$ that $\tilde u(.)$ is a selector of $U_{i}(.)$ for $i$ sufficiently large, and
$$
\mbox{`$t \in A^{\ell}$' $\implies$ `$u(t-h_{k})\in A^{\ell}, k=0,\ldots N$' 
and  `$t \notin A^{\ell}$' $\implies$ `$u(t-h_{k})\notin A^{\ell}, k=0,\ldots N$'}
$$
It follows that, for each $\ell$, 
$$\int_{[S,T]\cap A^{\ell}}\left(m(t,\{u(t-h_{k})\}) - m(t,\{\bar u(t-h_{k})\})\right) dt
\, \leq \, 0\,.
$$
Passing to the limit as $\ell \rightarrow \infty$, using the fact that meas$\,\{A^{\ell}\} \leq N^{K} \rightarrow 0$, we arrive at $\int_{[S,T]}\left(m(t,\{u(t-h_{k})\}) - m(t,\{\bar u(t-h_{k})\})\right) dt\, \leq \, 0\,$. We have shown that the assertions of Thm. \ref{main} are valid under (H1)-(H3), (A0)-(A1). 
\ \\

\noindent
Now suppose $(\bar x(.), \bar u(.))$ is an $L^{\infty}$ local minimizer under (H1)-(H3), but when either (A0) or (A1) are possibly violated.
 Then $(\bar y(.), \bar z(.)):[S-h,T]\rightarrow \R^{n+1}, \bar v(.):[S-h,T]\rightarrow \R^{2m +n})$ is an $L^{\infty}$ local minimizer for
the reformulated problem
$$
%\hspace{-0.3 in}
(P_r)
\left\{
 \begin{array}{l}
\mbox{Minimize } \; 
\displaystyle{
%J(y(.),v(.),d(.)):=
g(y(S-h),y(T)) + z(T)
}
%\int_{[S-h,T]}M ( t,\{ y(s-h_{k} )\}, \{v(s-h_{k})\})    dt 

%\left\{ 
%\begin{array}{ll}
%x(t -h _{k} )& \mbox{if } t-h_{k} \geq 0
%\\
 %d^{x}(t-h_{k}) & \mbox{if }t-h_{k} < 0
%\end{array}
%\right\}_{k=0}^{N},
%x_{N},u_{0},\ldots, u_{N});d_{0},\ldots,d_{N}),
%\end{eqnarray*},
%\left\{ \begin{array}{ll}
%u(t-h_{k}) & \mbox{if }t-h_{k}) \geq 0
%\\
% d^{u}(t-h_{k} & \mbox{if }t-h_{k} )< 0
%\end{array}
%\right\}_{k=0}^{N}
%)
\\
 \mbox{over  $(y(.),z(.)) \in W^{1,1} ([S-h,T];\R^{n+1})$ and measurable 
$u(.):[S-h,T] \rightarrow \R^{2m+n}$,} 
\\
\mbox{such that}
\\
%\hspace{4.0 in}
%\mbox{
%$u(.):[S,T]\rightarrow \R^{m}$ 
%and $d(.):[S-h_{N},S]$}
%\\ 
%\hspace{0.2 in} \mbox{such that } x(.)|_{[S,T]} \mbox{ is absolutely continuous},
 (\dot y(t),\dot z(t))= (\phi,M) ( t,\{ y(s-h_{k} )\}, \{v(s-h_{k})\})
 )\mbox{ a.e. }t\in [S-h,T],
%\hspace{0.2 in} (x(t),u(t))= d(t) \mbox{ a.e. }t\in[S,T],
\\
v(t)\in V(t)\mbox{ a.e. }t\in[S-h,T],
\\
((y(S-h),y(T))\in \R^{+}\times C,\, z(S-h)\geq 0\,.
\end{array}
\right.
$$
where
$
\bar y(t):= 
\bar x(S) \chi_{[S-h,S]}(t)+ \bar x(t) \chi_{[S,T]}(t),
%\left\{
%\begin{array}{ll}
%(\bar x(S))& \mbox{if } S-h\leq t \leq S\,,
%\\
%(\bar x(t))& \mbox{if } S < t \leq T
%\end{array}
%\right.
%\qquad
\bar z(t):= \int_{[S,S\wedge t]}M( ( s,\{ \bar x (s-h_{k} )\}, \{v(s-h_{k})\})
 ))ds
$ %
% \chi_{[S,T]}(t)

\noindent
$
\mbox{ and }\,
\bar v(t):=
(0,\bar d(t)) \chi_{[S-h,S]}(t)+ (\bar u, 0) \chi_{[S,T]}(t)\,.
%
%\left\{
%\begin{array}{ll}
%(0,\bar d(t))& \mbox{if } S-h\leq t \leq S\, .
%\\
%(\bar u, 0)& \mbox{if } S < t \leq T
%\end{array}
%\right.\;
$
\vspace{0.1 in}

\noindent
The underlying time interval is now $[S-h,T]$ and  $(\phi,M)(.\,.)$ and $V(.)$ are:
$$
\left\{
\begin{array}{l}
\phi(t, \{y_{k}\}, \{v_{k}=(u_{k},d_{k})\} )
\,:=\,  0\, \chi_{[S-h,S]}(t) \,+\,f(t, \{y_{k}\}, \{u_{k}\}; \{d_{k}\})\chi_{[S,T]}(t)
\\
M(t, \{y_{k}\}, \{v_{k}=(u_{k},d_{k})\} )
\,:=\,  \Lambda(t,d_{0})\, \chi_{[S-h,S]}(t) \,+\,L(t, \{y_{k}\}, \{u_{k}\};  \{d_{k}\})\chi_{[S,T]}(t)
\\
V(t) \,:=\, (\{0\} \times D(t)) \chi_{[S-h,S]}(t) +  ( U(t) \times \{0\}) \chi_{[S,T]}(t)\,.
\end{array}
\right.
$$
Notice that the cost integrands have been eliminated by `state augmentation' and  the `old' initial data $d(.)$ has been absorbed into the `new' control $v(.)$.  There is no need for `new' initial data because $((f,L)(t,.\,.))\chi_{[S,T]}(t)\equiv(0,0)$  for $t \in [S-h,S]$.) The data for the reformulated problem continues to satisfy (H1)-(H3), now with reference to $((\bar y(.), \bar z(.)), \bar v(.))$, but also the additional conditions.
%The assumptions are satisfied for the data for the reformulated problem (in addition to the hypotheses of Thm. \ref {main}). 
Application the special case of Thm. \ref{main}, in which (A0) and(A1) are  satisfied,
%, with reference to $(\bar y(.), \bar v(.))$ 
yields the conditions listed in Thm. \ref{main}, foro $L^{\infty}$ local minimizer $(\bar x(.), \bar u(.), \bar d(.))$, in the original problem.
\ \\

\noindent
Note that condition (b) in the theorem statement and the defining relations for the $p_{k}(.)$'s and $p(.)$ do indeed imply the `advance' differential inclusion  condition (b$^{*}$) for $p(.)$. 
%that, from condition (b) in the theorem statement,
This is because each $p_{k}(.)$ satisfies the relation $-\dot p_{k}(.) \in \mbox{co}\,\tilde \partial_{x_k} (p\cdot f -\lambda L)(.\,.)$, involving the projected limiting subdifferential $\tilde \partial_{x_k} (p \cdot f -\lambda L)$.  (b$^{*}$) now follows from the definition (\ref{p_eqn}) of $p(.)$ and condition  (\ref{p_comp}), which 
implies $-\dot p (t)= -\sum_{k=0}^{N} \dot{p}_{k}(t)$ a.e..  
\vspace{0.1 in}
\ \\

\noindent
Finally, we justify including the `pointwise' version of the Weierstrass condition (b*) in the assertions of Thm. \ref{main}. By consideration of the reformulated problem $(P_{r})$ of Step 5, we may restrict attention to the case when (A0) holds  and $L(.\,.)\equiv 0$. A simple analysis using `needle variations' permits us to deduce from the integral Weierstrass condition (b) that, for each $i$,
\begin{equation}
\label{special_W}
{\cal H}_{\lambda}(t,\bar u(t); \bar x(.), \bar u(.), p(.))= \underset{u \in U^{i}(t)}{\max}
{\cal H}_{\lambda}(t,u; \bar x(.), \bar u(.), p(.)) \mbox{ a.e. }t \in [S,T]\,. 
\end{equation}
Here $U_{i}(.)$ is the multifunction defined by (\ref{new_U_1}) and 
(\ref{new_U_2}). (${\cal H}_{\lambda}(t,u;.\,.)$ is the function (\ref{Hamiltonian}); we have suppressed reference to  $d(.)$ in the notation for ${\cal H}_{\lambda}(t,u;.\,.)$, since (H0) is assumed to be satisfied.) But we may use (\ref{approx}) to show that (\ref{special_W}) remains valid, when we substitute $U(t)$ for $U^{i}(t)$. We have confirmed the validity of the pointwise Weierstrass condition.

\section{Proof of Thm. \ref{thm3_1}}
To proof the PMP for free end-time problems, we initally assume that the following additional hypotheses are satisfied
%under the additional hypotheses, similar to some of these imposed in the initial steps of the proof the fixed time PMP, namely
\begin{itemize}
\item[(A0):]
$f(t,\{x_{k} \}, \{u_{k} \};\{d_{k}\})$ does not depend on the initial data $\{d_{k}\}$ and $\Lambda(.,.)\equiv 0$.
\item[]
(When (A0) is satisfied, we write $f(t,\{x_{k} \}, \{u_{k} \})$ in place of $f(t,\{x_{k} \}, \{u_{k} \};\{d_{k}\})$.)
%initial data integrand 
\item[(A1):] 
$L(.\,.) \equiv 0$. 
%There exists a bounded, ${\cal L}\times {\cal B}$ measurable function $\tilde L(.,.):[S,T]\times \R^{m}\to \R$ such that
%$ 
%L(t,\{x_{k}\},\{ u_k \})= \tilde L(t,u_{0})
%$.
%\item[(A2):] $D(s)$ and $U(t)$ are finite sets for all $s \in [S-h,S]$ and $t \in [S,T]$.
\item[(A2):] There exist integrable functions $c_{0}(.):[S,\bar T + \epsilon] \rightarrow \R$ and  $k_{0}(.):[S,T+ \epsilon]\rightarrow \R$ such that, for all selectors  $u(.)$ $U(.)$ and a.e. $t \in [S,\bar T + \epsilon]$, the mapping
$\{x_{k}\} \rightarrow f(t,\{x_{k}\},\{u(t-h_{k}) \})$
is $c_{0}(t)$ bounded  and $k_{0}(t)$-Lipschitz continuous on  $\R^{n \times (N+1)}$,
\vspace{0.05 in}
%(iii):
% $t \rightarrow \Lambda(t,d(t))$ is $c_{0}(t)$  bounded for a.e. $t \in [S-h, S]$.
 %\todo[inline]{I took $c_0$ defined in $[S-h,T]$ because it is present both in (i) and (iii)}
\end{itemize} 
The proof (under the additional hypotheses) proceeds in two steps. In the first step we consider the case when only the left endpoint of state trajectories is constrained. In the second, we show the  PMP for problems involving general endpoint constraints can be derived by applying it to a sequence of perturbed problems with free right endpoints (the case treated in Step 1), and passage to the limit.
\ \\

\noindent
 {\bf Step 1:} Let $(x^{*}(.),u^{*}(.),T^{*})$ be an $L^{\infty}$ be local minimizer the following problem which, when regarded as a special case of $(P_{FT})$, has  data satisfying (H1)-(H3), (A0) and (A2):
$$
(Q_{FT})
\left\{
\begin{array}{l}
\mbox{Minimize } \; g'(x(T),T)\,+\,\alpha \left( \int_{[S,T]} m(t,u(t))dt + |x(S)-x^{*}(S)| + |T-T^{*}|  \right) 
\\
\mbox{subject to }

%\mbox{over  $T \geq S$, $x(.) \in W^{1,1} ([S,T];\R^{n})$ and measurable functions}
%\\ 
%\\
%\hspace{01.0 in}\mbox{ 
%$u(.):[S-h,T] \rightarrow \R^{m} $,  $d(.)=(d^{x}(.),d^{u}(.)) :[S-h,S] \rightarrow \R^{n}\times \R^{m}$}
%\hspace{4.0 in}
%\mbox{
%$u(.):[S,T]\rightarrow \R^{m}$ 
%and $d(.):[S-h_{N},S]$}
%\\ 
%\hspace{0.2 in} \mbox{such that } x(.)|_{[S,T]} \mbox{ is absolutely continuous},
%\\
%\hspace{0.2 in}  \mbox{such that}
%\\
%\hspace{0.2 in}
\\
 \hspace{0.2 in} \dot x(t)= f(t,\{x(t-h_k)\},u(t)),\,  \mbox{ a.e. } t \in [S,T^{*}+ \epsilon ]
 \\
 \hspace{0.2 in} u(t) \in U(t)  \mbox{ a.e. } t \in [S,T^{*}+ \epsilon]\,,
\\
\hspace{0.2 in}x(S) \in C_{0},
\end{array}
\right.
$$
%($h_{1}$ is the smallest delay time.) Here, $g(.,.):\R \times \R^{n} \rightarrow \R$ is a Lipschitz continuous function) 
Here, $\alpha \geq 0$ is a given number, $m(.,.)$ is a given bounded, ${\cal L} \times {\cal B}$ measurable function and $g'(.,.)$ is a given Lipschitz continuous function. It is assumed that $T^{*}-S > h$.  
%such that
%$
%0 \leq m_{1}(.,.) \leq 1\,. 
%\mbox{ for all } u \in U(t),\, d \in D(s),  \mbox{ a.e. $t \in [S,T]$, $s \in [S-h,S]$ }\,.
%$
Our goal in this step is to prove the following necessary conditions: 
\ \\

\noindent
{\it There exist $p_{k}(.)\in W^{1,1}([S-h_{k},T^{*}];\R^{n})$, $k=0,\ldots, N$, satisfying (\ref{p_comp}) and $p(.)$ given by (\ref{p_eqn}) (when $T=T^{*}$), such that
\ \\

\noindent
(b)$\,'$: $\{-\dot p_{k}(t-h_{k})\}
%,\ldots,   -\dot p_{N}(t-h_{N}) ) 
\in \mbox{co}\, \partial_{x_{0},\ldots,x_{N}}\,p(t)\cdot f(t,\{x^*(t-h_k)\},u(t) ),
 \mbox{ a.e. } t \in [S,T^{*}]$,
\ \\
%$\hspace{0.2 in} \in 
%\mbox{co}\, \partial_{x_{0},\ldots,x_{N}}\,(p^{T}f -\lambda L)
 %(t,x_{0},\ldots, x_{N},\bar u(t-h_{0}),\ldots, \bar u(t-h_{N}); \bar d(t-h_{0}),\ldots, \bar d(t-h_{N}))$
 %
% $ \hspace{2.3 in}((x_{0},\ldots, x_{N})=  \bar x(t-h_{0}),\ldots, \bar x(t-h_{N}))$,
% 
%\noindent
% $\hspace{4.5 in} \mbox{a.e. } t \in [S,T].$
%\item[(c):] 

\noindent
(c)$\,'$:   
%for a.e. $t\in [S,T^{*}]$
%any selector $u(.)$ of $U(.)$,
$
(p\cdot f-\lambda m)(t,\{x^{*}_{i}(t-h_{k})\},u^{*}(t))\}) \,= \, \underset{u \in U(t)}{Max} (p\cdot f-\lambda m)(t,\{\bar x_{i}(t-h_{k})\},u)  \mbox{ a.e. }t\in [S,T^{*}],
$
%$
%{\cal H}(t,u^{*}(t); x^{*}(.), u^{*}(.), d^{*}(.), p(.))\geq \underset{u \in U(t)}{\max}
%{\cal H}(t,u; x^{*}(.), u^{*}(.), d^{*}(.), p(.))- \alpha, \mbox{ a.e. }t \in [S,T^{*}].
%$
\ \\

\noindent
(d)$\,'$: $(-p(T^{*}),\xi) \in  \partial  g'( x^{*}(T),T^*)+  2\alpha \B, \ \  p(S) \in \alpha \B +N_{C_{0}}(x^{*}(S))$, for some $\xi \in \R$ s.t.
$$
\xi \in \underset{t\rightarrow T^{*}}{\mbox{ess}}\;
\left\{
 \max_{u \in U(t)}\;  p(T^*)\cdot f(t,x^{*}(T^{*}-h_{0}),\ldots, 
  x^{*}(T^{*} -h_{N}), u)\right\}\,.
 % ; d^{*}(T^{*}-h_{0}),\ldots, d^{*}(T^{*}-h_{N}) )\right\}\,.
$$
}
%$({\cal H}(.\,.)$ and ${\cal H}^{0}(.\,.)$ are defined according to (\ref{Hamiltonian}) and  (\ref{Hamiltonian_0}), respectively, when the $-\lambda L(.\,.)$ term is removed.$)$ }
%\ \\
%
%\noindent
%\todo[inline]{I think in (e)' we need a $2\alpha\B$ because the cost has both the term $\alpha|T-T^*|$ and the term $\alpha\int_{[S,T]}m_1$. See later . .} 
We may assume, without loss of generality, that $g'(.,.)$ is continuously differentiable, not merely Lipschitz continuous. This is because, if $ g^{\prime}(.,.)$ were Lipschitz continuous, we could replace it by its $i$-quadratic inf convolution $ g^{\prime}_{i}(.,.)$, for  $i=1.2.\ldots$. For each $i$, $(x^{*}(.),u^{*}(.),T^{*})$ is a $\gamma_i$-minimizer of the perturbed problem, for some sequence $\gamma_i\downarrow0$. We may then apply Ekeland's theorem with the following metric on the space of processes:
%Equip ${\cal A}_{\epsilon}$ with the metric
\begin{eqnarray}
\label{eke_ft}
&&
 d_{{\cal E}}((x'(.),  u'(.),T') ,(x(.),  u(.),T) ) \;=\;
 |x'(S)- x(S)| \;+\; |T'-T|
\\
&& 
\nonumber
\hspace{2.0 in}
+\mbox{ meas } \{t \in [S,T\wedge T']\,|\, u'(t) \not= u(t) \}.
\end{eqnarray}
%\todo[inline]{I changed $\mbox{ meas } \{t \in [S,T\wedge T']\,|\, u'(t) \not= u(t) \}$. I think this is consistent with the metric chosen earlier: $\|x-x'\|_{L^\infty(S,T\wedge T')}+|T-T'|$ }
We thereby arrive at a minimizer $(x_{i}(.), u_{i}(.),T_i)$ for a perturbed version of $(Q_{FT})$, in which $\tilde g(.,.)$ is replaced by $ g^{\prime}_{i}(.,.)$. The element $(x_{i}(.), u_{i}(.),T_i)$ remains a minimizer when $g^{\prime}_{i}(.,.)$ is replaced by a quadratic function (plus a perturbation term) that majorizes $\tilde g_{i}(.,.)$ and coincides at 
$(T_{i}, x_{i}(T_{i}))$. We have arrived at in this way a problem again with the structure of $(Q_{FT})$, but in which $g(.,.)$ has been replaced by a continuously differentiable function. (It is precisely in anticipation of stage of the analysis that the cost in $(Q_{FT})$ is furnished with the `perturbation' term $\alpha\,(.\,.)$.) The special case of the Maximum Principle (with smooth terminal state and time cost) can be applied, with reference to $(x_{i}(.), u_{i}(.),T_i)$. We obtain the asserted necessary conditions (for Lipschitz continuous $\tilde g(.,.)$) in the limit as $i \rightarrow \infty$. (The details are very similar to those followed in Step  2 of the proof of Thm. \ref{main}.) 
\ \\

\noindent
So we assume $g^{\prime}(.,.)$ is continuously differentiable. For $T$ fixed at $T=T^{*}$, $(x^*(.), u^*(.))$ is an $L^{\infty}$ local minimizer for the corresponding fixed time problem.We then deduce from Thm. \ref{main} existence of functions $\{p_{k}(.)\}$ and $p(.)$ satisfying (\ref{p_comp}) and (\ref{p_eqn}), and conditions (b)$'$-(d)$'$. We also know that $p(S) \in \alpha \B +N_{C_{0}}(x^{*}(S))$. It remains to validate the transversality condition involving the optimal end-time. Take  $\delta \in (0,h_{1})$ such that $ T^*-\delta \geq S$. ($h_{1}$ is the shortest time delays period.)  Take also any $\gamma >0$ and let $v^{*}(.)$  be a measurable selector on $[T^{*},T^{*}+ \delta]$ of the multifunction
\begin{eqnarray*}
&&
t \rightarrow \{ u \in U(t) \,|\,   p(T^{*})\cdot f (t,x(T^*-h_{0}),\ldots,  x(T^*-h_{N}),u) \geq 
\\
&& \hspace{2.5 in} \underset{u'  \in U(t)}{\max}
p(T^{*})\cdot f(t,x(T^*-h_{0}),\ldots,  x(T^*-h_{N}),u')  -\gamma\}.
\end{eqnarray*}
Now consider an extension of $u^{*}(.):[S,T^*] \rightarrow \R^{m}$ to $[S,T^*+\delta]$, obtained by setting $u^{*}(t)= v^{*}(t)$ for $t \in (T^{*},T^{*}+ \delta]$. Extend also $x^{*}(.)$ to $[S,T^*+\delta]$, as the state trajectory corresponding to the extended control function $u^{*}(.)$ and the original initial state $x^{*}(S)$. Since $(x^{*}(.),u^{*}(.),T^{*})$ is an $L^{\infty}$ local minimizer, we have
\begin{equation}
\label{inequalities}
\tilde g( x^{*}(T^*- \delta),T^*-\delta)+2\alpha\delta \geq \tilde g( x^{*}(T^{*}),T^{*}) \mbox{ and } \  \tilde g( x^{*}(T^*+ \delta),T^*+\delta) +2\alpha\delta\geq \tilde g( x^{*}(T^*),T^*)
\end{equation}
%\todo[inline]{I added the quantities  $2\alpha\delta$.}
for all $\delta >0$ sufficiently small. Since $\tilde g(.,.)$ is continuously differentiable and $\nabla_{x}\, \tilde g(x^{*}(T^{*}),T^*)= -p(T^{*})$, we can deduce from the second inequality in (\ref{inequalities}) that
\begin{eqnarray*}
&& -2\alpha  \leq 
%\nabla_{T}g(T^{*}, x^{*}(T^{*})) - \delta^{-1}\left[ \int_{[T^{*},T^{*}+ \delta]} p^{T} (T^{*}) f(s, x^{*}(s-h_{0}),\ldots,
%x^{*}(s-h_{N}), u^{*}(s) ) ds\right] + o(|\delta|)
\\
&&
%&&
%\leq 
\nabla_{T} \tilde g( x^{*}(T^{*}),T^{*})- \delta^{-1}\left[ \int_{ [T^{*},T^{*}+ \delta]}  p( T^{*})\cdot  f(s, \{x^{*}(T^{*}-h_{k})\}, u^{*}(s))ds \right] + o(\delta)
\\
&&
\leq \nabla_{T} \tilde g( x^{*}(T^{*}),T^{*})
\\
&&
 - \underset{s \in [T^{*}-\delta, T^{*}+\delta ]}{\mbox{ess inf}} 
\left[ 
 %[T^{*},T^{*}+ \delta]}
 \underset{u \in U(s)}{\mbox{sup}}  p( T^{*})\cdot  f(s, \{x^{*}(T^{*}-h_{k})\}, u) 
\right] 
+  \gamma  +  o(\delta)\,.
\end{eqnarray*}
%\todo[inline]{Assuming that $k_f(.)$ is bounded around $T^*$ we have that $o(\delta)$ contains terms of the form:
%$$
%Const\x\sum_{k=0}^N\int_{T^*-h_k}^{T^*-h_k+\delta}c_f(t)dt
%$$
%Observe that such terms goes to zero as $\delta\downarrow0$ as desired but we do not use boundedness of $c_f(.)$. The interval $[T^*-\delta-h_k,T^*-h_k]$ may not be contained in $[T^*-\rho,T^*+\rho]$ where the function $c_f(t)$ is essentially bounded.
%}
(in which $o(.)$ is some increasing function such that $o(\delta') \rightarrow 0$ as $\delta' \downarrow 0$). (We make use of the fact that $T^{*} \geq S+h$ and the essential boundedness of $k_{f}(.)$ and $c_{f}(.)$ on a neighbourhood of $T^{*}$ to justify these relationships.)

%From the fixed time necessary conditions we know, however, that
%\begin{eqnarray*}
%&&
%p( T^{*})\cdot  f(s, x^{*}(T^{*}-h_{0}),
%\ldots,
%x^{*}(T^{*}-h_{N}), u^{*}(s))   
%\\
%&& \hspace{0.8 in}  \geq  \,\underset{u \in U(s)}{\mbox{sup}}  
%p( T^{*})\cdot  f(s, x^{*}(T^{*}-h_{0}),
%\ldots,
%x^{*}(T^{*}-h_{N}), u)ds 
%%\right] 
%- \alpha  - o_{1}(\delta)
%\end{eqnarray*}
%for $s \in [T^{*}-\delta,T^{*}]$. 

\ \\
\noindent
Exploiting the first inequality in (\ref{inequalities}), we arrive at
\begin{eqnarray*}
&& -2\alpha  \leq  
%\nabla_{T}g(T^{*}, x^{*}(T^{*})) - \delta^{-1}\left[ \int_{[T^{*},T^{*}+ \delta]} p^{T} (T^{*}) f(s, x^{*}(s-h_{0}),\ldots,
%x^{*}(s-h_{N}), u^{*}(s) ) ds\right] + o(|\delta|)
\\
&&
%&&
%\leq 
-\nabla_{T} \tilde g( x^{*}(T^{*}),T^*) + \delta^{-1}\left[ \int_{ [T^{*}-\delta,T^{*}]}  p( T^{*})\cdot  f(s,\{ x^{*}(T^{*}-h_{k})\}), u^{*}(s))ds \right] + o_{1}(\delta)
\\
&&
\leq -\nabla_{T} \tilde g( x^{*}(T^{*}),T^*)
\\
&&
 + \underset{s \in [T^{*}-\delta, T^{*} +\delta]}{\mbox{ess sup}} 
\left[ 
 %[T^{*},T^{*}+ \delta]}
 \underset{u \in U(s)}{\mbox{sup}}  p( T^{*})\cdot  f(s, \{x^{*}(T^{*}-h_{k})\}, u)
\right] 
+  o_{1}(\delta)\,,
\end{eqnarray*}
(for some increasing $o_{1}(.)$ such that $\lim_{s \downarrow 0} o_{1}(s)=0)$.
It follows 
%from the preceding relations that
\begin{multline*}
 \underset{s \in [T^{*}-\delta, T^{*}+\delta ]}{\mbox{ess inf}} 
\left[ 
 %[T^{*},T^{*}+ \delta]}
 \underset{u \in U(s)}{\mbox{sup}}  p( T^{*})\cdot  f(s, \{x^{*}(T^{*}-h_{k})\}, u)
\right] 
-  \gamma  -  o(\delta)-2\alpha
\\
\leq \, \nabla_{T} \tilde g( x^{*}(T^{*}),T^*) \, \leq \,
\\
 \underset{s \in [T^{*}-\delta, T^{*} +\delta]}{\mbox{ess sup}} \left[ \underset{u \in U(s)}{\mbox{sup}}  p( T^{*})\cdot  f(s,\{ x^{*}(T^{*}-h_{k})\}, u)
\right] 
+  o_{1}(\delta)+2\alpha\,.
\end{multline*}
Since $\gamma$ and $\delta$ are arbitrary positive numbers we conclude,
$$
\nabla_{T} \tilde g(T^{*}, x^{*}(T^{*}))\in \underset{t\rightarrow T^{*}}{\mbox{ess}}\;
\left\{
 \max_{u \in U(t)}\;  p(T^*)\cdot f(t,x^{*}(T^{*}-h_{0}),\ldots, 
  x^{*}(T^{*} -h_{N}), u)\right\}\;+\; 2\alpha\B\,.
$$
as required to complete the derivation of the necessary conditions for the free end-time problem of step 1.
\ \\

\noindent
{\bf Step 2:} Let $(\bar x(.),\bar u(.), \bar T)$ be a minimizer for  $(P_{FT})$. Assume that (H1)-(H3) and (A0)-(A2) are satisfied. 
We show that the assertions of Thm. \ref{thm3_1} are valid.
%, without the restrictions on the data of Step 1. 
(Now the endpoint cost function is $g( x(S),x(T),T)$ and the endpoint constraint is  $(x(S),x(T),T) \in C$.) 
%\ \\
%
%\noindent
%We can assume, without loss of generality, that the integral cost terms are absent. (The `state augmentation' technique of Step 4 in the proof of Thm. \ref{main} can be used to accommodate such terms if they are present.) 
%\noindent
\ \\

\noindent
Take $\gamma_{i} \downarrow 0$. For $i=1,2\ldots$, consider the problem with $(n+n)$-dimensional state vector $(z,x)$: 
$$
(P^{1}_{i})\left\{ 
\begin{array}{l}
\mbox{Minimize } 
\quad\max\{ g(z(T), x(T),T)- g(\bar z(\bar T),\bar x(\bar T),\bar T) + \gamma_{i}\,, d_{C}(z(T),x(T),T) \}
%J^{i}_{1}(y(.),x(.),u(.),d(.)) 
\mbox{ s.t. }
\\
%\mbox{subject to}
%\\
\\
%\left\{ 
%\begin{array}{l}
\dot z(t) = 0,\,
%\\
\dot x(t)= f(t,\{x(t-h_{k})\}, u(t))
 %\\
% \hspace{3.0 in} u(t-h_{N});d(t-h_{0}), \ldots, d(t-h_{N}),\,0)
%\end{array}
%\right.%\\
%||x(.)-\bar x(.) ||_{L^{\infty}} \leq \bar \epsilon
\\
\\
(z(S),x(S))\in \tilde C:= \{(z,x) \in \R^{n}\times \R^{n}\,|\, z=x\}\,.
\\
\\
||x(.)-\bar x(.)||_{L^{\infty}(S,T \wedge \bar T)} + |T-\bar T| \leq \bar \epsilon
\end{array}
\right.
$$
%in which
% 
%\begin{eqnarray*}
%&&{\cal A}^{1}_{\epsilon}\,:=\, \{ (x(.),y(.)),u(.),d(.)) \,|  \dot z(T)=0, (x(.),u(.),d(.)) \mbox{ is a  process for } (P),
%\\
%&& \hspace{3.0 in} (x(S),y(S)) \in C, \,  ||x(.)- \bar x(.)||_{L^{\infty}} \leq \epsilon
 %\, \}
%\end{eqnarray*}
%andt
%$$
%J^{i}_{1}((x(.),y(.)),u(.),d(.)) = 
%\int_{[S-h,S]}\Lambda(t,d(t))dt + 
%\max\{ g(x(S), y(S))- g(\bar z(T), \bar x(T)) + \gamma_{i}\,, d_{C}(x(T),z(T)) \}\,.
%$$
We see that, for $\bar \epsilon$ sufficiently small,
 $(\bar z(.)\equiv \bar x(S), \bar x(.),\bar u(.), \bar T)$ is an $\gamma_{i}$-minimizer for each $i$. From this point the analysis follows the same path as that in Step 4 of the proof of he necessary conditions for the fixed time problem. That is to say, we use Ekeland's Theorem to establish the existence of a new process $(z_i, x_{i}(.),u_{i}(.), T_{i})$, `close' to $(\bar z, \bar x(.),\bar u(.), \bar T)$ for large $i$, that is a minimizer for a perturbed problem. The perturbed problem has the special structure for which Step 1 provides necessary conditions. We apply the earlier derived necessary conditions, and obtain the assertions of the theorem in the limit as $i \rightarrow \infty$. The difference with the earlier `fixed time' analysis is that we now use the metric (\ref{eke_ft}) on free end-time processes, in place of the earlier metric on fixed time processes. 
 \ \\
 
 \noindent
So far, our proof of the Thm. \ref{thm3_1} covers only the special case  when the extra hypotheses (A0)-(A2). To show that the assertions of the Thm. remain valid when we remove (A0)-(A2) by techniques essentially the same as those employed in Step 5 of the proof of the fixed time PMP, based on the state augmentation and  domain extension.
% and inner approximation of the  multifunction $U(.)$ by a sequence $\{U_{i}(.)\}$.
%, as in Step  5 of the `fixed end-time' PMP proof. 
%Finally, to remove (A2), we consider a sequence of inner approximations $U_{i}(.)$ to the multifunction $U(.)$  
\section{Appendix} 
The purpose of this Appendix is to prove condition (c)$^ \prime$ in Lemma \ref{lemma_4_3}, namely the `integral Weierstrass condition' %associated with 
for problem 
$(\tilde P_{i})$. 
%Assume the stronger hypotheses: the bounds in (H3), write them $k_{0}(.)$ and $c_{0}(.)$, are independent of $i$ and  $c_{0}(.) \in L^{\infty}$ is an essentially bounded function.  We adapt the proof of Thm.  \ref{main}, given in Section \ref{Thm_2.1},  to obtain the  `integral' Weierstrass condition. Recall that Step 1 of the  proof involved the derivation of necessary conditions for the perturbed problem $(P_{i})$, 
%of Section \ref{Thm_2.1}, 
%summarized as Lemma \ref{main1}.   These necessary conditions included the approximate pointwise Weierstrass conditions (c)$^\prime$ and (d)$^\prime$, relating to variations of the control function and initial data function respectively.  The only significant respects in which the earlier proof requires modification is that the pointwise conditions (c)$^\prime$ and (d)$^\prime$ need to be replaced by integral conditions, namely:
Take an arbitrary selector $u(.)$ of $U(.)$. We must show
\ \\

\noindent
(c)$^{\prime \prime}$: $\int_{[S,T]}Q_{i}(t, \{ u(t-h_{k})\}) dt \leq \int_{[S,T]}Q_{i}(t, \{ u^{*}(t-h_{k})\}) dt$
\ \\

\noindent
where
$Q_{i}(t,\{u_{k}\}):= 
p(t)\cdot f(t, \{y^{*}_{k}(t)\}, \{u_{k}\}) -\tilde L((t,u_{0})- \gamma_{i}^{\frac{1}{2} } m_{i}(t, u_{0})\,.
$
(Note that the integrands are integrable in consequence of the supplementary hypothesis (A2).)
\ \\

\noindent
%A key role in the proof of this proposition is taken by t
%We now give the details of the proof of the first integral condition (c)$^{\prime \prime}$ alone; the proof of the second integral condition (d)$^{\prime \prime}$ is similar. 
We state for future use the following generalization of Hurwitz's theorem, concerning the simultaneous approximation of a finite collection of positive real numbers by rational numbers.  A proof is appears in \cite[Lemma 4.2]{warga3} of this classical result \cite[Thm. 200]{hardy}. ($\Z^{+}$ denotes the positive integers.)
\begin{lem} 
\label{lemma_approx}
Take positive numbers $h_{1},\ldots, h_{N}$. Then there exist a sequence $n_{j}\rightarrow \infty$ in $\Z^{+}$ and a sequence of   $\epsilon_{j} \downarrow 0$ in $\R$ such that
$$
\max_{k=1,\ldots,N}\left\{ \underset{m \in \Z^{+}}{\min}
\; |n_{j}h_{k}-m| \right\}\,\leq \, \epsilon_{j}, \mbox{ for  } 
% k=1, \ldots,N \mbox{and } 
j=1,2,\ldots 
$$
\end{lem}
\noindent
%Take a selector $u(.)$ of $U(.)$ and
Take any $K >0$ and define the set
$$
A^{K}:=\{t \in [S,T]\,|\,  |c_{0}(t)| \leq K   \}\,.
$$
Here $c_{0}(.) \in L^{1}$ is as in supplementary hypothesis (A2). Define
$
\tilde u(t)=\,
\left\{
\begin{array}{ll}
u(t)& \mbox{ if }t \in [S,T] \backslash A^{K},
\\
 u^{*}(t) & \mbox{ if } t \in A^{K}  \,.
\end{array}
\right.
$
\ \\

\noindent
 Take $\lambda \in (0,1)$. For $j=1,2, \ldots$ let $n_{j}$ and $\epsilon_{j}$ be as in Lemma \ref{lemma_approx} and
$$
 A_{j,\lambda}:= \bigcap_{k=0}^{n_{j}-1} \left[S+ \frac{k  (T-S)}{n_{j}}, S+\frac{(k +\lambda)(T-S)}{n_{j}}\right]\,.
$$
Now define the selector $u_{j,\lambda}(.):[S,T] \rightarrow \R^{m}$ be
$$
u_{j,\lambda}(.)\,=\,
\left\{
\begin{array}{ll}
\tilde u(t)& \mbox{ if $t \in A_{j,\lambda}$ },
\\
 u^{*}(t) & \mbox{ if } t \in [S,T] \,\backslash \, A_{j,\lambda}  \,.
\end{array}
\right.
$$
In view of Lemma \ref{lemma_approx}, 
\begin{eqnarray*}
&&\{ u_{j,\lambda}(t-h_{k})\}
 \,=\, (\{\tilde u(t-h_{k})\})\chi_{A_{j,\lambda}}(t) + (\{ u^{*}(t-h_{k})\}) )
\chi_{([0,T] \backslash A_{j,\lambda})}(t)
\end{eqnarray*}
for all $t \in [S,T] \, \backslash \, {\cal B}_{j, \lambda}$, where 
\vspace{0.1 in}

${\cal B}_{j, \lambda}\,:=\left(\left\{S, S+\frac{1}{n_{j}}(T-S), \ldots,S+ \frac{n_{j}}{n_{j}}(T-S) 
, \right. \right.$

$
\left.\left.
\, S+\frac{\lambda}{n_{j}}(T-S), S+ \frac{1+\lambda}{n_{j}}(T-S),\ldots, S+\frac{(n_{j}-1) +\lambda}{n_{j}} (T-S) \right\} \,+\, 2\epsilon_{j}(T-S)(-1,+1) \right)
\, \cap\, [S,T]\,.
$
\vspace{0.1 in}

\noindent Note that $\sigma_{j}:= \mbox{meas}\,\{{\cal B}_{j,\lambda}\} \leq 4 \epsilon_{j} \times \frac{2n_{i}+1}{n_{i}}|T-S| \rightarrow 0$, as $j \rightarrow 0$.
\ \\

\noindent
Now write $x_{j,\lambda}(.)$ for the solution to the dynamical equation in $(\tilde P_{i})$, when $u(.)= u_{j, \lambda}(.)$, $\{y_{k}(.)\}= \{y^{*}_{k}(.)\}$,with initial value $x_{j,\lambda}(S)= x^{*}(S)$. Since meas$\,\{A_{j,\lambda}\}= \lambda$ and since $c_{0}(.)$ is essentially bounded on $[S,T] \backslash A^{K}$, we have
$$
||x_{j,\lambda}(.) - x^{*} (.) ||_{L^{\infty}}\,\leq \,  \gamma(K) \,(\mbox{meas}\{A_{j, \lambda}\}+  \mbox{meas}\{{\cal B}_{j,\lambda}\})\, \leq\, \gamma(K) (\lambda + \sigma_{j})
$$
for some number $\gamma(K)$ that depends on $K$ but is independent of $j$. 
%Define 
%$$
%L_{i}^{\prime}(t,u_{0}, \ldots,u_{N}):= 
%L_{i}(t,u_{0})- \gamma_{i}^{\frac{1}{2} } m^{i}_{1}(t, u_{0}) -p(t)\cdot f(t, \{y^{*}_{k}\}, u_{0}, u_{N})\,.
%$$
According to (\ref{estimate}),
\begin{eqnarray*}
&&\int_{[S,T]} \left[ 
Q_{i}(t,\{ u_{j,\lambda}(t-h_{k})\}) -
Q_{i}(t, u^{*}\{(t-h_{k})\})
\right]dt \,\leq\,
\\
&& \hspace{2.3 in}  \theta(||x_{j, \lambda}(.)-x^{*}(.)||_{L^{\infty}})+||k_{0}(.)||_{L^{1}} ||x_{j, \lambda}(.)-x^{*}(.)||^{2}_{L^{\infty}}\,.
\end{eqnarray*}
Noting that $\tilde u(.)$ coincides with $u^{*}(.)$ on $A^{K}$, we see that
\begin{eqnarray*}
\hspace{-0.4 in}&&\int_{[S,T]} 
Q_{i}(t, \{u_{j,\lambda}(t-h_{k})\})dt \, \leq \, 
% -
%L_{i}^{\prime}(t, u^{*}(t-h_{0}\ldots, u^{*}(t-h_{N}))
%\right]
%\\
%\hspace{-0.3 in}&& 
%\hspace{0.2 in} 
\int_{[S,T]\backslash A^{K}} 
\left[ Q_{i}(t, \{u(t-h_{k}\}),\,
Q_{i}(t,  \{ u^{*}(t-h_{k}\})
\right] q_{j,\lambda}(t)dt \,+\, \gamma^{\prime}(K) \sigma_{j}
%\mbox{ meas}\, \{{\cal B}_{j,\lambda}\}
\,.
\end{eqnarray*}
for some number $\gamma^{\prime}(K)>0$ independent of $j$, in which
$$
q_{j,\lambda}(t):= 
\left[
\begin{array}{cc} 1 &0\end{array}
\right]^{T} 
\chi_{A_{j,\lambda}}(t)+  \left[\begin{array}{cc}  0 & 1\end{array}\right]^{T} \chi_{[0,T] \backslash A_{j,\lambda}}(t)\,.
$$
The sequence $\{q_{j,\lambda}(.)\}$ is equi-integrable and uniformly bounded and therefore has a weak limit in $L^{1}$. By consideration of $C^{1}$ test functions, we can easily show that
$q_{j, \lambda}(.) \,\rightarrow \, 
\left[\begin{array}{cc}  \lambda & 1-\lambda \end{array}\right]^{T}
$
%[\lambda,\,  ]^{T}
$\mbox{weakly$^{*}$ in $L^{\infty}$ as $ j \rightarrow \infty$}
$, and so
\begin{eqnarray*}
&&\int_{[S,T]\backslash A^{K}} 
\left[ Q_{i}(t, \{u(t-h_{k})\}),\,
 Q_{i}(t,  \{ u^{*}(t-h_{k})\})
\right] q_{j,\lambda}(t)dt \, 
\\
&&
\hspace{0.3 in}\rightarrow \;
\int_{[S,T]\backslash A^{K}}  \left[\lambda Q_{i}(t, \{u(t-h_{k})\})+ (1-\lambda)
Q_{i}(t,  \{u^{*}(t-h_{k})\})
\right]dt \, 
%+\, \bar k \mbox{ meas}\, \{{\cal B}_{j}\}
\,.
\end{eqnarray*}
as $j \rightarrow \infty$. We deduce from the preceding relations that
\begin{eqnarray*}
\hspace{-0.2 in}&&\lambda \,\int_{[S,T]\backslash A^{K}} 
\left[ Q_{i}(t, \{ u(t-h_{k})\}) -
 Q_{i}(t, \{ u^{*}(t-h_{k})\})
\right] dt \, 
\\
\hspace{-0.2 in}
&&
\hspace{0.3 in}\leq
\underset{j \rightarrow \infty}{\mbox{lim sup}}\left[ \bar k 
%\mbox{ meas}\, \{{\cal B}_{j,\lambda}\}
\sigma_{j}
+
\theta(\bar k (\lambda+ \sigma_{j})) - (\bar k (\lambda + \sigma_{j}))^{2}||k_{0}(.)||_{L^{1}}  
\right]\,= \theta(\bar k \lambda  ) - (\bar k \lambda)^{2}||k_{0}(.)||_{L^{1}} \,,
\end{eqnarray*}
which is valid for any $\lambda \in (0,1]$. Dividing across by $\lambda$ and passing to the limit as $\lambda \downarrow 0$ gives
$$
\int_{[S,T]\backslash A^{K}} 
\left[ Q_{i}(t, \{ u(t-h_{k})\}) -
 Q_{i}(t, \{ u^{*}(t-h_{k})\})
\right] dt \, \leq\,  0\,.
$$
But the integrand here is an integrable function. Since meas$\,\{ A^{K} \}\rightarrow 0$ as $K \rightarrow \infty$, the inequality is valid when $[S,T]$ replaces $[S,T] \backslash A^{K}$. We have confirmed condition (c)$^{ \prime} $.
%This is the Weierstrass condition in integral form. It has been derived under the additional hypotheses (A1)-(A5) (su[plemented by the requirement that $c_{0}(.)$ is essentially bounded). From this point the proof of Prop.  \ref{main1} follows closely that of Thm. \ref{main}, the only difference being that it is shown that the integrall form of the Weierstrass condition is preserved as we carry out the remaining steps (Step 2 - Step 5) of the proof to remove the additional hypotheses.
\section{Proof of Thm. \ref{thm3_1}}

{\bf Acknowledgement.} This work was funded by the European Union
under the 7th Framework Programme ``FP7-PEOPLE-2010-ITN'',
grant agreement number 264735-SADCO.

\end{document}